\documentclass[11pt]{amsart}

\usepackage[a4paper,top=20mm,bottom=30mm,left=30mm,right=30mm,includehead]{geometry}

\usepackage{xcolor}
\usepackage{hyperref}
\usepackage{amsmath, amssymb, amsthm, amsfonts}
\usepackage{upgreek}

\usepackage{listings}
\hypersetup{
    colorlinks=true,
    pdftitle={Theta Relations},
    pdfauthor={Rong Chen, Haigang Zhou, Xiao-Jie Zhu},
    pdfsubject={Mathematics},
}

\numberwithin{equation}{section}

\theoremstyle{plain}
\newtheorem{thm}{Theorem}[section]
\newtheorem*{mthm*}{Main Theorem}
\newtheorem*{thm*}{Theorem}
\newtheorem{coro}[thm]{Corollary}

\newtheorem{lemm}[thm]{Lemma}

\theoremstyle{definition}

\newtheorem{examp}[thm]{Example}

\theoremstyle{remark}
\newtheorem{rema}[thm]{Remark}

\newcommand\rmi{\mathrm i}

\newcommand\etp[1]{\mathfrak{e}\left(#1\right)}
\newcommand\tr{\mathop{\mathrm{tr}}}
\newcommand\numZ{\mathbb Z}
\newcommand\numQ{\mathbb Q}
\newcommand\numR{\mathbb R}
\newcommand\numC{\mathbb C}
\newcommand\F{\mathbb F}

\newcommand*\abs[1]{\lvert#1\rvert}

\newcommand\glnZ[1]{\mathrm{GL}_{#1}(\mathbb{Z})}

\newcommand\tbtmat[4]{\left(\begin{smallmatrix}{#1} & {#2} \\ {#3} & {#4}\end{smallmatrix}\right)}
\newcommand\tbtMat[4]{\left(\begin{matrix}{#1} & {#2} \\ {#3} & {#4}\end{matrix}\right)}

\makeatletter
\@namedef{subjclassname@2020}{\textup{2020} Mathematics Subject Classification}
\makeatother

\begin{document}

\title{Theta Relations}

\author{Rong Chen}
\address{Department of Mathematics, Shanghai Normal University}
\email{rchen@shnu.edu.cn}

\author{Haigang Zhou}
\address{School of Mathematical Sciences, Tongji University}
\email{haigangz@tongji.edu.cn}

\author{Xiao-Jie Zhu}
\address{School of Mathematical Sciences,
Key Laboratory of MEA (Ministry of Education) \& Shanghai Key Laboratory of PMMP,
East China Normal University}
\email[Corresponding author]{zhuxiaojiemath@outlook.com}



\begin{abstract}
Motivated by the classical quartic and cubic theta identities, we establish a product formula for Siegel--Jacobi theta series associated with rationally isometric positive-definite quadratic spaces. Our main theorem expresses a product of theta series of arbitrary Siegel degree as a finite double character sum over quotients of commensurable lattices. For the root lattice \(A_2\), it yields a finite Fourier identity for cubes of cubic theta functions, a new proof of the two-variable cubic identity, and seven independent relations in degree two. A second specialization gives a new square analogue of the two-variable cubic identity. At the origin it reduces to \(a(q)a(q^2)=b(q)b(q^2)+c(q)c(q^2)\); in degree two, the same seven relation patterns recur with each cube replaced by the corresponding product at \(\tau\) and \(2\tau\). We further extend the cubic identities on \(A_2\) to \(\ell\)-th-power identities on \(A_{\ell-1}\) for every \(\ell\geq3\), and to \(2\ell\)-th-power identities on \(A_{\ell-1}\) for every \(\ell\geq2\), in arbitrary Siegel degree. The doubled-order construction recovers Jacobi's quartic identity and gives its analogues in arbitrary Siegel degree.
\end{abstract}

\maketitle

\tableofcontents

\section{Introduction}
\label{sec:Introduction}
For \(\abs q<1\), define
\[
 \vartheta_2(q)=\sum_{k\in\numZ}q^{(k+1/2)^2},\qquad
 \vartheta_3(q)=\sum_{k\in\numZ}q^{k^2},\qquad
 \vartheta_4(q)=\sum_{k\in\numZ}(-1)^kq^{k^2}.
\]
Jacobi~\cite[p.~90, Eq.~(14)]{Jac12} proved the fundamental quartic
identity
\begin{equation*}
 \vartheta_3(q)^4=\vartheta_2(q)^4+\vartheta_4(q)^4;
\end{equation*}

J.~M. Borwein and P.~B. Borwein~\cite{BB91} discovered a cubic
counterpart naturally associated with the root lattice
\(A_2\).  Let
\[
 Q(m,n)=m^2+mn+n^2,\qquad
 \omega=\exp(2\uppi\rmi/3),
\]
and define
\begin{equation*}
 \begin{aligned}
 a(q)&=\sum_{m,n\in\numZ}q^{Q(m,n)},\\
 b(q)&=\sum_{m,n\in\numZ}\omega^{m-n}q^{Q(m,n)},\\
 c(q)&=\sum_{m,n\in\numZ}q^{Q(m+1/3,n+1/3)}.
 \end{aligned}
\end{equation*}
The Borweins~\cite[Eq.~(2.3)]{BB91} proved that
\begin{equation}
\label{eq:intro Borwein identity}
 a(q)^3=b(q)^3+c(q)^3,
\end{equation}
so the fourth-power relation for the classical one-dimensional theta
functions is replaced by a third-power relation attached to \(A_2\).
This identity is fundamental in the theory of elliptic functions to
the cubic base; see \cite{BB91,BBG94,HGB93}.

Schultz~\cite{Sch13} introduced nine cubic theta functions in two complex
variables, four of which play the role of the four classical theta
functions.  Following Schultz~\cite[pp.~621--622]{Sch13}, let
\(\etp{v}:=\exp(2\uppi\rmi v)\), \(\tau\in\mathfrak H_1\) (the upper half-plane),
\(q=\etp{\tau}\), \(\boldsymbol u=(u_1,u_2)^T\in\numC^2\),
\(x=\etp{u_1}\), and \(y=\etp{u_2}\), and define
\begin{equation}
\label{eq:intro Schultz functions}
\begin{aligned}
 \Theta_1(\boldsymbol u;q)
 &=\sum_{m,n\in\numZ}
   \omega^{1+m-n}q^{Q(m+1/3,n+1/3)}
   x^{m+1/3}y^{n+1/3},\\
 \Theta_2(\boldsymbol u;q)
 &=\sum_{m,n\in\numZ}
   q^{Q(m+1/3,n+1/3)}
   x^{m+1/3}y^{n+1/3},\\
 \Theta_3(\boldsymbol u;q)
 &=\sum_{m,n\in\numZ}q^{Q(m,n)}x^my^n,\\
 \Theta_4(\boldsymbol u;q)
 &=\sum_{m,n\in\numZ}
   \omega^{m-n}q^{Q(m,n)}x^my^n.
\end{aligned}
\end{equation}
The fractional powers are interpreted through the additive variables
\(\tau,u_1,u_2\).  These four functions are the cubic analogues of
the four classical Jacobi theta functions in the sense developed by
Schultz~\cite{Sch13}.  Schultz~\cite[Corollary~4.3]{Sch13} proved that
\begin{equation}
\label{eq:intro Schultz identity}
 \Theta_1(\boldsymbol u;q)^3+\Theta_3(\boldsymbol u;q)^3
 =\Theta_2(\boldsymbol u;q)^3+\Theta_4(\boldsymbol u;q)^3.
\end{equation}
The definitions give
\[
 \Theta_1(\boldsymbol0;q)=0, \qquad
 \Theta_2(\boldsymbol0;q)=c(q),\qquad
 \Theta_3(\boldsymbol0;q)=a(q),\qquad
 \Theta_4(\boldsymbol0;q)=b(q).
\]
Thus \eqref{eq:intro Schultz identity} specializes to
\eqref{eq:intro Borwein identity}.  The identity
\eqref{eq:intro Schultz identity} was subsequently revisited by Chan,
Chan, Liu, and Zudilin~\cite{CCLZ26}, who gave two further proofs and
several related identities.  Zhu~\cite[Theorem~7.12]{Zhu23} gave a
computational proof of this identity.  Zhu~\cite[Theorem~7.11]{Zhu23}
also proved an analogous identity for the \(D_4\) lattice.

Our main result is a general theta-series identity whose matrix
specializations yield all the identities above, among others.

To describe the scope, let \(M\) be a symmetric positive-definite
rational matrix of size \(m\), let \(n\geq1\), and let
\(a,b\in\numR^{m,n}\).  We consider the Siegel--Jacobi theta series
\[
 \Theta_{M,n}[a,b](\tau,z)
 =\sum_{g\in\numZ^{m,n}}
 \etp{\tr\left(
 \frac12\tau(g+a)^TM(g+a)+(g+a)^TM(z+b)\right)},
\]
for \((\tau,z)\in\mathfrak H_n\times\numC^{m,n}\).  For the notation,
see Subsection~\ref{subsec:notation}.  Ziegler~\cite{Zie89} studied
theta series of this type from the viewpoint of Jacobi forms of
higher degree; see Subsection~\ref{sec:Jacobi Theta series} for the
basic properties used here.

The main result is Theorem~\ref{thm:main}; we first state a simplified
version.
Let \(n,m,\mu\) be positive integers, let \(M_1,M_2\) be symmetric
positive-definite rational matrices of size \(m\), and let
\(\tau\in\mathfrak H_n\) and
\(a_i,b_i\in\numR^{m,n}\), \(z_i\in\numC^{m,n}\), for
\(i=1,2,\ldots,\mu\).  Suppose that
\(A\in\mathrm{GL}_{m\mu}(\numQ)\) satisfies
\(A^TM_1^{(\mu)}A=M_2^{(\mu)}\).
Define
\begin{align*}
(a_1'^{T},\cdots,a_\mu'^T)&=(a_1^{T},\cdots,a_\mu^T)A^{-T},\\
(b_1'^{T},\cdots,b_\mu'^T)&=(b_1^{T},\cdots,b_\mu^T)A^{-T},\\
(w_1^{T},\cdots,w_\mu^T)&=(z_1^{T},\cdots,z_\mu^T)A^{-T}.
\end{align*}
Then
\begin{equation}
\label{eq:main_intro}
\prod_{i=1}^{\mu}\Theta_{M_1,n}[a_i,b_i](\tau,z_i)=
\sum_{\alpha,\beta}c_{\alpha,\beta}\prod_{i=1}^{\mu}\Theta_{M_2,n}[a_i'+\alpha_i,b_i'+\beta_i](\tau,w_i),
\end{equation}
where $(\alpha_1^{T},\cdots,\alpha_\mu^T):=\alpha^T$ and $(\beta_1^{T},\cdots,\beta_\mu^T):=\beta^T$.
Here $\alpha$ and $\beta$ range over certain finite subsets of
$\numQ^{m\mu,n}$, and the coefficients $c_{\alpha,\beta}$ are
constant.
See Theorem~\ref{thm:main} for the precise statement.

The proof relies on a lemma of Igusa~\cite[p.~136, Lemma~1]{Igu72}.
Igusa~\cite[pp.~137 and~139]{Igu72} used this lemma to derive two
classical theta relations of Riemann.  These relations are also special cases
of our main theorem, obtained by specializing \(M_1=(1)\); see
Theorems~\ref{thm:bilinear-relation} and~\ref{thm:quadrilinear}.  In
addition to these classical formulas, theta relations have been studied
extensively; see, for example, \cite{CGG15,Fre91,GW25,Igu79}.
Gritsenko and Wang~\cite{GW25}, in particular, show that certain finite
sums of products of theta series are holomorphic Jacobi forms and then
derive theta relations from the theory of Jacobi forms.  By contrast,
the generic theta relation \eqref{eq:main_intro} can be proved without
using the modular
and elliptic transformation laws of
\(\Theta_{M,n}[a,b](\tau,z)\).  The proof requires neither \(q\)-series
techniques nor complex analysis beyond the absolute convergence needed
to rearrange the defining series.  Thus our approach is essentially
algebraic once absolute convergence is known.  Consequently, the proof
of \eqref{eq:main_intro} is uniform for all \(n\) and all
\(M_1,M_2,A\) satisfying \(A^TM_1^{(\mu)}A=M_2^{(\mu)}\), and we
obtain applications in arbitrary Siegel degree (genus) \(n\), thereby
extending identities previously known only in degree one.

\subsection{Applications}

All families of theta identities below arise by specializing the main
formula \eqref{eq:main_intro}.  The procedure is as follows:
\begin{enumerate}
  \item Choose an appropriate $A$.
  \item Specify suitable representative sets for $\alpha$ and $\beta$.
  \item Identify theta series on the right-hand side that are proportional to each other.
  \item Combine the relations linearly into the desired relation.
\end{enumerate}

\emph{The first application} concerns the \(A_2\) lattice.  Write
\(\F_3=\numZ/3\numZ\).  For
\(\mathbf p,\mathbf q\in\F_3^n\), we define
\begin{equation*}
 \Theta^{(n)}_{\mathbf p,\mathbf q}(\tau,z)
 :=\Theta_{M,n}[\rho\mathbf p,\sigma\mathbf q](\tau,z),
 \qquad
 \theta^{(n)}_{\mathbf p,\mathbf q}(\tau)
 :=\Theta^{(n)}_{\mathbf p,\mathbf q}(\tau,0),
\end{equation*}
where $M=\tbtmat{2}{-1}{-1}{2}$ is the Gram matrix of the $A_2$ lattice,
and where \(\rho=(1/3,-1/3)^T\) and
\(\sigma=(1/3,1/3)^T\).  When \(n=1\), these nine functions
correspond, after a linear change of Jacobi variables and a phase
normalization, to those studied by Schultz~\cite{Sch13}; see
Example~\ref{examp:cubic}.  Let \(C=\tbtmat{1}{1}{-1}{2}\).
Our first identity is
\begin{equation}
\label{eq:intro_app1}
 3^n\Theta^{(n)}_{\mathbf p,\mathbf q}(\tau,z)^3
 =\sum_{\mathbf u,\mathbf v\in\F_3^n}
  \omega^{\mathbf p\cdot\mathbf v+\mathbf q\cdot\mathbf u}
  \Theta^{(n)}_{\mathbf u,\mathbf v}(\tau,Cz)
  \theta^{(n)}_{\mathbf u,\mathbf v}(\tau)^2.
\end{equation}
This is proved in Theorem~\ref{r: cubic Fourier theorem}.  For $n=1$,
it immediately implies Schultz's identity
\eqref{eq:intro Schultz identity}; see Theorem~\ref{r: Schultz theorem}.

In degree two, we classify the theta constants
$\theta^{(2)}_{\mathbf p,\mathbf q}(\tau)$ into thirteen inequivalent
nonzero classes.  Using \eqref{eq:intro_app1}, we prove that
these theta series $\theta^{(2)}_{\mathbf p,\mathbf q}(\tau)$ satisfy seven independent
cubic relations.
See Theorem~\ref{r: degree two cubic theorem} for the precise statement
and proof.

To write one of the above cubic relations in expanded form, let
\[
 \tau=\begin{pmatrix}\tau_1&\tau_0\\\tau_0&\tau_2\end{pmatrix}
 \in\mathfrak H_2,
 \qquad q=\etp{\tau_1},\quad s=\etp{\tau_0},\quad
 t=\etp{\tau_2}.
\]
Then
\begin{align}
\label{eq:intro degree two cubic example}
&\left(
 \sum_{j,k,\ell,h\in\numZ}
 q^{Q(j,k)}
 s^{2j\ell+k\ell+jh+2kh}t^{Q(\ell,h)}
 \right)^3\\
={}&
 \left(
 \sum_{j,k,\ell,h\in\numZ}
 \omega^{j-k+\ell-h}q^{Q(j,k)}
 s^{2j\ell+k\ell+jh+2kh}t^{Q(\ell,h)}
 \right)^3 \notag\\
&+
 \left(
 \sum_{j,k,\ell,h\in\numZ+1/3}
 q^{Q(j,k)}s^{2j\ell+k\ell+jh+2kh}t^{Q(\ell,h)}
 \right)^3 \notag\\
&+
 \left(
 \sum_{\substack{j,k\in\numZ\\
                  \ell,h\in\numZ+1/3}}
 \omega^{j-k}q^{Q(j,k)}
 s^{2j\ell+k\ell+jh+2kh}t^{Q(\ell,h)}
 \right)^3 \notag\\
&+
 \left(
 \sum_{\substack{j,k\in\numZ+1/3\\
                  \ell,h\in\numZ}}
 \omega^{\ell-h}q^{Q(j,k)}
 s^{2j\ell+k\ell+jh+2kh}t^{Q(\ell,h)}
 \right)^3. \notag
\end{align}
All powers in this formula are understood through the additive
variables \(\tau_1,\tau_0,\tau_2\).  In the formal limit
\(t\to0\), the left-hand side tends to \(a(q)^3\), the first term on
the right tends to \(b(q)^3\), the next two terms vanish, and the last
term tends to \(c(q)^3\).  Thus this degree-two identity specializes
exactly to \eqref{eq:intro Borwein identity}.

\emph{The second application} is the following square counterpart of
\eqref{eq:intro_app1}:
\begin{equation}
\label{eq:intro_app2}
3^n\Theta^{(n)}_{\mathbf p,\mathbf q}(\tau,z)
\Theta^{(n)}_{\mathbf p,-\mathbf q}(2\tau,2z)
=\sum_{\mathbf u,\mathbf v\in\F_3^n}
\omega^{\mathbf p\cdot\mathbf v+\mathbf q\cdot\mathbf u}
\Theta^{(n)}_{\mathbf u,\mathbf v}(\tau,Cz)
\theta^{(n)}_{\mathbf u,\mathbf v}(2\tau).
\end{equation}
See Theorem~\ref{thm:square-counterpart} for the precise statement and
proof.  For \(\boldsymbol u=(u_1,u_2)^T\), define
\[
 \boldsymbol u^*=\begin{pmatrix}u_1-u_2\\-u_2\end{pmatrix}.
\]
An important consequence of \eqref{eq:intro_app2} is the new identity
\begin{equation}
\label{eq:intro Schultz square identity}
\begin{aligned}
 &\Theta_1(\boldsymbol u;q)\Theta_1(2\boldsymbol u^*;q^2)
 +\Theta_3(\boldsymbol u;q)\Theta_3(2\boldsymbol u^*;q^2)\\
 &\qquad=
 \Theta_2(\boldsymbol u;q)\Theta_2(2\boldsymbol u^*;q^2)
 +\Theta_4(\boldsymbol u;q)\Theta_4(2\boldsymbol u^*;q^2),
\end{aligned}
\end{equation}
where we adopt Schultz's notation \eqref{eq:intro Schultz functions}.
Theorem~\ref{r: Schultz square theorem} proves a more general form of
this identity.
At \(\boldsymbol u=\boldsymbol0\), it reduces to
\begin{equation*}
 a(q)a(q^2)=b(q)b(q^2)+c(q)c(q^2).
\end{equation*}
Borwein, Borwein, and Garvan~\cite[Theorem~2.6(i)]{BBG94} proved this
identity.  To the best of our knowledge, the
identity \eqref{eq:intro Schultz square identity} with two independent Jacobi
variables is new.

In degree two, the same seven relation patterns recur for the square
counterparts, with each cube \(\theta(\tau)^3\) replaced by
\(\theta(\tau)\theta(2\tau)\); see
Theorem~\ref{r: degree two square theorem}.

\emph{The third application} concerns root lattices of arbitrary rank.
Let $\ell\geq2$, $d=\ell-1$, and $\xi=\etp{1/\ell}$.  Suppose that
$\boldsymbol\lambda,\boldsymbol\kappa\in(\numZ/\ell\numZ)^n$ satisfy
$\boldsymbol\lambda\cdot\boldsymbol\kappa\not\equiv0\pmod\ell$.
Let $M_\ell$ be the Gram matrix of the $A_d$ lattice; see
\eqref{r: p data}, and let
$\boldsymbol r=\begin{pmatrix}1&2&\cdots&d\end{pmatrix}^{T}$.  Then

\begin{equation}
\label{eq:intro_app3_2l}
 \sum_{\substack{\boldsymbol a\in(\numZ/\ell\numZ)^n,
                   \,\boldsymbol c\in\{0,1\}^n\\
                   \boldsymbol b_1,\ldots,\boldsymbol b_{d-1}
                   \in(\numZ/2\ell\numZ)^n}}
 \xi^{\boldsymbol\lambda\cdot\boldsymbol a+
       \boldsymbol\kappa\cdot
       (\boldsymbol b_1+\cdots+\boldsymbol b_{d-1}+\boldsymbol c)}
 \Theta_{M_\ell,n}^{\,2\ell}
 \!\left[\frac{\boldsymbol r\boldsymbol a}{\ell},
          \frac{B'}{2\ell}\right](\tau,z)=0,
\end{equation}
where \(B'=(\boldsymbol b_1^T,\ldots,\boldsymbol b_{d-1}^T,\boldsymbol c^T)^T\).
If, moreover, \(\ell\geq3\) and every
component of \(\boldsymbol\kappa\) lies in
\(2\numZ/\ell\numZ\) when $\ell$ is even, then
\begin{equation}
\label{eq:intro_app3_l}
 \sum_{\substack{\boldsymbol a,\boldsymbol b_1,\ldots,
                  \boldsymbol b_{d-1}
                  \in(\numZ/\ell\numZ)^n}}
 \xi^{\boldsymbol\lambda\cdot\boldsymbol a+
       \boldsymbol\kappa\cdot
       (\boldsymbol b_1+\cdots+\boldsymbol b_{d-1})}
 \Theta_{M_\ell,n}^{\,\ell}
 \!\left[\frac{\boldsymbol r\boldsymbol a}{\ell},
          \frac{B}{\ell}\right](\tau,z)=0,
\end{equation}
where \(B=(\boldsymbol b_1^T,\ldots,\boldsymbol b_{d-1}^T,\boldsymbol 0^T)^T\).
The two identities are proved in
Theorems~\ref{r: doubled order power identity} and
\ref{r: unified p power identity}, respectively.

For every \(\ell\geq3\), the formula \eqref{eq:intro_app3_l} extends
the third-power relations on \(A_2\) to \(\ell\)-th-power relations on
\(A_{\ell-1}\), and \eqref{eq:intro_app3_2l} produces
\(2\ell\)-th-power identities valid also for \(\ell=2\).  The case
\(\ell=2\) recovers
Jacobi's quartic identity~\cite[p.~90, Eq.~(14)]{Jac12} in degree one
and gives analogues involving a Jacobi variable and valid in arbitrary
Siegel degree.  In degree two it yields five independent relations among the
fourth powers of the ten even theta constants; comparison with Cl\'ery,
van der Geer, and Grushevsky~\cite[Section~6, Eq.~(6.1)]{CGG15} shows
that they form a basis of all
linear relations among these fourth powers.  Moreover, we classify the
degree-one theta constants associated with \(A_3\) and \(A_4\) and
provide explicit fourth- and fifth-power
identities, which are specializations of \eqref{eq:intro_app3_l}.

The paper is organized as follows.  Section~2 collects the required
facts about finite abelian groups and commensurable lattices, and
introduces the Siegel--Jacobi theta series used throughout.
The final subsection of Section~2 states and proves the main product
formula.
Sections~3 and~4 treat the cubic identities and their square
counterparts, respectively.  The same-order identities on root
lattices are developed in Section~5, and the doubled-order identities,
including the quartic specializations, are considered in Section~6.

\subsection{Notation}
\label{subsec:notation}

Let \(R\) be a commutative ring, and let \(n,m\) be positive integers.
The symbol \(R^{n,m}\) denotes the additive group of \(n\times m\)
matrices over \(R\), \(\mathrm{GL}_n(R)\) the group
of invertible matrices of size \(n\), and, when \(R=\numQ\) or
\(R=\numR\), \(\mathrm{GL}^{\mathrm{pos}}_n(R)\) the set of
symmetric positive-definite matrices of size \(n\).  The symbol
\(\mathrm{M}^{\mathrm{sym}}_{n,n}(R)\) denotes the set of
\(n\times n\) symmetric matrices.  For a matrix \(A\), \(A^T\)
denotes its transpose, and \(A^{-T}\) denotes its inverse transpose
(if it exists). 


For $A=(a_{ij})\in R^{n,m}$ and $B\in R^{s,t}$, the Kronecker product
$A\otimes B$ is the $ns \times mt$ block matrix
$(a_{ij}\cdot B)_{1\leq i \leq n,1\leq j \leq m}$.  For a positive
integer $\mu$, define $B^{(\mu)}:=I_{\mu}\otimes B$, where $I_\mu$ is
the identity matrix of size $\mu$.  For matrices
$M_1,M_2,\dots,M_k$ of the same size, write
\[
[M_1,M_2,\dots,M_k]:=(M_1^T,M_2^T,\dots,M_k^T)^T.
\]
The symbol \(e_i\) denotes the \(i\)-th standard basis vector in the
ambient dimension determined by context.

The Siegel upper half-space of degree $n$ is
\begin{equation*}
\mathfrak{H}_n:=\{u+\rmi v\colon u\in \mathrm{M}^{\mathrm{sym}}_{n,n}(\numR), v \in \mathrm{GL}^{\mathrm{pos}}_n(\numR)\}.
\end{equation*}

We use the notation $\etp{x}:=\exp(2\uppi\rmi x)$ and
$q^r:=\exp(2\uppi\rmi r\tau)$ for $\tau\in\mathfrak{H}_1$ and
$r\in\numR$.  For
$\tau=\tbtmat{\tau_{1}}{\tau_{0}}{\tau_{0}}{\tau_{2}}\in\mathfrak{H}_2$,
we use $q^r:=\exp(2\uppi\rmi r\tau_1)$,
$s^r:=\exp(2\uppi\rmi r\tau_0)$, and
$t^r:=\exp(2\uppi\rmi r\tau_2)$.  The condition
$\tau\in\mathfrak{H}_2$ is equivalent to $\Im(\tau_1)>0$,
$\Im(\tau_2)>0$, and
$\Im(\tau_1)\Im(\tau_2)-\Im(\tau_0)^2>0$.

\section{The generic formula}

\subsection{Prerequisites: groups and lattices}
\label{sec:Prerequisites: groups and lattices}
Let $G$ be a finite abelian group.  Its group of linear characters is
denoted by $\widehat{G}$ and consists of all homomorphisms from $G$ to
$\numC^\times$, the multiplicative group of nonzero complex numbers.
It is well known that $G \cong \widehat{G}$ noncanonically; each cyclic
decomposition of $G$ determines an isomorphism.  Theorem~\ref{thm:main}
relies on the following lemma of Igusa.
\begin{lemm}
\label{lemm:L1sum}
Let $L$ be an abelian group written additively, and let $L_1$ and $L_2$
be finite-index subgroups.  Suppose that
\begin{equation*}
\Phi\colon L \to \numC
\end{equation*}
is an $L^1$-function; that is,
$\sum_{g \in L}\vert\Phi(g)\vert<+\infty$.  Then
\begin{equation*}
[L\colon L_1]\sum_{g\in L_1}\Phi(g)=\sum_{\chi\in\widehat{L/L_1},g\in L/L_2}\left(\sum_{h\in L_2}\chi(g+h)\Phi(g+h)\right).
\end{equation*}
\end{lemm}
Here $[L\colon L_1]$ denotes the index of $L_1$ in $L$, also written
$\vert{L/L_1}\vert$.  Igusa~\cite[p.~136, Lemma~1]{Igu72} proved this
lemma.

We shall use this lemma when \(L\) is a lattice---that is, a free
abelian group of finite rank.  In this paper, however, the terminology
``lattice'' will always refer to a \(\numZ\)-module \(L\) in a
background finite-dimensional vector space \(V\) over \(\numR\) or
\(\numQ\) such that \(L\) has a \(\numZ\)-basis which is simultaneously
a basis of \(V\).


Fix a space $V$ as above, and let $L_1$ and $L_2$ be lattices in $V$.
They are \emph{commensurable} if both $[L_1\colon L_1\cap L_2]$ and
$[L_2\colon L_1\cap L_2]$ are finite.  Commensurability is an
equivalence relation, and the sum and intersection of two commensurable
lattices are again lattices in $V$.  In particular, every finite-index
submodule of a lattice is a lattice.  For positive integers $m,n$, the
module $\numZ^{m,n}$ of $m\times n$ integral matrices is a lattice in
$\numR^{m,n}$.  The lattices in $\numR^{m,n}$ commensurable with
$\numZ^{m,n}$ are precisely those contained in $\numQ^{m,n}$; hence any
two lattices in $\numQ^{m,n}$ are commensurable.  Consequently, for any
nonsingular rational $m\times m$ matrix $A$, there is another such
matrix $\widetilde{A}$ satisfying
$\numZ^{m,n}+A\cdot\numZ^{m,n}=\widetilde{A}^{-1}\cdot\numZ^{m,n}$.

Applying Theorem~\ref{thm:main} requires an explicit choice of
$\widetilde{A}$ for a given $A$.


\begin{lemm}
\label{lemm:findA1}
Let $m,n$ be positive integers and let $A\in\mathrm{GL}_m(\numQ)$.
Then:
\begin{enumerate}
  \item[(a)] There exist $B_1,B_2\in\mathrm{GL}_m(\numZ)$ such that $B_1^{-1}AB_2$ is a diagonal matrix.
  \item[(b)] If $B_1^{-1}AB_2=\mathrm{diag}\left(\frac{q_1}{p_1},\ldots,\frac{q_m}{p_m}\right)$, where $p_i,q_i$ are positive integers with $\gcd(q_i,p_i)=1$, then $\numZ^{m,n}+A\cdot\numZ^{m,n}=\widetilde{A}^{-1}\cdot\numZ^{m,n}$ with $\widetilde{A}=\mathrm{diag}\left(p_1,\ldots,p_m\right)\cdot B_1^{-1}\in\numZ^{m,m}$. Moreover, $\abs{\det(\widetilde{A})}=\prod_{i=1}^mp_i$.
\end{enumerate}
\end{lemm}
\begin{proof}
For (a), after multiplying \(A\) by a nonzero integer, we may assume
that \(A\in\numZ^{m,m}\).  Choose a nonzero entry \(d\) of \(A\) and
use permutation matrices to move it to the \((m,m)\)-entry.  Using
elementary integral row and column operations, replace every other
entry in the last column and last row by its remainder modulo \(d\).
If a nonzero remainder occurs, move it to the \((m,m)\)-entry and
repeat.  The positive integer \(\lvert d\rvert\) strictly decreases
whenever this happens, so the process terminates.  At termination,
elementary row and column operations clear every off-diagonal entry in
the last row and column.  Thus there are
\(D_1,D_2\in\mathrm{GL}_m(\numZ)\) and a nonzero integer \(d\) such that
\[
 D_1^{-1}AD_2=\mathrm{diag}(A_0,d),
 \qquad A_0\in\numZ^{m-1,m-1}.
\]
Induction on \(m\) proves (a).

For (b), $B_j\numZ^{m,n}=\numZ^{m,n}$ for $j=1,2$ because
$B_j\in\mathrm{GL}_m(\numZ)$.  With
$D=\mathrm{diag}\left(\frac{q_1}{p_1},\ldots,\frac{q_m}{p_m}\right)$,
\begin{align*}
\numZ^{m,n}+A\numZ^{m,n}&=B_1\numZ^{m,n}+B_1B_1^{-1}AB_2B_2^{-1}\numZ^{m,n}=B_1\numZ^{m,n}+B_1D\numZ^{m,n}\\
&=B_1(\numZ^{m,n}+D\numZ^{m,n})=\widetilde{A}^{-1}\numZ^{m,n}.
\end{align*}
Finally, \(\lvert\det B_1\rvert=1\), so
\(\lvert\det\widetilde A\rvert=\prod_{i=1}^m p_i\).
\end{proof}
\begin{rema}
\label{rema:compute_AA1}
(1) If $A$ is integral, Lemma~\ref{lemm:findA1}(a) may be strengthened
by requiring each diagonal entry of $B_1^{-1}AB_2$ to be positive and
to divide the next.  The resulting matrix is the uniquely determined
Smith normal form of $A$.  In SageMath~\cite{Sag26}, it can be computed
by calling \lstinline!A.smith_form(integral=True)!; the transformation
matrices $B_1$ and $B_2$ are also available, as described in the
documentation returned by \lstinline!A.smith_form?!.

(2) If an integral matrix $\widetilde{A}$ satisfies
$\numZ^{m,n}+A\cdot\numZ^{m,n}=\widetilde{A}^{-1}\cdot\numZ^{m,n}$,
then $G\widetilde{A}$ may replace $\widetilde{A}$ for any $G\in\glnZ{m}$,
since $(G\widetilde{A})^{-1}\cdot\numZ^{m,n}=\widetilde{A}^{-1}G^{-1}\cdot\numZ^{m,n}=\widetilde{A}^{-1}\cdot\numZ^{m,n}$.  Thus
$\widetilde{A}$ may be reduced to row Hermite normal form.  In
SageMath~\cite{Sag26}, the Hermite normal form of a matrix $B$ is
computed by calling \lstinline!B.hermite_form()!.
\end{rema}

\subsection{Jacobi theta series of arbitrary degree}
\label{sec:Jacobi Theta series}
Let $n,m$ be positive integers, let $M$ be an $m\times m$ symmetric
positive-definite matrix over $\numQ$, and let $a,b \in \numR^{m,n}$.
The theta series considered in this paper has the general form
\begin{multline}
\label{eq:SiegelTheta}
\Theta_{M,n}[a,b](\tau,z):=\sum_{g\in\numZ^{m,n}}\etp{\tr\left(\frac{1}{2}\tau(g+a)^T\cdot M\cdot(g+a)+(g+a)^{T}\cdot M\cdot(z+b)\right)},
\end{multline}
where $(\tau,z)\in \mathfrak{H}_n \times \numC^{m,n}$.  This series
converges normally\footnote{Andrianov and Zhuravl\"ev~\cite[Prop.~1.3]{AZ95}
prove convergence for $z=a=b=0$; the general case follows similarly.
See also Freitag~\cite[p.~388, Prop.~8.2]{Fre11}.} and hence defines a
holomorphic function on $\mathfrak{H}_n \times \numC^{m,n}$.  The
corresponding theta constant is
\[
\theta_{M,n}[a,b](\tau):=\Theta_{M,n}[a,b](\tau,0).
\]

For $a,b \in \numQ^{m,n}$, these theta series and their transformation
laws lie in the framework of Siegel--Jacobi forms; see
Ziegler~\cite{Zie89} and Zhou, Wang, and Lu~\cite{ZWL10}.  One aim of
this paper is to derive
polynomial relations among $\Theta_{M,n}[a,b](\tau,z)$ without using
modular forms, Jacobi forms, or $q$-series.

\begin{examp}
\label{examp:cubic}
Let $m=2$, $M=\tbtmat{2}{-1}{-1}{2}$, and $n=1$.  Then
\begin{multline*}
\Theta_{M,n}[(a_1,a_2),(b_1,b_2)](\tau,(z_1,z_2))=\etp{2a_1b_1-a_1b_2-a_2b_1+2a_2b_2}\\
\cdot\sum_{g_1,g_2\in\numZ}\etp{2g_1b_1-g_1b_2-g_2b_1+2g_2b_2}q^{(g_1+a_1)^2-(g_1+a_1)(g_2+a_2)+(g_2+a_2)^2}\zeta_1^{g_1+a_1}\zeta_2^{g_2+a_2},
\end{multline*}
where \(\zeta_1=\etp{2z_1-z_2}\) and
\(\zeta_2=\etp{-z_1+2z_2}\).  Directly from the definitions,
\[
\begin{aligned}
 \theta_{M,n}[(0,0),(0,0)]&=a(q),\\
 \theta_{M,n}[(0,0),(1/3,1/3)]&=b(q),\\
 \theta_{M,n}[(2/3,1/3),(0,0)]&=c(q).
\end{aligned}
\]
The Borweins~\cite{BB91} and Hirschhorn, Garvan, and
Borwein~\cite{HGB93} studied the functions \(a(q)\), \(b(q)\), and
\(c(q)\).  More generally, the functions
\(\Theta_{M,n}[a,b](\tau,z)\), where
\begin{align}
a&\in\{(0,0),\,(2/3,1/3),\,(1/3,2/3)\},\label{eq:cubicavalues}\\
b&\in\{(0,0),\,(1/3,1/3),\,(2/3,2/3)\},\label{eq:cubicbvalues}
\end{align}
correspond to the nine functions $\Theta_{ij}(\mathbf{u}\vert\tau)$
studied by Schultz~\cite[p.~630]{Sch13}.  See also
Zhu~\cite[\S~7.2]{Zhu23} and Chan, Chan, Liu, and
Zudilin~\cite{CCLZ26}.  Zhu~\cite[Eqs.~(7.76) and~(7.77)]{Zhu23}
records the representatives in \eqref{eq:cubicavalues} and
\eqref{eq:cubicbvalues}.
\end{examp}

\begin{examp}
For $m=1$ and $M=(1)$, the function $\Theta_{(1),n}[0,0]$ is the
Riemann theta function, which plays an important role in the theory of
abelian varieties; see, for example, \cite{Lan23}.
\end{examp}

We conclude this subsection by observing that two characteristic pairs
may determine the same theta series up to a constant factor.
\begin{lemm}
\label{lemm:twochar}
Suppose that $a_1,b_1,a_2,b_2 \in \numR^{m,n}$ satisfy $a_2-a_1\in\numZ^{m,n}$ and $b_2-b_1\in M^{-1}\numZ^{m,n}$. Then we have
\[
\Theta_{M,n}[a_2,b_2]=\etp{\tr(a_1^TM(b_2-b_1))}\Theta_{M,n}[a_1,b_1].
\]
In particular, if $a_1\in\frac{1}{N}\numZ^{m,n}$, where $N$ is a positive integer, then
\[
\Theta_{M,n}[a_2,b_2]^N=\Theta_{M,n}[a_1,b_1]^N.
\]
\end{lemm}
\begin{proof}
Apply the change of variables $g=g'+a_1-a_2$ in the defining series for
$\Theta_{M,n}[a_2,b_2]$.  The result follows by direct calculation.
\end{proof}

\subsection{The main theorem}
\label{sec:The main theorem}
We now state and prove the main theorem.
\begin{thm}
\label{thm:main}
Let \(n,m,\mu\) be positive integers, let
\(M_1,M_2\in\mathrm{GL}_m^{\mathrm{pos}}(\numQ)\), and let
\(\tau\in\mathfrak H_n\) and
\(a_i,b_i\in\numR^{m,n}\), \(z_i\in\numC^{m,n}\), for
\(i=1,2,\ldots,\mu\).  Suppose that
\(A\in\mathrm{GL}_{m\mu}(\numQ)\) satisfies
\(A^TM_1^{(\mu)}A=M_2^{(\mu)}\).  Let \(\widetilde A\) be an
integral matrix of size \(m\mu\) such that
\(\numZ^{m\mu,n}+A\numZ^{m\mu,n}
=\widetilde A^{-1}\numZ^{m\mu,n}\).
Define
\begin{align*}
(a_1'^{T},\cdots,a_\mu'^T)&=(a_1^{T},\cdots,a_\mu^T)A^{-T},\\
(b_1'^{T},\cdots,b_\mu'^T)&=(b_1^{T},\cdots,b_\mu^T)A^{-T},\\
(w_1^{T},\cdots,w_\mu^T)&=(z_1^{T},\cdots,z_\mu^T)A^{-T}.
\end{align*}
Then
\begin{multline}
\label{eq:mainIdentity}
\abs{\det \widetilde{A}}^n\cdot\prod_{i=1}^{\mu}\Theta_{M_1,n}[a_i,b_i](\tau,z_i)=\\
\sum_{\alpha,\beta}\etp{-\tr\left(\sum_{i=1}^{\mu}a_i'^TM_2\beta_i\right)}\cdot\prod_{i=1}^{\mu}\Theta_{M_2,n}[a_i'+\alpha_i,b_i'+\beta_i](\tau,w_i),
\end{multline}
where \(\alpha=(\alpha_1^T,\ldots,\alpha_\mu^T)\) runs over a complete system of representatives for
\begin{equation}
\label{eq:alphaset}
(\widetilde{A}A)^{-1}\numZ^{m\mu,n}/\numZ^{m\mu,n},
\end{equation}
and \(\beta=(\beta_1^T,\ldots,\beta_\mu^T)\) runs over a complete system of representatives for
\begin{equation}
\label{eq:betaset}
A^{-1}{M_1^{(\mu)}}^{-1}\numZ^{m\mu,n}/
A^{-1}{M_1^{(\mu)}}^{-1}\widetilde{A}^T\numZ^{m\mu,n}.
\end{equation}
\end{thm}

We first prove a preliminary version.  Let
$n,m,\mu,M_1,M_2,a_i,b_i,z_i$ be as in Theorem~\ref{thm:main}.  For
$v_1,\cdots,v_\mu\in\numC^{m,n}$, define
\begin{equation}
\label{eq:bracket-notation}
[v_1,\cdots,v_\mu]:=(v_1^T,\cdots,v_\mu^T)^T\in\numC^{m\mu,n}
\end{equation}
for convenience.  For \(j=1,2\), we define the bilinear maps
\begin{align}
\label{eq:bilinear}
B_j\colon\numQ^{m\mu,n} \times \numQ^{m\mu,n} &\to \numQ^{n,n} \\
([v_1,\cdots,v_\mu], [v_1',\cdots,v_\mu']) &\mapsto \sum_{i=1}^{\mu}v_i^TM_jv_i',\notag
\end{align}
where \(v_i,v_i'\in\numQ^{m,n}\).  Equivalently,
$B_j(v,v')=v^TM_j^{(\mu)}v'$.  Let $\sigma$ be a $\numQ$-linear
isomorphism of $\numQ^{m\mu,n}$ satisfying\footnote{A more general
version may be obtained by assuming only that
$B_1(\sigma(v_1),\sigma(v_2))=r\cdot B_2(v_1,v_2)$ for all
$v_1,v_2\in\numQ^{m\mu,n}$ and some fixed $r\in\numQ$ with $r>0$.
This extension lies outside the scope of the paper.}
\begin{equation}
\label{eq:sigmaB1B2}
B_1(\sigma(v),\sigma(v'))=B_2(v,v')
\end{equation}
for every $v,v'\in\numQ^{m\mu,n}$.  Extend $B_j$ to
$\numC^{m\mu,n}\times\numC^{m\mu,n}$ by the same formula and extend
$\sigma$ canonically to a $\numC$-linear isomorphism of
$\numC^{m\mu,n}$.  Then \eqref{eq:sigmaB1B2} remains valid for
$v,v'\in\numC^{m\mu,n}$.  Define
\begin{align}
[a_1',\cdots,a_\mu']&=\sigma^{-1}[a_1,\cdots,a_\mu], \label{eq:asigma}\\
[b_1',\cdots,b_\mu']&=\sigma^{-1}[b_1,\cdots,b_\mu], \label{eq:bsigma}\\
[w_1,\cdots,w_\mu]&=\sigma^{-1}[z_1,\cdots,z_\mu]. \label{eq:zsigma}
\end{align}
Let
\begin{equation}
\label{eq:LandL1}
L_1=\numZ^{m\mu,n},\quad L_2=\sigma(\numZ^{m\mu,n}), \quad L=L_1+L_2.
\end{equation}
These are pairwise commensurable lattices in $\numQ^{m\mu,n}$.
\begin{lemm}
\label{lemm:thetaRelationLemma1}
If $\sigma$ satisfies the conditions above, then
\begin{multline}
\label{eq:thetaRelation}
\vert L/L_1\vert\cdot\prod_{i=1}^{\mu}\Theta_{M_1,n}[a_i,b_i](\tau,z_i)\\
=\sum_{\chi\in \widehat{L/L_1},g \in \sigma^{-1}(L)/L_1}\chi(\sigma(g))\sum_{h_1,\cdots,h_\mu\in\numZ^{m,n}}\chi(\sigma[h_1,\cdots,h_\mu])\\
\cdot\prod_{i=1}^{\mu}\etp{\tr\left(\frac{1}{2}\tau v_i^TM_2v_i+v_i^{T}M_2(w_i+b_i')\right)},
\end{multline}
where $v_i:=h_i+g_i+a_i'$ and $g_i\in\numQ^{m,n}$ is defined by $[g_1,\dots,g_\mu]:=g$.
\end{lemm}
\begin{proof}
For \(g=[g_1,\ldots,g_\mu]\in L\), define
\[
\Phi(g)=\prod_{i=1}^{\mu}\etp{\tr\left(\frac{1}{2}\tau (g_i+a_i)^TM_1(g_i+a_i)+(g_i+a_i)^{T}M_1(z_i+b_i)\right)}.
\]
Since \(L/L_1\) is finite and the defining theta series converges
absolutely on every coset of \(L_1\) (see
Subsection~\ref{sec:Jacobi Theta series}), \(\Phi\in L^1(L)\).
Lemma~\ref{lemm:L1sum} then implies
\begin{equation}
\label{eq:IgusaInserting}
\abs{L/L_1}\sum_{g\in L_1}\Phi(g)=\sum_{\chi\in\widehat{L/L_1},g\in L/L_2}\left(\sum_{h\in L_2}\chi(g+h)\Phi(g+h)\right).
\end{equation}
Here and below, sums indexed by quotient groups are taken over arbitrary
complete sets of representatives, and characters are evaluated on the
corresponding cosets.  The map
$L/L_2\rightarrow\sigma^{-1}(L)/L_1$ given by
$g+L_2\mapsto\sigma^{-1}(g)+L_1$ is a group isomorphism.  Hence
\begin{align}
&\sum_{\chi\in\widehat{L/L_1},g\in L/L_2}\left(\sum_{h\in L_2}\chi(g+h)\Phi(g+h)\right)\label{eq:tempSumchigh}\\
=&\sum_{\chi\in \widehat{L/L_1},g \in \sigma^{-1}(L)/L_1}\chi(\sigma(g))\sum_{h_1,\cdots,h_\mu\in\numZ^{m,n}}\chi(\sigma[h_1,\cdots,h_\mu])\Phi(\sigma(g+h)). \notag
\end{align}
Here $h=[h_1,\cdots,h_\mu]$.  By
\eqref{eq:sigmaB1B2}--\eqref{eq:zsigma},
\begin{align*}
\Phi(\sigma(g+h))=\prod_{i=1}^{\mu}\etp{\tr\left(\frac{1}{2}\tau v_i^TM_2v_i+v_i^{T}M_2(w_i+b_i')\right)}.
\end{align*}
Substitution into \eqref{eq:tempSumchigh} gives
\begin{equation}
\label{eq:RHSequality}
\text{RHS of \eqref{eq:IgusaInserting}}= \text{RHS of \eqref{eq:thetaRelation}}.
\end{equation}
Absolute convergence also gives
\begin{equation}
\label{eq:LHSequality}
\text{LHS of \eqref{eq:IgusaInserting}}= \text{LHS of \eqref{eq:thetaRelation}}.
\end{equation}
Equations \eqref{eq:IgusaInserting}, \eqref{eq:RHSequality}, and
\eqref{eq:LHSequality} now imply \eqref{eq:thetaRelation}.
\end{proof}

Theorem~\ref{thm:main} follows by applying the preceding lemma to a
specific $\sigma$.
\begin{lemm}
\label{lemm:thetaRelationLemma2}
Suppose $A$ is a rational matrix of size $m\mu$ such that
$A^T\cdot M_1^{(\mu)}\cdot A=M_2^{(\mu)}$.  Then the linear isomorphism
$\sigma(v)=Av$ satisfies \eqref{eq:sigmaB1B2}, and
\eqref{eq:thetaRelation} holds for this $\sigma$.
\end{lemm}
\begin{proof}
For arbitrary $v,v'\in\numQ^{m\mu,n}$,
\begin{equation*}
B_1(\sigma(v),\sigma(v'))=(Av)^TM_1^{(\mu)}(Av')=v^TM_2^{(\mu)}v'=B_2(v,v').
\end{equation*}
Thus \eqref{eq:sigmaB1B2} holds, and
Lemma~\ref{lemm:thetaRelationLemma1} gives \eqref{eq:thetaRelation}.
\end{proof}

The following lemma gives explicit expressions for $\chi(\sigma(g))$
and $\chi(\sigma(h))$ in \eqref{eq:thetaRelation}.
\begin{lemm}
\label{lemm:thetaRelationLemma3}
Under the assumptions and with the notation of Theorem~\ref{thm:main},
the map
\begin{align}
{M_2^{(\mu)}}^{-1}\widetilde{A}^{-T}\numZ^{m\mu,n}/{M_2^{(\mu)}}^{-1}\numZ^{m\mu,n}&\rightarrow\widehat{\widetilde{A}^{-1}\numZ^{m\mu,n}/\numZ^{m\mu,n}}\label{eq:isoDual}\\
b+{M_2^{(\mu)}}^{-1}\numZ^{m\mu,n}&\mapsto\left(v+\numZ^{m\mu,n}\mapsto\etp{\tr(b^TM_2^{(\mu)}\widetilde{A}v)}\right)\notag
\end{align}
is a group isomorphism.
\end{lemm}
\begin{proof}
Because $\widetilde{A}$ is integral, the domain and codomain of
\eqref{eq:isoDual} are well-defined finite abelian groups.  Write
$\overline{b}:=b+{M_2^{(\mu)}}^{-1}\numZ^{m\mu,n}$ and
$\overline{v}:=v+\numZ^{m\mu,n}$, and define
\begin{equation}
\label{eq:chib}
\chi_{\overline{b}}(\overline{v}):=\etp{\tr(b^TM_2^{(\mu)}\widetilde{A}v)}.
\end{equation}
For $b\in{M_2^{(\mu)}}^{-1}\widetilde{A}^{-T}\numZ^{m\mu,n}$ and
$v_1,v_2\in \widetilde{A}^{-1}\numZ^{m\mu,n}$ with
$\overline{v_1}=\overline{v_2}$, we have
$\chi_{\overline{b}}(\overline{v_1})=\chi_{\overline{b}}(\overline{v_2})$
because $b^TM_2^{(\mu)}\widetilde{A}(v_1-v_2)$ is integral.  Thus
$\chi_{\overline{b}}$ is well-defined on
$\widetilde{A}^{-1}\numZ^{m\mu,n}/\numZ^{m\mu,n}$ and is independent of
the representative of $\overline{b}$.  Moreover,
$\chi_{\overline{b}}(\overline{v_1}+\overline{v_2})=\chi_{\overline{b}}(\overline{v_1})\cdot\chi_{\overline{b}}(\overline{v_2})$
and $\chi_{\overline{b}}(\overline{0})=1$; hence
$\chi_{\overline{b}}\in\widehat{\widetilde{A}^{-1}\numZ^{m\mu,n}/\numZ^{m\mu,n}}$.
Similarly,
\(\chi_{\overline{b_1}+\overline{b_2}}
=\chi_{\overline{b_1}}\chi_{\overline{b_2}}\) and
\(\chi_{\overline0}(\overline v)=1\); hence \eqref{eq:isoDual} is a
group homomorphism.  To prove injectivity, suppose that
$\chi_{\overline{b}}(\overline{v})=1$ for all
$v\in \widetilde{A}^{-1}\numZ^{m\mu,n}$.  Then
$\tr(b^TM_2^{(\mu)}x)\in\numZ$ for every $x\in\numZ^{m\mu,n}$.
Taking $x$ successively to be the standard basis matrices gives
$b^TM_2^{(\mu)}\in \numZ^{n,m\mu}$ and hence
$b\in {M_2^{(\mu)}}^{-1}\numZ^{m\mu,n}$, so
$\overline{b}=\overline{0}$.  Thus the map is injective.  Its domain
and codomain have the same finite cardinality
$\abs{\det(\widetilde{A})}^n$, so it is also surjective.
\end{proof}

We now prove the main theorem.
\begin{proof}[Proof of Theorem~\ref{thm:main}]
By Lemma~\ref{lemm:thetaRelationLemma2}, \eqref{eq:thetaRelation} holds
for $\sigma(v)=Av$.  Equation \eqref{eq:LandL1} and the definition of
$\widetilde{A}$ give
$L/L_1=\widetilde{A}^{-1}\numZ^{m\mu,n}/\numZ^{m\mu,n}$ and hence
$\abs{L/L_1}=\abs{\det(\widetilde{A})}^n$.  Therefore
\begin{equation}
\label{eq:proofMain1}
\text{LHS of \eqref{eq:thetaRelation}}=\abs{\det \widetilde{A}}^n\cdot\prod_{i=1}^{\mu}\Theta_{M_1,n}[a_i,b_i](\tau,z_i).
\end{equation}
Let $\chi_{\overline{b}}$ be the character in \eqref{eq:chib}.
Lemma~\ref{lemm:thetaRelationLemma3} allows us to replace the sum
$\sum_{\chi\in\widehat{L/L_1}}f(\chi)$ by
$\sum_{\overline{b}\in{M_2^{(\mu)}}^{-1}\widetilde{A}^{-T}\numZ^{m\mu,n}/{M_2^{(\mu)}}^{-1}\numZ^{m\mu,n}}f(\chi_{\overline{b}})$.
Furthermore,
\begin{align*}
A^{-1}{M_1^{(\mu)}}^{-1}\numZ^{m\mu,n}/A^{-1}{M_1^{(\mu)}}^{-1}\widetilde{A}^T\numZ^{m\mu,n}&\rightarrow{M_2^{(\mu)}}^{-1}\widetilde{A}^{-T}\numZ^{m\mu,n}/{M_2^{(\mu)}}^{-1}\numZ^{m\mu,n}\\
\beta+A^{-1}{M_1^{(\mu)}}^{-1}\widetilde{A}^T\numZ^{m\mu,n}&\mapsto{M_2^{(\mu)}}^{-1}\widetilde{A}^{-T}M_1^{(\mu)}A\beta+{M_2^{(\mu)}}^{-1}\numZ^{m\mu,n}
\end{align*}
is a group isomorphism, so the sum may be written as
\[
\sum\nolimits_{\beta\in A^{-1}{M_1^{(\mu)}}^{-1}\numZ^{m\mu,n}/A^{-1}{M_1^{(\mu)}}^{-1}\widetilde{A}^T\numZ^{m\mu,n}}f\left(\chi_{\overline{{M_2^{(\mu)}}^{-1}\widetilde{A}^{-T}M_1^{(\mu)}A\beta}}\right).
\]
Moreover,
$\sigma^{-1}(L)/L_1=(\widetilde{A}A)^{-1}\numZ^{m\mu,n}/\numZ^{m\mu,n}$.
Thus the right-hand side of \eqref{eq:thetaRelation} becomes
\begin{multline}
\label{eq:proofMain2}
\sum_{\alpha,\beta}\chi_{\overline{{M_2^{(\mu)}}^{-1}\widetilde{A}^{-T}M_1^{(\mu)}A\beta}}(\overline{A\alpha})\sum_{h\in\numZ^{m\mu,n}}\chi_{\overline{{M_2^{(\mu)}}^{-1}\widetilde{A}^{-T}M_1^{(\mu)}A\beta}}(\overline{Ah})\\
\cdot\prod_{i=1}^{\mu}\etp{\tr\left(\frac{1}{2}\tau v_i^TM_2v_i+v_i^{T}M_2(w_i+b_i')\right)},
\end{multline}
where $v_i:=h_i+\alpha_i+a_i'$, $[h_1,\cdots,h_\mu]=h$, and
$[\alpha_1,\cdots,\alpha_\mu]=\alpha$; here \(\alpha\) and \(\beta\)
run over complete systems of representatives for the quotient groups in
\eqref{eq:alphaset} and \eqref{eq:betaset}, respectively.
Equation \eqref{eq:chib} gives
\begin{equation*}
\chi_{\overline{{M_2^{(\mu)}}^{-1}\widetilde{A}^{-T}M_1^{(\mu)}A\beta}}(\overline{A\alpha})=\etp{\tr(\beta^TA^TM_1^{(\mu)}A\alpha)}=\etp{\tr(\beta^TM_2^{(\mu)}\alpha)}.
\end{equation*}
The analogous formula holds for
$\chi_{\overline{{M_2^{(\mu)}}^{-1}\widetilde{A}^{-T}M_1^{(\mu)}A\beta}}(\overline{Ah})$.
Substitution into \eqref{eq:proofMain2} shows that the right-hand side
of \eqref{eq:thetaRelation} equals
\begin{multline*}
\sum_{\alpha,\beta}\etp{-\tr(\beta^TM_2^{(\mu)}a')}\sum_{h\in\numZ^{m\mu,n}}\etp{\tr(\beta^TM_2^{(\mu)}(h+\alpha+a'))}\\
\cdot\prod_{i=1}^{\mu}\etp{\tr\left(\frac{1}{2}\tau v_i^TM_2v_i+v_i^{T}M_2(w_i+b_i')\right)},
\end{multline*}
where $a'=[a_1',\cdots,a_\mu']$.  By the definition of
$\Theta_{M_2,n}$, this expression simplifies to
\[
\sum_{\alpha,\beta}\etp{-\tr\left(\sum_{i=1}^{\mu}a_i'^TM_2\beta_i\right)}\cdot\prod_{i=1}^{\mu}\Theta_{M_2,n}[a_i'+\alpha_i,b_i'+\beta_i](\tau,w_i).
\]
Together with \eqref{eq:proofMain1} and \eqref{eq:thetaRelation}, this
proves the theorem.
\end{proof}

\begin{rema}
(1) For given $M_1,M_2$, and $A$, Lemma~\ref{lemm:findA1} provides
the matrix $\widetilde{A}$ needed to apply Theorem~\ref{thm:main}; this
step is independent of $n$.  Choosing the matrix supplied by that lemma,
with $m$ replaced by $m\mu$, considerably simplifies the summation
ranges for $\alpha$ and $\beta$.  By
Remark~\ref{rema:compute_AA1}(2), $\widetilde{A}$ may also be taken in
Hermite normal form.

(2) The quotient group \eqref{eq:alphaset} is well-defined because the
assumption
$\numZ^{m\mu,n}+A\numZ^{m\mu,n}=\widetilde{A}^{-1}\numZ^{m\mu,n}$
implies $A\numZ^{m\mu,n}\subseteq \widetilde{A}^{-1}\numZ^{m\mu,n}$,
or equivalently
$\numZ^{m\mu,n}\subseteq(\widetilde{A}A)^{-1}\numZ^{m\mu,n}$.

(3) The independence of the right-hand side of \eqref{eq:mainIdentity}
from the representatives of $\alpha$ and $\beta$ requires verification.
For each $\overline{\alpha}=\alpha+\numZ^{m\mu,n}\in(\widetilde{A}A)^{-1}\numZ^{m\mu,n}/\numZ^{m\mu,n}$ and each $\overline{\beta}=\beta+A^{-1}{M_1^{(\mu)}}^{-1}\widetilde{A}^T\numZ^{m\mu,n}\in A^{-1}{M_1^{(\mu)}}^{-1}\numZ^{m\mu,n}/A^{-1}{M_1^{(\mu)}}^{-1}\widetilde{A}^T\numZ^{m\mu,n}$, the term
\[
\etp{-\tr\left(\sum_{i=1}^{\mu}a_i'^TM_2\beta_i\right)}\cdot\prod_{i=1}^{\mu}\Theta_{M_2,n}[a_i'+\alpha_i,b_i'+\beta_i](\tau,w_i)
\]
is independent of these choices.  To prove this, let
$\alpha'\in(\widetilde{A}A)^{-1}\numZ^{m\mu,n}$ and
$\beta'\in A^{-1}{M_1^{(\mu)}}^{-1}\numZ^{m\mu,n}$ satisfy
$\overline{\alpha}=\overline{\alpha'}$ and
$\overline{\beta}=\overline{\beta'}$.  Define
$\Delta\alpha_i=\alpha'_i-\alpha_i$ and
$\Delta\beta_i=\beta'_i-\beta_i$.  Then
\[
\Delta\beta=\beta'-\beta\in A^{-1}{M_1^{(\mu)}}^{-1}\widetilde{A}^T\numZ^{m\mu,n}={M_2^{(\mu)}}^{-1}(\widetilde{A}A)^T\numZ^{m\mu,n}\subseteq {M_2^{(\mu)}}^{-1}\numZ^{m\mu,n},
\]
and hence $\Delta\beta_i\in M_2^{-1}\numZ^{m,n}$, while
$\Delta\alpha_i\in\numZ^{m,n}$.  Lemma~\ref{lemm:twochar} therefore
gives
\begin{align*}
&\etp{-\tr\left(\sum_{i=1}^{\mu}a_i'^TM_2\beta'_i\right)}\\
&\quad{}\cdot\prod_{i=1}^{\mu}
\Theta_{M_2,n}[a_i'+\alpha'_i,b_i'+\beta'_i](\tau,w_i)\\
={}&\etp{-\tr\left(\sum_{i=1}^{\mu}a_i'^TM_2\beta'_i\right)}
\etp{\tr\left(\sum_{i=1}^{\mu}(a_i'+\alpha_i)^TM_2\Delta\beta_i\right)}\\
&\quad{}\cdot\prod_{i=1}^{\mu}
\Theta_{M_2,n}[a_i'+\alpha_i,b_i'+\beta_i](\tau,w_i)\\
={}&\etp{\tr\left(\sum_{i=1}^{\mu}\alpha_i^TM_2\Delta\beta_i\right)}
\etp{-\tr\left(\sum_{i=1}^{\mu}a_i'^TM_2\beta_i\right)}\\
&\quad{}\cdot\prod_{i=1}^{\mu}
\Theta_{M_2,n}[a_i'+\alpha_i,b_i'+\beta_i](\tau,w_i).
\end{align*}
Finally, $\alpha^T\in\numZ^{n,m\mu}\cdot(\widetilde{A}A)^{-T}$ and
$M_2^{(\mu)}\Delta\beta\in (\widetilde{A}A)^T\numZ^{m\mu,n}$, so
\begin{align*}
\etp{\tr\left(\sum_{i=1}^{\mu}\alpha_i^TM_2\Delta\beta_i\right)}=\etp{\tr(\alpha^TM_2^{(\mu)}\Delta\beta)}=1.
\end{align*}
\end{rema}

\subsection{Riemann's bilinear and quadrilinear relations}
\label{sub:riemann_formulas}
Specializing Theorem~\ref{thm:main} to $m=1$ and $M_1=(1)$ recovers
two classical Riemann theta relations.  The Riemann theta function is
usually defined as $\Theta_{(1),n}[0,0](\tau,z)$ for
$(\tau,z)\in \mathfrak{H}_n \times \numC^{1,n}$.  Throughout this
subsection, fix a positive integer $n$ and write
$\Theta[a,b](\tau,z):=\Theta_{(1),n}[a,b](\tau,z)$.

\begin{thm}[Riemann's bilinear relation]
\label{thm:bilinear-relation}
Let $n\geq1$, $\tau\in\mathfrak H_n$, $z_1,z_2\in\numC^{1,n}$, and $a_1,b_1,a_2,b_2\in\numR^{1,n}$. Then we have
\begin{multline*}
2^n\Theta[a_1,b_1](\tau,z_1)\cdot\Theta[a_2,b_2](\tau,z_2)=\sum_{b \in \frac{1}{2}\numZ^{1,n}/\numZ^{1,n}}\etp{-2b\cdot a_1^T} \\
\cdot\Theta[a_1+a_2, (b_1+b_2)/2+b]\left(\frac{\tau}{2},\frac{z_1+z_2}{2}\right)\cdot\Theta[a_1-a_2, (b_1-b_2)/2+b]\left(\frac{\tau}{2},\frac{z_1-z_2}{2}\right).
\end{multline*}
\end{thm}
\begin{proof}
In Theorem~\ref{thm:main}, take $m=1$, $\mu=2$, $M_1=(1)$,
$M_2=(1/2)$, and $A=\tbtmat{1/2}{1/2}{1/2}{-1/2}$.  Then
$A^T\cdot M_1^{(\mu)}\cdot A=M_2^{(\mu)}$, and
$\numZ^{m\mu,n}+A\numZ^{m\mu,n}=\widetilde{A}^{-1}\numZ^{m\mu,n}$
for $\widetilde{A}=A^{-1}$.  Hence \eqref{eq:mainIdentity} becomes
\begin{multline}
\label{eq:RiemannThetaBilinear}
2^n\Theta[a_1,b_1](\tau,z_1)\cdot\Theta[a_2,b_2](\tau,z_2)=\sum_{\beta \in A^{-1}\numZ^{2,n}/A^{-2}\numZ^{2,n}}\etp{-\frac{1}{2}\left(\beta_1\cdot a_1'^T+\beta_2\cdot a_2'^T\right)} \\
\cdot\Theta_{(1/2),n}[a_1', b_1'+\beta_1]\left(\tau,w_1\right)\cdot\Theta_{(1/2),n}[a_2', b_2'+\beta_2]\left(\tau,w_2\right),
\end{multline}
where
\begin{equation*}
\begin{pmatrix}a_1'\\a_2'\end{pmatrix}=A^{-1}\begin{pmatrix}a_1\\a_2\end{pmatrix}=\begin{pmatrix}a_1+a_2\\a_1-a_2\end{pmatrix},\quad
\begin{pmatrix}b_1'\\b_2'\end{pmatrix}=\begin{pmatrix}b_1+b_2\\b_1-b_2\end{pmatrix},\quad
\begin{pmatrix}w_1\\w_2\end{pmatrix}=\begin{pmatrix}z_1+z_2\\z_1-z_2\end{pmatrix}.
\end{equation*}
The map
\[
\frac{1}{2}\numZ^{1,n}/\numZ^{1,n}\rightarrow A^{-1}\numZ^{2,n}/A^{-2}\numZ^{2,n},\quad b+\numZ^{1,n}\mapsto\begin{pmatrix}2b\\2b\end{pmatrix}+A^{-2}\numZ^{2,n}
\]
is a group monomorphism.  Since
$\abs{\frac{1}{2}\numZ^{1,n}/\numZ^{1,n}}=2^n=\abs{A^{-1}\numZ^{2,n}/A^{-2}\numZ^{2,n}}$,
it is an isomorphism.  Apply the corresponding change of summation
variables to the right-hand side of \eqref{eq:RiemannThetaBilinear} and
use
\begin{align*}
\Theta_{(1/2),n}[a_1', b_1'+2b]\left(\tau,w_1\right)&=\Theta[a_1+a_2, (b_1+b_2)/2+b]\left(\frac{\tau}{2},\frac{z_1+z_2}{2}\right),\\
\Theta_{(1/2),n}[a_2', b_2'+2b]\left(\tau,w_2\right)&=\Theta[a_1-a_2, (b_1-b_2)/2+b]\left(\frac{\tau}{2},\frac{z_1-z_2}{2}\right),
\end{align*}
to obtain the stated formula.
\end{proof}

\begin{thm}[Riemann's quadrilinear relation]
\label{thm:quadrilinear}
Let $n\geq1$, $\tau\in\mathfrak H_n$, $z_i\in\numC^{1,n}$, and $a_i,b_i\in\numR^{1,n}$ for $i=1,2,3,4$. Then
\begin{equation}
\label{eq:Riemann-theta-relation}
2^n\prod_{i=1}^4\Theta[a_i,b_i](\tau,z_i)=\sum_{a,b\in\frac{1}{2}\numZ^{1,n}/\numZ^{1,n}}\etp{-2b\cdot a_1^T}\prod_{i=1}^4\Theta[a_i'+a,b_i'+b](\tau,w_i),
\end{equation}
where
\begin{align*}
(a_1'^T,a_2'^T,a_3'^T,a_4'^T)&=(a_1^T,a_2^T,a_3^T,a_4^T)U,\\
(b_1'^T,b_2'^T,b_3'^T,b_4'^T)&=(b_1^T,b_2^T,b_3^T,b_4^T)U,\\
(w_1^T,w_2^T,w_3^T,w_4^T)&=(z_1^T,z_2^T,z_3^T,z_4^T)U,
\end{align*}
and
\[
U=\frac{1}{2}\begin{pmatrix}
1 & 1 & 1 & 1\\
1 & 1 & -1 & -1\\
1 & -1 & 1 & -1\\
1 & -1 & -1 & 1
\end{pmatrix}.
\]
\end{thm}
\begin{proof}
The proof is similar to that of the preceding theorem, so we give only
a sketch.  In Theorem~\ref{thm:main}, take $m=1$, $\mu=4$,
$M_1=M_2=(1)$, and $A=U^{-1}=U$.  Lemma~\ref{lemm:findA1} and
Remark~\ref{rema:compute_AA1} give
\begin{equation*}
\widetilde{A}=\begin{pmatrix}
 1 & 0 & 0 & 1\\
 0 & 1 & 0 & 1\\
 0 & 0 & 1 & 1\\
 0 & 0 & 0 & 2\\
\end{pmatrix}.
\end{equation*}
All hypotheses of Theorem~\ref{thm:main} are therefore satisfied.  The
representative sets needed to reduce \eqref{eq:mainIdentity} to
\eqref{eq:Riemann-theta-relation} are given by the isomorphisms
\begin{align*}
\frac{1}{2}\numZ^{1,n}/\numZ^{1,n}&\rightarrow(\widetilde{A}A)^{-1}\numZ^{4,n}/\numZ^{4,n}\\
a+\numZ^{1,n}&\mapsto[a,a,a,a]+\numZ^{4,n}
\end{align*}
and
\begin{align*}
\frac{1}{2}\numZ^{1,n}/\numZ^{1,n}&\rightarrow A^{-1}\numZ^{4,n}/A^{-1}\widetilde{A}^T\numZ^{4,n}\\
b+\numZ^{1,n}&\mapsto[b,b,b,b]+A^{-1}\widetilde{A}^T\numZ^{4,n}.
\end{align*}
For these representatives, \(\sum_{i=1}^4a_i'=2a_1\), so the
exterior phase in \eqref{eq:mainIdentity} is
\(\etp{-2b\cdot a_1^T}\).  Substituting the two representative sets
into \eqref{eq:mainIdentity} gives
\eqref{eq:Riemann-theta-relation}.
\end{proof}

\section{Cubic theta identities}
\label{r: cubic theta section}

\subsection{The cubic Fourier identity}

We begin with a direct specialization of Theorem~\ref{thm:main}.
Throughout this section, we use
\begin{equation}
\label{r: cubic basic data}
 M=\tbtMat{2}{-1}{-1}{2},\qquad
 \rho=\frac13\begin{pmatrix}1\\-1\end{pmatrix},\qquad
 \sigma=\frac13\begin{pmatrix}1\\1\end{pmatrix},\qquad
 \omega=\etp{1/3}.
\end{equation}
Recall that \(\F_3=\numZ/3\numZ\).  The choice of integral
representatives is immaterial and hence left implicit.
For row vectors
\(\mathbf p,\mathbf q\in\F_3^n\), we write
\begin{equation}
\label{r: indexed cubic theta}
 \Theta^{(n)}_{\mathbf p,\mathbf q}(\tau,z)
 :=\Theta_{M,n}[\rho\mathbf p,\sigma\mathbf q](\tau,z),
 \qquad
 \theta^{(n)}_{\mathbf p,\mathbf q}(\tau)
 :=\Theta^{(n)}_{\mathbf p,\mathbf q}(\tau,0).
\end{equation}
Lemma~\ref{lemm:twochar} shows that
\(\Theta^{(n)}_{\mathbf p,\mathbf q}\) is independent of the
integral representatives of \(\mathbf p\) and \(\mathbf q\).
Scalar products of row vectors are taken modulo \(3\); thus
\(\mathbf p\cdot\mathbf v=\mathbf p\mathbf v^T\).
\begin{thm}
\label{r: cubic Fourier theorem}
Let \(n\geq1\), \(\tau\in\mathfrak H_n\),
\(z\in\numC^{2,n}\), and \(C=\tbtmat{1}{1}{-1}{2}\).
For \(\mathbf p,\mathbf q\in\F_3^n\), we have
\begin{equation}
\label{r: cubic Fourier formula}
 3^n\Theta^{(n)}_{\mathbf p,\mathbf q}(\tau,z)^3
 =\sum_{\mathbf u,\mathbf v\in\F_3^n}
  \omega^{\mathbf p\cdot\mathbf v+\mathbf q\cdot\mathbf u}
  \Theta^{(n)}_{\mathbf u,\mathbf v}(\tau,Cz)
  \theta^{(n)}_{\mathbf u,\mathbf v}(\tau)^2.
\end{equation}
In particular,
\begin{equation}
\label{r: cubic constant Fourier formula}
 3^n\theta^{(n)}_{\mathbf p,\mathbf q}(\tau)^3
 =\sum_{\mathbf u,\mathbf v\in\F_3^n}
  \omega^{\mathbf p\cdot\mathbf v+\mathbf q\cdot\mathbf u}
  \theta^{(n)}_{\mathbf u,\mathbf v}(\tau)^3.
\end{equation}
\end{thm}
\begin{proof}
We apply Theorem~\ref{thm:main} with
\(m=2\), \(\mu=3\), and \(M_1=M_2=M\), where
\begin{equation*}
 A=\frac13
 \begin{pmatrix}
  2&-1& 2&-1& 2&-1\\
  1& 1& 1& 1& 1& 1\\
  2&-1&-1&-1&-1& 2\\
  1& 1& 1&-2&-2& 1\\
  2&-1&-1& 2&-1&-1\\
  1& 1&-2& 1& 1&-2
 \end{pmatrix},
\end{equation*}
and
\begin{equation*}
 \widetilde{A}=
 \begin{pmatrix}
 1&0&0&0&0&-2\\
 0&1&0&0&0&-1\\
 0&0&1&0&0&-2\\
 0&0&0&1&0&-1\\
 0&0&0&0&1&-2\\
 0&0&0&0&0& 3
 \end{pmatrix}.
\end{equation*}
Direct multiplication gives
\begin{equation}
\label{r: cubic matrix identities}
 A^{T}M^{(3)}A=M^{(3)},
 \qquad \det A=1,
 \qquad \det \widetilde{A}=3.
\end{equation}
To verify the lattice condition, let
\[
 u_0=(2,1,2,1,2,1)^{T}\in\numZ^{6,1}.
\]
Every column of \(3A\) is congruent to \(u_0\) modulo
\(3\numZ^{6,1}\), and the first column is \(u_0\).  Hence
\begin{equation*}
 \numZ^{6,n}+A\numZ^{6,n}
 =\numZ^{6,n}+\frac13u_0\numZ^{1,n}.
\end{equation*}
The first five columns of \(\widetilde{A}^{-1}\) are \(e_1,\ldots,e_5\),
and the last is \(u_0/3\).  Since
\(e_6=u_0-2e_1-e_2-2e_3-e_4-2e_5\), these columns generate
\(\numZ^{6,1}+\frac13\numZ u_0\).  Consequently,
\begin{equation}
\label{r: cubic lattice condition}
 \numZ^{6,n}+A\numZ^{6,n}=\widetilde{A}^{-1}\numZ^{6,n}.
\end{equation}
Thus all the hypotheses of Theorem~\ref{thm:main} are satisfied.

We specialize the characteristics and Jacobi variables as follows:
\[
 z_1=z_2=z_3=z,\qquad
 a_1=a_2=a_3=\rho\mathbf p,\qquad
 b_1=b_2=b_3=\sigma\mathbf q.
\]
The left-hand side of \eqref{eq:mainIdentity} is
\begin{equation}
\label{r: cubic theorem left side}
 \abs{\det \widetilde{A}}^n\prod_{i=1}^{3}
 \Theta_{M,n}[a_i,b_i](\tau,z_i)
 =3^n\Theta^{(n)}_{\mathbf p,\mathbf q}(\tau,z)^3.
\end{equation}

It remains to identify the representatives in
\eqref{eq:alphaset} and \eqref{eq:betaset}.  For this purpose, we write
\[
 \widehat\rho=[\rho,\rho,\rho],\qquad
 \widehat\sigma=[\sigma,\sigma,\sigma].
\]
The notation is that of \eqref{eq:bracket-notation}.
By \eqref{r: cubic lattice condition}, \(\widetilde{A}A\) is integral, while
\eqref{r: cubic matrix identities} gives \(\abs{\det(\widetilde{A}A)}=3\).
Moreover, \(A\widehat\rho=e_1\).
Since \(\widetilde{A}A\widehat\rho=e_1\) and
\(\widehat\rho\notin\numZ^{6,1}\), the class of
\(\widehat\rho\) is nonzero in the one-column quotient
\((\widetilde{A}A)^{-1}\numZ^{6,1}/\numZ^{6,1}\), which has order
\(3\).
Thus
\[
 \{\widehat\rho\mathbf u:\mathbf u\in\F_3^n\}
\]
is a complete set of representatives for
\((\widetilde{A}A)^{-1}\numZ^{6,n}/\numZ^{6,n}\).
For the second quotient, multiplication by \(M^{(3)}A\) induces
the isomorphism
\begin{equation}
\label{r: cubic beta quotient}
 \frac{A^{-1}{M^{(3)}}^{-1}\numZ^{6,1}}
      {A^{-1}{M^{(3)}}^{-1}\widetilde{A}^{T}\numZ^{6,1}}
 \rightarrow
 \frac{\numZ^{6,1}}{\widetilde{A}^{T}\numZ^{6,1}}.
\end{equation}
A direct calculation gives
\begin{equation*}
 M^{(3)}A\widehat\sigma=e_2,
 \qquad
 \widetilde{A}^{-T}e_2
 =\begin{pmatrix}0&1&0&0&0&1/3\end{pmatrix}^{T}
 \notin\numZ^{6,1}.
\end{equation*}
The quotient on the right of \eqref{r: cubic beta quotient} has order
\(3\).  Hence the class of \(e_2\) generates it, and
\[
 \{\widehat\sigma\mathbf v:\mathbf v\in\F_3^n\}
\]
is a complete set of representatives for
\[
 \frac{A^{-1}{M^{(3)}}^{-1}\numZ^{6,n}}
 {A^{-1}{M^{(3)}}^{-1}\widetilde{A}^T\numZ^{6,n}}.
\]
Since both quotients are formed columnwise, the preceding one-column
computations establish the asserted systems of representatives for
every \(n\).

For every \(V\in\numC^{2,n}\), direct multiplication gives
\begin{equation*}
 A^{-1}\begin{pmatrix}V\\V\\V\end{pmatrix}
 =\begin{pmatrix}CV\\0\\0\end{pmatrix},\qquad
 C\rho=\begin{pmatrix}0\\-1\end{pmatrix},\qquad
 C\sigma=\begin{pmatrix}2/3\\1/3\end{pmatrix},\qquad
 MC\sigma=\begin{pmatrix}1\\0\end{pmatrix}.
\end{equation*}
Thus the transformed parameters in Theorem~\ref{thm:main} are
\[
 (a'_1,a'_2,a'_3)=(C\rho\mathbf p,0,0),\quad
 (b'_1,b'_2,b'_3)=(C\sigma\mathbf q,0,0),\quad
 (w_1,w_2,w_3)=(Cz,0,0).
\]
For these representatives, the exterior factor in
\eqref{eq:mainIdentity} is
\[
 \etp{-\tr\left((C\rho\mathbf p)^TM\sigma\mathbf v\right)}
 =\omega^{\mathbf p\cdot\mathbf v}.
\]
By Lemma~\ref{lemm:twochar},
\begin{align}
\Theta_{M,n}[C\rho\mathbf p+\rho\mathbf u,C\sigma\mathbf q+\sigma\mathbf v](\tau,Cz)&=\etp{\tr((\rho\mathbf u)^TMC\sigma\mathbf q)}\Theta_{M,n}[\rho\mathbf u,\sigma\mathbf v](\tau,Cz)\notag\\
&=\omega^{\mathbf q\cdot\mathbf u}\Theta^{(n)}_{\mathbf u,\mathbf v}(\tau,Cz).\label{r: cubic characteristic transformation}
\end{align}
Therefore,
\begin{equation*}
\text{RHS of }\eqref{eq:mainIdentity}=\sum_{\mathbf u,\mathbf v\in\F_3^n}\omega^{\mathbf p\cdot\mathbf v+\mathbf q\cdot\mathbf u}\Theta^{(n)}_{\mathbf u,\mathbf v}(\tau,Cz)\theta^{(n)}_{\mathbf u,\mathbf v}(\tau)^2.
\end{equation*}
Combining this identity with \eqref{r: cubic theorem left side} proves
\eqref{r: cubic Fourier formula}.  The specialization \(z=0\) gives
\eqref{r: cubic constant Fourier formula}.
\end{proof}

The following lemma supplies the two reductions used below.  Its
vanishing assertion turns the Fourier transform of
\eqref{r: cubic Fourier formula} into a family of relations, while its sign
symmetries reduce the number of theta series occurring in degree two.

\begin{lemm}
\label{lemm:cubic-vanishing-symmetry}
Let \(n\geq1\).  For
\(\mathbf p,\mathbf q\in\F_3^n\), we have
\begin{equation}
\label{r: cubic vanishing}
 \mathbf p\cdot\mathbf q\ne0
 \quad\Longrightarrow\quad
 \theta^{(n)}_{\mathbf p,\mathbf q}=0,
\end{equation}
and
\begin{equation}
\label{r: cubic sign symmetry}
 \theta^{(n)}_{\mathbf p,\mathbf q}
 =\theta^{(n)}_{-\mathbf p,\mathbf q}
 =\theta^{(n)}_{\mathbf p,-\mathbf q}
 =\theta^{(n)}_{-\mathbf p,-\mathbf q}.
\end{equation}
\end{lemm}
\begin{proof}
The matrices
\[
 R=\begin{pmatrix}0&-1\\1&-1\end{pmatrix},\qquad
 J=\begin{pmatrix}0&1\\1&0\end{pmatrix}
\]
satisfy
\[
 \begin{gathered}
 R^{T}MR=M,\qquad J^{T}MJ=M,\\
 R\rho-\rho=\begin{pmatrix}0\\1\end{pmatrix},\qquad
 M(R\sigma-\sigma)=\begin{pmatrix}-1\\0\end{pmatrix},\\
 \rho^TM(R\sigma-\sigma)=-\frac13.
 \end{gathered}
\]
The change of variables induced by \(R\), followed by
Lemma~\ref{lemm:twochar}, gives
\[
 \theta^{(n)}_{\mathbf p,\mathbf q}
 =\theta_{M,n}[R\rho\mathbf p,R\sigma\mathbf q]
 =\omega^{-\mathbf p\cdot\mathbf q}
  \theta^{(n)}_{\mathbf p,\mathbf q},
\]
which proves \eqref{r: cubic vanishing}.  The changes induced by
\(J\) and \(-J\), together with
\[
 J\rho=-\rho,\qquad J\sigma=\sigma,\qquad
 (-J)\rho=\rho,\qquad(-J)\sigma=-\sigma,
\]
give \eqref{r: cubic sign symmetry}.
\end{proof}

\begin{coro}
\label{r: cubic character identity}
Let \(n\geq1\), \(\tau\in\mathfrak H_n\), and
\(z\in\numC^{2,n}\).  If
\(\boldsymbol\lambda,\boldsymbol\kappa\in\F_3^n\) and
\(\boldsymbol\lambda\cdot\boldsymbol\kappa\ne0\), then
\begin{equation}
\label{r: cubic character formula}
 \sum_{\boldsymbol a,\boldsymbol b\in\F_3^n}
 \omega^{\boldsymbol\lambda\cdot\boldsymbol a-
          \boldsymbol\kappa\cdot\boldsymbol b}
 \Theta^{(n)}_{\boldsymbol a,\boldsymbol b}(\tau,z)^3=0.
\end{equation}
\end{coro}

\begin{proof}
We multiply \eqref{r: cubic Fourier formula} by
\(\omega^{\boldsymbol\lambda\cdot\mathbf p-
\boldsymbol\kappa\cdot\mathbf q}\) and sum over
\(\mathbf p,\mathbf q\in\F_3^n\).  The left-hand side of the
resulting equality is
\[
 3^n\sum_{\mathbf p,\mathbf q\in\F_3^n}
 \omega^{\boldsymbol\lambda\cdot\mathbf p-
          \boldsymbol\kappa\cdot\mathbf q}
 \Theta^{(n)}_{\mathbf p,\mathbf q}(\tau,z)^3,
\]
whereas its right-hand side is
\[
 \sum_{\mathbf u,\mathbf v\in\F_3^n}\Theta^{(n)}_{\mathbf u,\mathbf v}(\tau,Cz)\theta^{(n)}_{\mathbf u,\mathbf v}(\tau)^2
 \left(\sum_{\mathbf p\in\F_3^n}\omega^{\mathbf p\cdot(\boldsymbol\lambda+\mathbf v)}\right)
 \left(\sum_{\mathbf q\in\F_3^n}\omega^{\mathbf q\cdot(\mathbf u-\boldsymbol\kappa)}\right).
\]
For \(\mathbf x\in\F_3^n\), character orthogonality gives
\[
 \sum_{\mathbf p\in\F_3^n}\omega^{\mathbf p\cdot\mathbf x}
 =\begin{cases}3^n,&\mathbf x=\boldsymbol0,\\0,&\mathbf x\ne\boldsymbol0.
 \end{cases}
\]
Consequently, only
\((\mathbf u,\mathbf v)=(\boldsymbol\kappa,-\boldsymbol\lambda)\)
survives.  Dividing by \(3^n\) yields
\[
 \sum_{\boldsymbol a,\boldsymbol b\in\F_3^n}
 \omega^{\boldsymbol\lambda\cdot\boldsymbol a-
          \boldsymbol\kappa\cdot\boldsymbol b}
 \Theta^{(n)}_{\boldsymbol a,\boldsymbol b}(\tau,z)^3
 =3^n
 \Theta^{(n)}_{\boldsymbol\kappa,-\boldsymbol\lambda}(\tau,Cz)
 \theta^{(n)}_{\boldsymbol\kappa,-\boldsymbol\lambda}(\tau)^2.
\]
The right-hand side vanishes by \eqref{r: cubic vanishing}, since
\(\boldsymbol\lambda\cdot\boldsymbol\kappa\ne0\).
\end{proof}

\subsection{Schultz's identity}
\label{sec:schultz_identity}
We give an independent proof of Schultz's identity~\cite{Sch13}.

To formulate the indexed version of the identity, let
\(\tau\in\mathfrak H_1\), \(q=\etp{\tau}\),
\(Q(x,y)=x^2+xy+y^2\), and
\(\boldsymbol u=(u_1,u_2)^T\in\numC^2\).  Write
\(x=\etp{u_1}\) and \(y=\etp{u_2}\).  For \(r,s\in\F_3\), define the
normalized, doubly indexed \(q\)-series below, with \(r\) and \(s\)
represented by \(0,1,2\):
\begin{equation}
\label{r: Schultz q series}
 \mathcal T_{r,s}(\boldsymbol u\mid\tau)
 :=\omega^{rs}\sum_{m,n\in\numZ}
 \omega^{s(m-n)}q^{Q(m+r/3,n+r/3)}
 x^{m+r/3}y^{n+r/3}.
\end{equation}
The fractional powers in \eqref{r: Schultz q series} are interpreted
through the additive variables \(\tau,u_1,u_2\), rather than through
choices of logarithms of \(q,x,y\).  With
\[
 z(\boldsymbol u)=M^{-1}\begin{pmatrix}u_1\\-u_2\end{pmatrix},
\]
the substitution \(g_1=m\), \(g_2=-n\) in
\eqref{eq:SiegelTheta} gives, for all \(r,s\in\F_3\),
\[
 \mathcal T_{r,s}(\boldsymbol u\mid\tau)
 =\omega^{rs}\Theta^{(1)}_{r,s}(\tau,z(\boldsymbol u)).
\]

Under the identification above, the following family is equivalent to
that established by Schultz~\cite[Corollary~4.3]{Sch13}.
Schultz~\cite[p.~622]{Sch13} proved its \(r=s=1\) case earlier in the
same paper.  Chan, Chan, Liu, and Zudilin~\cite[Theorem~1.1]{CCLZ26}
restated this case.  The same authors~\cite[Sections~2 and~4]{CCLZ26}
gave two further proofs.
Zhu~\cite[Theorem~7.12]{Zhu23} gave a computer-assisted proof using
the theory of Jacobi forms.
\begin{thm}[Schultz]
\label{r: Schultz theorem}
Let \(\tau\in\mathfrak H_1\) and \(\boldsymbol u\in\numC^2\).  For
\(r,s\in\F_3\), we have
\begin{equation}
\label{r: Schultz cubic formula}
 \mathcal T_{r,s}(\boldsymbol u\mid\tau)^3
 +\mathcal T_{0,0}(\boldsymbol u\mid\tau)^3
 =\mathcal T_{r,0}(\boldsymbol u\mid\tau)^3
 +\mathcal T_{0,s}(\boldsymbol u\mid\tau)^3.
\end{equation}
\end{thm}
\begin{proof}
Fix \(r,s\in\F_3\).  In \eqref{r: cubic Fourier formula}, take
\(n=1\) and \(z=z(\boldsymbol u)\).  Adding the instances indexed by
\((r,s)\) and \((0,0)\) and subtracting those indexed by \((r,0)\)
and \((0,s)\) gives the right-hand side
\[
 \sum_{a,b\in\F_3}(1-\omega^{sa})(1-\omega^{rb})
 \Theta^{(1)}_{a,b}(\tau,Cz(\boldsymbol u))
 \theta^{(1)}_{a,b}(\tau)^2.
\]
Every summand vanishes.  Indeed, if \(sa=0\) or \(rb=0\), then its
coefficient is zero; otherwise \(a,b\ne0\), so \(ab\ne0\) in
\(\F_3\), and \eqref{r: cubic vanishing} gives
\(\theta^{(1)}_{a,b}(\tau)=0\).  Since
\(\mathcal T_{r,s}^3=(\Theta^{(1)}_{r,s})^3\), the corresponding
left-hand side is three times the difference between the two sides of
\eqref{r: Schultz cubic formula}.
\end{proof}

\medskip
\noindent\emph{Remark.}
For the four functions in the Introduction, the definitions in
\eqref{eq:intro Schultz functions} and
\eqref{r: Schultz q series} give
\[
 \Theta_1=\mathcal T_{1,1},\qquad
 \Theta_2=\mathcal T_{1,0},\qquad
 \Theta_3=\mathcal T_{0,0},\qquad
 \Theta_4=\mathcal T_{0,1}.
\]
Thus the \(r=s=1\) member of
\eqref{r: Schultz cubic formula} is precisely
\eqref{eq:intro Schultz identity}, originally proved by
Schultz~\cite[p.~622]{Sch13}.

At \(\boldsymbol u=0\), we use the customary notation
\begin{equation}
\label{r: Borwein series}
 \begin{aligned}
 a(q)&=\mathcal T_{0,0}(0\mid\tau)
      =\sum_{m,n\in\numZ}q^{Q(m,n)},\\
 b(q)&=\mathcal T_{0,1}(0\mid\tau)
      =\sum_{m,n\in\numZ}\omega^{m-n}q^{Q(m,n)},\\
 c(q)&=\mathcal T_{1,0}(0\mid\tau)
      =\sum_{m,n\in\numZ}q^{Q(m+1/3,n+1/3)}.
 \end{aligned}
\end{equation}
For \(r=s=1\) and \(\boldsymbol u=0\), the identification above and
\eqref{r: cubic vanishing} show that
\(\mathcal T_{1,1}(0\mid\tau)=0\).  Thus
\eqref{r: Schultz cubic formula} specializes to the cubic identity of
Borwein and Borwein~\cite[Eq.~(2.3)]{BB91}.
\begin{coro}
\label{coro:Borweins}
For \(\abs{q}<1\), we have
\begin{equation}
\label{r: Borwein formula}
 a(q)^3=b(q)^3+c(q)^3.
\end{equation}
\end{coro}

\subsection{Degree two}
\label{r: cubic degree two subsection}
We derive degree-two analogues of the Borweins' cubic identity
(Corollary~\ref{coro:Borweins}).


Let
\[
 \tau=\begin{pmatrix}\tau_1&\tau_0\\\tau_0&\tau_2\end{pmatrix}
 \in\mathfrak H_2,\qquad
 q=\etp{\tau_1},\quad s=\etp{\tau_0},\quad t=\etp{\tau_2}.
\]
For
\(\boldsymbol a=(a_1,a_2)\) and
\(\boldsymbol b=(b_1,b_2)\) in \(\F_3^2\), the definition
\eqref{r: indexed cubic theta} gives the following expansion, obtained
by writing the two columns of \(g+\rho\boldsymbol a\) as
\((j,-k)^T\) and \((r,-h)^T\):
\begin{equation}
\label{r: degree two q series}
 \theta^{(2)}_{\boldsymbol a,\boldsymbol b}(\tau)
 =\sum_{\substack{j,k\in\numZ+a_1/3\\r,h\in\numZ+a_2/3}}
 \omega^{b_1(j-k)+b_2(r-h)}q^{Q(j,k)}
 s^{(2j+k)r+(j+2k)h}t^{Q(r,h)}.
\end{equation}
We use the notation
\[
\boldsymbol 0=(0,0),\quad
 \boldsymbol e_1=(1,0),\quad
 \boldsymbol e_2=(0,1),\quad
 \boldsymbol d_+=(1,1),\quad
 \boldsymbol d_-=(1,-1)=(1,2).
\]
The pairs \((\boldsymbol a,\boldsymbol b)\) satisfying
\(\boldsymbol a\cdot\boldsymbol b=0\), modulo the two sign
symmetries in \eqref{r: cubic sign symmetry}, are represented by
\begin{equation}
\label{r: degree two representatives}
\begin{aligned}
\mathcal I=\{&
(\boldsymbol0,\boldsymbol0),
(\boldsymbol0,\boldsymbol e_1),(\boldsymbol0,\boldsymbol e_2),
(\boldsymbol0,\boldsymbol d_+),(\boldsymbol0,\boldsymbol d_-),\\
&(\boldsymbol e_1,\boldsymbol0),(\boldsymbol e_2,\boldsymbol0),
(\boldsymbol d_+,\boldsymbol0),(\boldsymbol d_-,\boldsymbol0),
(\boldsymbol e_1,\boldsymbol e_2),(\boldsymbol e_2,\boldsymbol e_1),
(\boldsymbol d_+,\boldsymbol d_-),(\boldsymbol d_-,\boldsymbol d_+)\}.
\end{aligned}
\end{equation}
The first five representatives in \(\mathcal I\) have
\(\boldsymbol a=\boldsymbol0\), the next four have
\(\boldsymbol b=\boldsymbol0\) and \(\boldsymbol a\ne\boldsymbol0\),
and the final four have both entries nonzero.  These three groups
represent precisely five, four, and four orbits, respectively, under the two sign
symmetries in \eqref{r: cubic sign symmetry}.  By
Lemma~\ref{lemm:cubic-vanishing-symmetry}, every nonzero indexed series
is therefore represented by one of the 13 pairs in \(\mathcal I\).
A comparison of the lowest-order terms in
\eqref{r: degree two q series} shows that these 13 series are nonzero
and pairwise distinct.  For completeness, let
\(c_{A,B,C}=[q^As^Bt^C]\) and consider, in order, the seven triples
\[
 (1/3,-1/3,1/3),\ (1/3,1/3,1/3),\ (1/3,0,1),\
 (1,0,1/3),\ (0,0,1),\ (1,0,0),\ (1,-1,1).
\]
Direct extraction from \eqref{r: degree two q series} gives the
following coefficient vectors, in the order in which the 13
representatives occur in \(\mathcal I\):
\[
\begin{gathered}
 (0,0,0,0,6,6,12),\quad (0,0,0,0,6,-3,-6),\quad
 (0,0,0,0,-3,6,-6),\\
 (0,0,0,0,-3,-3,-6),\quad (0,0,0,0,-3,-3,12),\\
 (0,0,6,0,0,0,0),\quad (0,0,0,6,0,0,0),\quad
 (6,0,0,0,0,0,0),\quad (0,6,0,0,0,0,0),\\
 (0,0,-3,0,0,0,0),\quad (0,0,0,-3,0,0,0),\quad
 (-3,0,0,0,0,0,0),\quad (0,-3,0,0,0,0,0).
\end{gathered}
\]
These vectors show again that the 13 series are nonzero and pairwise
distinct.  We index them by
\((\boldsymbol a,\boldsymbol b)\in\mathcal I\).
In the remainder of this subsection, we suppress the superscript and
write $\theta_{\boldsymbol a,\boldsymbol b}(\tau)=\theta^{(2)}_{\boldsymbol a,\boldsymbol b}(\tau)$.

The following linear-algebraic calculation applies to both the cubic
and square identities.  For an array
\(X=(X_{\boldsymbol a,\boldsymbol b})\) indexed by
\(\F_3^2\times\F_3^2\), we write
\[
 \mathcal S_{\boldsymbol\lambda,\boldsymbol\kappa}(X)
 =\sum_{\boldsymbol a,\boldsymbol b\in\F_3^2}
 \omega^{\boldsymbol\lambda\cdot\boldsymbol a-
          \boldsymbol\kappa\cdot\boldsymbol b}
 X_{\boldsymbol a,\boldsymbol b}.
\]

\begin{lemm}
\label{r: degree two character reduction}
Suppose that
\[
 X_{\boldsymbol a,\boldsymbol b}=0
 \quad\text{if }\boldsymbol a\cdot\boldsymbol b\ne0,
 \qquad
 X_{\boldsymbol a,\boldsymbol b}
 =X_{-\boldsymbol a,\boldsymbol b}
 =X_{\boldsymbol a,-\boldsymbol b}
 =X_{-\boldsymbol a,-\boldsymbol b}.
\]
Let
\[
\begin{gathered}
 S_1=\mathcal S_{\boldsymbol e_2,\boldsymbol e_2}(X),\quad
 S_2=\mathcal S_{\boldsymbol e_2,\boldsymbol d_+}(X),\quad
 S_3=\mathcal S_{\boldsymbol e_2,\boldsymbol d_-}(X),\quad
 S_4=\mathcal S_{\boldsymbol e_1,\boldsymbol e_1}(X),\\
 S_5=\mathcal S_{\boldsymbol e_1,\boldsymbol d_+}(X),\quad
 S_6=\mathcal S_{\boldsymbol d_+,\boldsymbol e_2}(X),\quad
 S_7=\mathcal S_{\boldsymbol d_-,\boldsymbol e_2}(X).
\end{gathered}
\]
We define \(\boldsymbol R(X)=(R_1,\ldots,R_7)^T\) by
\[
\begin{aligned}
 R_1={}&X_{\boldsymbol0,\boldsymbol0}
 -X_{\boldsymbol0,\boldsymbol d_+}
 -X_{\boldsymbol d_+,\boldsymbol0}
 -X_{\boldsymbol e_2,\boldsymbol e_1}
 -X_{\boldsymbol e_1,\boldsymbol e_2},\\
 R_2={}&(X_{\boldsymbol0,\boldsymbol d_-}
 -X_{\boldsymbol d_+,\boldsymbol0})
 -(X_{\boldsymbol0,\boldsymbol d_+}
 -X_{\boldsymbol d_-,\boldsymbol0}),\\
 R_3={}&(X_{\boldsymbol0,\boldsymbol d_-}
 -X_{\boldsymbol d_+,\boldsymbol0})
 -(X_{\boldsymbol0,\boldsymbol e_1}
 -X_{\boldsymbol e_2,\boldsymbol0}),\\
 R_4={}&(X_{\boldsymbol0,\boldsymbol d_-}
 -X_{\boldsymbol d_+,\boldsymbol0})
 -(X_{\boldsymbol0,\boldsymbol e_2}
 -X_{\boldsymbol e_1,\boldsymbol0}),\\
 R_5={}&(X_{\boldsymbol0,\boldsymbol d_-}
 -X_{\boldsymbol d_+,\boldsymbol d_-})
 -(X_{\boldsymbol0,\boldsymbol d_+}
 -X_{\boldsymbol d_-,\boldsymbol d_+}),\\
 R_6={}&(X_{\boldsymbol0,\boldsymbol d_-}
 -X_{\boldsymbol d_+,\boldsymbol d_-})
 -(X_{\boldsymbol0,\boldsymbol e_1}
 -X_{\boldsymbol e_2,\boldsymbol e_1}),\\
 R_7={}&(X_{\boldsymbol0,\boldsymbol d_-}
 -X_{\boldsymbol d_+,\boldsymbol d_-})
 -(X_{\boldsymbol0,\boldsymbol e_2}
 -X_{\boldsymbol e_1,\boldsymbol e_2}).
\end{aligned}
\]
Then
\begin{equation}
\label{r: degree two reduction matrix}
 3\boldsymbol R(X)=
 \begin{pmatrix}
  0& 1& 0& 1& 0& 1& 0\\
  0& 1&-1& 0& 0& 1&-1\\
  0& 0& 0& 0& 1& 0&-1\\
  1& 0& 0&-1& 1& 0&-1\\
  0& 1&-1& 0& 0& 0& 0\\
 -1& 1& 0& 0& 0& 0& 0\\
  0& 0& 0&-1& 1& 0& 0
 \end{pmatrix}
 \begin{pmatrix}S_1\\S_2\\S_3\\S_4\\S_5\\S_6\\S_7\end{pmatrix}.
\end{equation}
Moreover, the seven components of \(\boldsymbol R(X)\) are linearly
independent as linear forms in the thirteen entries
\(X_{\boldsymbol a,\boldsymbol b}\) indexed by
\((\boldsymbol a,\boldsymbol b)\in\mathcal I\).
\end{lemm}

\begin{proof}
In each character sum \(\mathcal S_{\boldsymbol\lambda,\boldsymbol\kappa}(X)\),
the vanishing condition removes the pairs \((\boldsymbol a,\boldsymbol b)\)
with nonzero scalar product.  Grouping the remaining terms into the
sign orbits represented by \(\mathcal I\) gives coefficients determined
by \(1+1=2\) and \(\omega+\omega^{-1}=-1\).  For example,
\begin{align*}
S_1=&X_{\boldsymbol0,\boldsymbol0} + 2X_{\boldsymbol0,\boldsymbol e_1} - X_{\boldsymbol0,\boldsymbol e_2} - X_{\boldsymbol0,\boldsymbol d_+} - X_{\boldsymbol0,\boldsymbol d_-}\\
&+ 2X_{\boldsymbol e_1,\boldsymbol0} - X_{\boldsymbol e_2,\boldsymbol0} - X_{\boldsymbol d_+,\boldsymbol0} - X_{\boldsymbol d_-,\boldsymbol0}\\
&- 2X_{\boldsymbol e_1,\boldsymbol e_2} - 2X_{\boldsymbol e_2,\boldsymbol e_1} + X_{\boldsymbol d_+,\boldsymbol d_-} + X_{\boldsymbol d_-,\boldsymbol d_+}.
\end{align*}
Expanding the seven indicated character
sums as above gives \eqref{r: degree two reduction matrix}.  To prove
independence, consider the minor formed by the columns corresponding to
\[
 X_{\boldsymbol0,\boldsymbol0},\quad
 X_{\boldsymbol0,\boldsymbol e_1},\quad
 X_{\boldsymbol0,\boldsymbol e_2},\quad
 X_{\boldsymbol0,\boldsymbol d_+},\quad
 X_{\boldsymbol e_1,\boldsymbol0},\quad
 X_{\boldsymbol e_2,\boldsymbol0},\quad
 X_{\boldsymbol d_+,\boldsymbol0}
\]
in the coefficient matrix of \(\boldsymbol R(X)\) has determinant
\(1\).
\end{proof}

\begin{thm}
\label{r: degree two cubic theorem}
The cubes of the thirteen theta series
\eqref{r: degree two q series}, indexed by the set
\(\mathcal I\) in \eqref{r: degree two representatives}, satisfy
\begin{equation}
\label{r: degree two main relation}
 \theta_{\boldsymbol0,\boldsymbol0}^3
 =\theta_{\boldsymbol0,\boldsymbol d_+}^3
  +\theta_{\boldsymbol d_+,\boldsymbol0}^3
  +\theta_{\boldsymbol e_2,\boldsymbol e_1}^3
  +\theta_{\boldsymbol e_1,\boldsymbol e_2}^3,
\end{equation}
\begin{equation}
\label{r: degree two first chain}
 \theta_{\boldsymbol0,\boldsymbol d_-}^3-\theta_{\boldsymbol d_+,\boldsymbol0}^3
 =\theta_{\boldsymbol0,\boldsymbol d_+}^3-\theta_{\boldsymbol d_-,\boldsymbol0}^3
 =\theta_{\boldsymbol0,\boldsymbol e_1}^3-\theta_{\boldsymbol e_2,\boldsymbol0}^3
 =\theta_{\boldsymbol0,\boldsymbol e_2}^3-\theta_{\boldsymbol e_1,\boldsymbol0}^3,
\end{equation}
and
\begin{equation}
\label{r: degree two second chain}
 \theta_{\boldsymbol0,\boldsymbol d_-}^3-\theta_{\boldsymbol d_+,\boldsymbol d_-}^3
 =\theta_{\boldsymbol0,\boldsymbol d_+}^3-\theta_{\boldsymbol d_-,\boldsymbol d_+}^3
 =\theta_{\boldsymbol0,\boldsymbol e_1}^3-\theta_{\boldsymbol e_2,\boldsymbol e_1}^3
 =\theta_{\boldsymbol0,\boldsymbol e_2}^3-\theta_{\boldsymbol e_1,\boldsymbol e_2}^3.
\end{equation}
\end{thm}

\begin{proof}
We apply Lemma~\ref{r: degree two character reduction} to
\[
 X_{\boldsymbol a,\boldsymbol b}
 =\theta_{\boldsymbol a,\boldsymbol b}(\tau)^3.
\]
Its hypotheses follow from \eqref{r: cubic vanishing} and
\eqref{r: cubic sign symmetry}.  Corollary~\ref{r: cubic character identity},
with \(n=2\) and \(z=0\), shows that each of the seven character sums
in \eqref{r: degree two reduction matrix} vanishes.  Hence
\(\boldsymbol R(X)=0\), which is precisely
\eqref{r: degree two main relation},
\eqref{r: degree two first chain}, and
\eqref{r: degree two second chain}.
\end{proof}

\section{Square companions}
\label{r: square companions section}
\subsection{The square Fourier identity}

We retain the notation of Section~\ref{r: cubic theta section}.
The following square counterpart of
Theorem~\ref{r: cubic Fourier theorem} is a direct specialization of
Theorem~\ref{thm:main}.
\begin{thm}
\label{thm:square-counterpart}
Let \(n\geq1\), \(\tau\in\mathfrak H_n\),
\(z\in\numC^{2,n}\), and \(C=\tbtmat{1}{1}{-1}{2}\).
For every \(\mathbf p,\mathbf q\in\F_3^n\),
\begin{equation}
\label{r: square Fourier formula}
3^n\Theta^{(n)}_{\mathbf p,\mathbf q}(\tau,z)
\Theta^{(n)}_{\mathbf p,-\mathbf q}(2\tau,2z)
=\sum_{\mathbf u,\mathbf v\in\F_3^n}
\omega^{\mathbf p\cdot\mathbf v+\mathbf q\cdot\mathbf u}
\Theta^{(n)}_{\mathbf u,\mathbf v}(\tau,Cz)
\theta^{(n)}_{\mathbf u,\mathbf v}(2\tau).
\end{equation}
At \(z=0\), the second theta factor on the left-hand side may be written with
\(\mathbf q\) in place of \(-\mathbf q\).
\end{thm}

\begin{proof}
We apply Theorem~\ref{thm:main} with \(m=4\), \(\mu=1\),
\(M_1=M_2=\widetilde M=\operatorname{diag}(M,2M)\), and
\begin{equation*}
 A=\frac13
 \begin{pmatrix}
  2&-1&-4& 2\\
  1& 1&-2&-2\\
  2&-1& 2&-1\\
  1& 1& 1& 1
 \end{pmatrix},\qquad
 \widetilde{A}=\begin{pmatrix}
  1&0&0&-2\\
  0&1&0&-1\\
  0&0&1&-2\\
  0&0&0& 3
 \end{pmatrix}.
\end{equation*}
Direct multiplication gives
\[
A^T\widetilde M A=\widetilde M,\qquad
\det A=1,\qquad\det \widetilde{A}=3.
\]
All columns of \(3A\) are congruent to \(u_0=(2,1,2,1)^T\) modulo
\(3\numZ^4\).  The first three columns of \(\widetilde{A}^{-1}\)
are \(e_1,e_2,e_3\), and its last column is \(u_0/3\).  Since
\(e_4=u_0-2e_1-e_2-2e_3\), it follows that
\[
 \numZ^{4,n}+A\numZ^{4,n}=\widetilde{A}^{-1}\numZ^{4,n}.
\]
We choose the characteristics and Jacobi variable as
\[
a_1=[\rho,\rho]\mathbf p,\qquad
b_1=[\sigma,\sigma]\mathbf q,\qquad
z_1=[z,z].
\]
More explicitly,
\begin{align}
\label{eq:LHS41-becomeSquare}
\text{LHS of }\eqref{eq:mainIdentity}=&\abs{\det \widetilde{A}}^n\cdot\prod_{i=1}^{\mu}\Theta_{M_1,n}[a_i,b_i](\tau,z_i)\\
=&3^n\Theta_{M,n}[\rho\mathbf p,\sigma\mathbf q](\tau,z)\Theta_{2M,n}[\rho\mathbf p,\sigma\mathbf q](\tau,z)\notag\\
=&3^n\Theta^{(n)}_{\mathbf p,\mathbf q}(\tau,z)\Theta^{(n)}_{\mathbf p,-\mathbf q}(2\tau,2z),\notag
\end{align}
where the equality
$\Theta_{2M,n}[\rho\mathbf p,\sigma\mathbf q](\tau,z)=\Theta^{(n)}_{\mathbf p,-\mathbf q}(2\tau,2z)$
follows from Lemma~\ref{lemm:twochar}.
Complete systems of representatives for the two quotients are
\[
\alpha_{\mathbf u}:=\begin{pmatrix}\rho\\-\rho\end{pmatrix}\mathbf u,\qquad
\beta_{\mathbf v}:=\begin{pmatrix}\sigma\\-\sigma\end{pmatrix}\mathbf v,
\qquad\mathbf u,\mathbf v\in\F_3^n.
\]
Indeed,
\[
 \widetilde{A}A\begin{pmatrix}\rho\\-\rho\end{pmatrix}=e_1,
 \qquad
 \widetilde M A\begin{pmatrix}\sigma\\-\sigma\end{pmatrix}=e_2.
\]
Multiplication by \(\widetilde M A\) identifies the second quotient
with \(\numZ^{4,n}/\widetilde{A}^T\numZ^{4,n}\).
The two quotient groups in \eqref{eq:alphaset} and
\eqref{eq:betaset} both have order \(3^n\); the displayed classes
are nonzero because
\(\bigl(\rho^T,-\rho^T\bigr)^T\notin\numZ^4\) and
\(\widetilde{A}^{-T}e_2=(0,1,0,1/3)^T\notin\numZ^4\).
Since the quotients and the displayed parametrizations are formed
column by column, the parametrizations are complete for every \(n\).
Moreover,
\[
a_1'=\begin{pmatrix}C\rho\\0\end{pmatrix}\mathbf p,\qquad
b_1'=\begin{pmatrix}C\sigma\\0\end{pmatrix}\mathbf q,\qquad
w_1=\begin{pmatrix}Cz\\0\end{pmatrix}.
\]
Therefore,
\begin{align*}
\text{RHS of }\eqref{eq:mainIdentity}&=\sum_{\mathbf u,\mathbf v\in\F_3^n}\etp{-\tr\left(a_1'^T\widetilde{M}\beta_{\mathbf v}\right)}\Theta_{\widetilde{M},n}[a_1'+\alpha_{\mathbf u},b_1'+\beta_{\mathbf v}](\tau,w_1)\\
&=\sum_{\mathbf u,\mathbf v\in\F_3^n}\omega^{\mathbf p\cdot\mathbf v+\mathbf q\cdot\mathbf u}\Theta^{(n)}_{\mathbf u,\mathbf v}(\tau,Cz)\theta^{(n)}_{\mathbf u,\mathbf v}(2\tau).
\end{align*}
The second equality follows from Lemma~\ref{lemm:twochar},
\eqref{r: cubic characteristic transformation}, and
\eqref{r: cubic sign symmetry}.  Combining this identity with
\eqref{eq:LHS41-becomeSquare} in \eqref{eq:mainIdentity} gives
\eqref{r: square Fourier formula}.  The specialization \(z=0\),
together with \eqref{r: cubic sign symmetry}, proves the final
assertion.
\end{proof}

The finite Fourier transform of
\eqref{r: square Fourier formula} yields the square counterpart of
Corollary~\ref{r: cubic character identity}.  This is the form used
both for the square Schultz identity and for the degree-two
relations below.

\begin{coro}
\label{r: cubic square character identity}
Let \(n\geq1\), \(\tau\in\mathfrak H_n\), and
\(z\in\numC^{2,n}\).  If
\(\boldsymbol\lambda,\boldsymbol\kappa\in\F_3^n\) and
\(\boldsymbol\lambda\cdot\boldsymbol\kappa\ne0\), then
\begin{equation}
\label{r: cubic square character formula}
\sum_{\boldsymbol a,\boldsymbol b\in\F_3^n}
 \omega^{\boldsymbol\lambda\cdot\boldsymbol a-
          \boldsymbol\kappa\cdot\boldsymbol b}
 \Theta^{(n)}_{\boldsymbol a,\boldsymbol b}(\tau,z)
 \Theta^{(n)}_{\boldsymbol a,-\boldsymbol b}(2\tau,2z)=0.
\end{equation}
\end{coro}

\begin{proof}
We multiply \eqref{r: square Fourier formula} by
\(\omega^{\boldsymbol\lambda\cdot\mathbf p-
\boldsymbol\kappa\cdot\mathbf q}\) and sum over
\(\mathbf p,\mathbf q\in\F_3^n\).  The left-hand side of the
resulting equality is
\[
 3^n\sum_{\mathbf p,\mathbf q\in\F_3^n}
 \omega^{\boldsymbol\lambda\cdot\mathbf p-
          \boldsymbol\kappa\cdot\mathbf q}
 \Theta^{(n)}_{\mathbf p,\mathbf q}(\tau,z)
 \Theta^{(n)}_{\mathbf p,-\mathbf q}(2\tau,2z),
\]
whereas its right-hand side is
\[
 \sum_{\mathbf u,\mathbf v\in\F_3^n}
  \Theta^{(n)}_{\mathbf u,\mathbf v}(\tau,Cz)
  \theta^{(n)}_{\mathbf u,\mathbf v}(2\tau)
 \left(\sum_{\mathbf p\in\F_3^n}
 \omega^{\mathbf p\cdot(\boldsymbol\lambda+\mathbf v)}\right)
 \left(\sum_{\mathbf q\in\F_3^n}
 \omega^{\mathbf q\cdot(\mathbf u-\boldsymbol\kappa)}\right).
\]
For \(\mathbf x\in\F_3^n\), character orthogonality gives
\[
 \sum_{\mathbf p\in\F_3^n}\omega^{\mathbf p\cdot\mathbf x}
 =\begin{cases}3^n,&\mathbf x=\boldsymbol0,\\
                 0,&\mathbf x\ne\boldsymbol0.
  \end{cases}
\]
Consequently, only
\((\mathbf u,\mathbf v)=(\boldsymbol\kappa,-\boldsymbol\lambda)\)
survives.  Dividing by \(3^n\), we obtain
\[
 \sum_{\boldsymbol a,\boldsymbol b\in\F_3^n}
 \omega^{\boldsymbol\lambda\cdot\boldsymbol a-
          \boldsymbol\kappa\cdot\boldsymbol b}
 \Theta^{(n)}_{\boldsymbol a,\boldsymbol b}(\tau,z)
 \Theta^{(n)}_{\boldsymbol a,-\boldsymbol b}(2\tau,2z)
  =3^n
 \Theta^{(n)}_{\boldsymbol\kappa,-\boldsymbol\lambda}(\tau,Cz)
 \theta^{(n)}_{\boldsymbol\kappa,-\boldsymbol\lambda}(2\tau).
\]
Since
\(\boldsymbol\kappa\cdot(-\boldsymbol\lambda)
=-\boldsymbol\lambda\cdot\boldsymbol\kappa\ne0\), the theta
constant on the right-hand side vanishes by
\eqref{r: cubic vanishing}, applied with \(2\tau\) in place of
\(\tau\).  This proves
\eqref{r: cubic square character formula}.
\end{proof}

\subsection{The square analogue of Schultz's identity}
We retain the notation of Subsection~\ref{sec:schultz_identity}.
For \(\boldsymbol u=(u_1,u_2)^T\), write
\begin{equation}
\label{r: Schultz star variable}
 \boldsymbol u^*=\begin{pmatrix}u_1-u_2\\-u_2\end{pmatrix}.
\end{equation}

\begin{thm}
\label{r: Schultz square theorem}
Let \(\tau\in\mathfrak H_1\) and \(\boldsymbol u\in\numC^2\).  For
\(r,s\in\F_3\), we have
\begin{equation}
\label{r: Schultz square formula}
\begin{aligned}
&\mathcal T_{r,s}(\boldsymbol u\mid\tau)
 \mathcal T_{r,s}(2\boldsymbol u^*\mid2\tau)
+\mathcal T_{0,0}(\boldsymbol u\mid\tau)
 \mathcal T_{0,0}(2\boldsymbol u^*\mid2\tau)\\
&\qquad=
\mathcal T_{r,0}(\boldsymbol u\mid\tau)
 \mathcal T_{r,0}(2\boldsymbol u^*\mid2\tau)
+\mathcal T_{0,s}(\boldsymbol u\mid\tau)
 \mathcal T_{0,s}(2\boldsymbol u^*\mid2\tau).
\end{aligned}
\end{equation}
\end{thm}
\begin{proof}
For \(r,s\in\F_3\), the change of summation variables
\[
 m'=m,\qquad n'=-m-n-r
\]
in \eqref{r: Schultz q series} gives
\begin{equation}
\label{r: Schultz star transformation}
 \mathcal T_{r,s}(\boldsymbol u^*\mid\tau)
 =\mathcal T_{r,-s}(\boldsymbol u\mid\tau).
\end{equation}
Indeed, if \(a=m+r/3\) and \(b=n+r/3\), then
\(n'+r/3=-a-b\), and hence
\[
 Q(a,b)=Q(a,-a-b),\qquad
 a(u_1-u_2)-bu_2=au_1+(-a-b)u_2.
\]
Moreover,
\[
 \omega^{rs+s(m-n)}
 =\omega^{-rs-s(m'-n')},
\]
because
\[
 [rs+s(m-n)]-[-rs-s(m'-n')]=3s(m+r)\in3\numZ.
\]
This is exactly the
character factor in \(\mathcal T_{r,-s}\).

In \eqref{r: square Fourier formula}, take \(n=1\) and
\(z=z(\boldsymbol u)\).  Adding the instances indexed by
\((r,s)\) and \((0,0)\) and subtracting those indexed by \((r,0)\)
and \((0,s)\) gives the right-hand side
\[
 \sum_{a,b\in\F_3}(1-\omega^{sa})(1-\omega^{rb})
 \Theta^{(1)}_{a,b}(\tau,Cz(\boldsymbol u))
 \theta^{(1)}_{a,b}(2\tau).
\]
Every summand vanishes: its coefficient is zero if \(sa=0\) or
\(rb=0\); otherwise \(a,b\ne0\), and
\eqref{r: cubic vanishing} applies.

The normalization in \eqref{r: Schultz q series} cancels between
the two factors on the left-hand side of
\eqref{r: square Fourier formula}.  Thus the instance indexed by
\((p,q)\) has left-hand side
\[
 3\mathcal T_{p,q}(\boldsymbol u\mid\tau)
   \mathcal T_{p,-q}(2\boldsymbol u\mid2\tau).
\]
By \eqref{r: Schultz star transformation}, this equals
\[
 3\mathcal T_{p,q}(\boldsymbol u\mid\tau)
   \mathcal T_{p,q}(2\boldsymbol u^*\mid2\tau).
\]
The indicated linear combination is therefore three times the
difference between the two sides of
\eqref{r: Schultz square formula}, which proves the result.
\end{proof}

\begin{rema}
Taking \(r=s=1\) and using the correspondence in the remark after
Theorem~\ref{r: Schultz theorem}, formula
\eqref{r: Schultz square formula} becomes exactly the
single-subscript identity \eqref{eq:intro Schultz square identity} in
Schultz's notation~\cite[pp.~621--622]{Sch13}.
\end{rema}

Borwein, Borwein, and Garvan~\cite[Theorem~2.6(i)]{BBG94} proved the
quadratic modular equation obtained by taking \(r=s=1\) and
\(\boldsymbol u=0\) in Theorem~\ref{r: Schultz square theorem}.
The same authors~\cite[Theorem~2.6(ii)]{BBG94} stated the equivalent
form obtained by combining this equation with \eqref{r: Borwein formula}.
\begin{coro}
For \(\abs{q}<1\), we have
\begin{equation}
\label{r: Ramanujan square formula}
 a(q)a(q^2)=b(q)b(q^2)+c(q)c(q^2).
\end{equation}
Moreover, if \(a(q)a(q^2)\ne0\), then, with
\(\lambda=c(q)/a(q)\) and \(\nu=c(q^2)/a(q^2)\), we have
\begin{equation*}
 (1-\lambda\nu)^3=(1-\lambda^3)(1-\nu^3).
\end{equation*}
\end{coro}
\begin{proof}
For \(r=s=1\) and \(\boldsymbol u=0\),
\(\boldsymbol u^*=0\), and the term involving
\(\mathcal T_{1,1}\) vanishes by the identification preceding
Theorem~\ref{r: Schultz theorem} and
\eqref{r: cubic vanishing}.  Formula
\eqref{r: Ramanujan square formula} follows from this fact,
\eqref{r: Borwein series}, and
\eqref{r: Schultz square formula}.  The second conclusion follows from
\eqref{r: Borwein formula} and \eqref{r: Ramanujan square formula}.
\end{proof}

\subsection{Degree two}

We retain the notation of
Subsection~\ref{r: cubic degree two subsection}.  Define
\[
 X_{\boldsymbol a,\boldsymbol b}
 =
 \theta^{(2)}_{\boldsymbol a,\boldsymbol b}(\tau)
 \theta^{(2)}_{\boldsymbol a,\boldsymbol b}(2\tau).
\]
Specializing Corollary~\ref{r: cubic square character identity} to
\(n=2\) and \(z=0\), and using
\eqref{r: cubic sign symmetry} in the second factor, gives
\begin{equation}
\label{r: degree two square character formula}
 \mathcal S_{\boldsymbol\lambda,\boldsymbol\kappa}(X)=0
 \qquad
 (\boldsymbol\lambda\cdot\boldsymbol\kappa\ne0).
\end{equation}
As in the preceding degree-two subsection, we suppress the superscript
and write $\theta_{\boldsymbol a,\boldsymbol b}(\tau)=\theta^{(2)}_{\boldsymbol a,\boldsymbol b}(\tau)$.

\begin{thm}
\label{r: degree two square theorem}
The seven relations
\eqref{r: degree two main relation},
\eqref{r: degree two first chain}, and
\eqref{r: degree two second chain}
remain valid after the uniform replacement
\begin{equation}
\label{r: cube square replacement}
\theta_{\boldsymbol a,\boldsymbol b}(\tau)^3
\longmapsto
\theta_{\boldsymbol a,\boldsymbol b}(\tau)
\theta_{\boldsymbol a,\boldsymbol b}(2\tau).
\end{equation}
\end{thm}

\begin{proof}
The array \(X\) satisfies the hypotheses of
Lemma~\ref{r: degree two character reduction} by
\eqref{r: cubic vanishing} and \eqref{r: cubic sign symmetry}.
Equation~\eqref{r: degree two square character formula} makes the
seven character sums in \eqref{r: degree two reduction matrix}
vanish.  Hence \(\boldsymbol R(X)=0\), which gives the stated
relations after the replacement \eqref{r: cube square replacement}.
\end{proof}

\section{Theta identities of arbitrary order}
\label{r: arbitrary order section}

Corollary~\ref{r: cubic character identity} gives a cubic character
relation on the root lattice \(A_2\).  We establish its same-order
analogue on \(A_{\ell-1}\) for every \(\ell\geq3\).  Here, a same-order
identity on \(A_{\ell-1}\) is a homogeneous relation among
\(\ell\)-th powers of theta series attached to that lattice.

\subsection{The general same-order identity}

The argument rests on a finite Fourier relation associated with a
Coxeter transformation of \(A_{\ell-1}\).

Fix \(\ell\geq2\) and write \(d=\ell-1\).  We use
\begin{equation}
\label{r: p data}
 M_\ell=
 \begin{pmatrix}
  2&-1&&&\\
  -1&2&-1&&\\
  &\ddots&\ddots&\ddots&\\
  &&-1&2&-1\\
  &&&-1&2
 \end{pmatrix}_{d\times d},
 \qquad
 \boldsymbol r=\begin{pmatrix}1&2&\cdots&d\end{pmatrix}^{T},
 \qquad
 \xi=\etp{1/\ell}.
\end{equation}
Write \(\boldsymbol1=(1,\ldots,1)^T\), with $d$ entries, and let \(L\) be the lower
triangular matrix whose entries on and below the diagonal are all equal to one.
We also use
\begin{equation}
\label{r: Coxeter matrix}
 R=
 \begin{pmatrix}
  0&0&\cdots&0&-1\\
  1&0&\cdots&0&-1\\
  0&1&\cdots&0&-1\\
  \vdots&&\ddots&&\vdots\\
  0&0&\cdots&1&-1
 \end{pmatrix}_{d\times d}.
\end{equation}
(If $\ell=2$, then $R=(-1)$.)
The following identities will be used repeatedly:
\begin{equation}
\label{r: Coxeter identities}
 \begin{gathered}
 L\boldsymbol1=\boldsymbol r,
 \qquad R^TM_\ell R=M_\ell,
 \qquad \lvert\det(R-I_d)\rvert=\ell,\\
 (R^{-1}-I_d)M_\ell^{-1}=-L,
 \qquad I_d-M_\ell L=R^T,
 \qquad M_\ell\boldsymbol r=\ell e_d,
 \qquad R^{-1}\boldsymbol r=\boldsymbol r-\ell e_d,\\
 \boldsymbol1^T(R^{-1}-I_d)=-\ell e_1^T,
 \qquad\boldsymbol1^T(R-I_d)=-\ell e_d^T.
 \end{gathered}
\end{equation}
They follow by direct multiplication.  In particular, the
integral matrix
\begin{equation}
\label{r: auxiliary integral matrix}
 E=\ell(R-I_d)^{-1},
 \qquad
 E_{ij}=
 \begin{cases}
  i-\ell,&j\leq i,\\
  i,&j>i,
 \end{cases}
\end{equation}
satisfies (we identify $\numZ^d$ with $\numZ^{d,1}$)
\begin{equation}
\label{r: auxiliary lattice identities}
 (R-I_d)\numZ^d
 =\{x\in\numZ^d:\boldsymbol1^Tx\equiv0\pmod\ell\},
 \qquad
 E^T\numZ^d=\ell\numZ^d+\numZ\boldsymbol1.
\end{equation}
Indeed,
\(\boldsymbol1^T(R-I_d)=-\ell e_d^T\) gives the inclusion in the
first equality of \eqref{r: auxiliary lattice identities}.  Both
lattices have index \(\ell\) in
\(\numZ^d\): this follows respectively from
\(\lvert\det(R-I_d)\rvert=\ell\) and from the surjectivity of
\(x\mapsto\boldsymbol1^Tx\pmod\ell\).  For the second equality,
the explicit formula in \eqref{r: auxiliary integral matrix}
shows that every row of \(E\) belongs to
\(\ell\numZ^d+\numZ\boldsymbol1\).  The two lattices again have
the same index, since
\[
 |\det E|=\ell^{d-1}
 \quad\text{and}\quad
 [\numZ^d:\ell\numZ^d+\numZ\boldsymbol1]=\ell^{d-1}.
\]

For \(n\geq1\), we define
\begin{equation}
\label{r: finite groups}
 \mathcal X_n=
 \{x\in(\numZ/\ell\numZ)^{d,n}:\boldsymbol1^Tx=0\},
 \qquad
 \mathcal Y_n=
 \frac{(\numZ/\ell\numZ)^{d,n}}
 {\{\boldsymbol1\boldsymbol\eta:
       \boldsymbol\eta\in(\numZ/\ell\numZ)^{1,n}\}}.
\end{equation}
The denominator in \(\mathcal Y_n\) consists of the matrices whose
rows are all equal.  The groups \(\mathcal X_n\) and \(\mathcal Y_n\)
are finite abelian groups, each of order
\(\ell^{(d-1)n}\).  Define
\begin{equation}
\label{r: finite pairing}
 \langle x,s\rangle=\xi^{\tr(x^Ts)}
\end{equation}
for $(x,s)\in\mathcal X_n\times\mathcal Y_n$.
\begin{lemm}
\label{lemm:XnYn}
The formula \eqref{r: finite pairing} gives a well-defined bicharacter on $\mathcal X_n\times\mathcal Y_n$. Moreover, if $x\neq0$ (which can occur only if $\ell\geq3$), then there exists $s\in\mathcal Y_n$ such that $\langle x,s\rangle\neq1$.
\end{lemm}
\begin{proof}
For $x\in\mathcal X_n$ and $y\in(\numZ/\ell\numZ)^{d,n}$, define $\langle x,y\rangle_0=\xi^{\tr(x^Ty)}$; this is a well-defined bicharacter. For $s=y+\{\boldsymbol1\boldsymbol\eta:\boldsymbol\eta\in(\numZ/\ell\numZ)^{1,n}\}\in\mathcal Y_n$, the definition gives $\langle x,s\rangle=\langle x,y\rangle_0$. If $y'$ is another representative of $s$, then $y-y'=\boldsymbol1\boldsymbol\eta$ for some $\boldsymbol\eta\in(\numZ/\ell\numZ)^{1,n}$. Hence
\begin{equation*}
\langle x,y-y'\rangle_0=\xi^{\tr(x^T(y-y'))}=\xi^{\tr(x^T(\boldsymbol1\boldsymbol\eta))}=\xi^{\tr((x^T\boldsymbol1)\cdot\boldsymbol\eta)}=1,
\end{equation*}
since \(\boldsymbol1^Tx=0\) in
\((\numZ/\ell\numZ)^{1,n}\) by the definition of \(\mathcal X_n\).
Thus $\langle x,s\rangle$ is independent of the representative of $s$, so \eqref{r: finite pairing} is well defined on $\mathcal X_n\times\mathcal Y_n$. For the second assertion, choose $y\in(\numZ/\ell\numZ)^{d,n}$ such that $\tr(x^Ty)\not\equiv0\pmod{\ell}$; such a choice is possible because $x\neq0$. Then $\langle x,y\rangle_0\neq1$, and hence $\langle x,s\rangle\neq1$ for $s=y+\{\boldsymbol1\boldsymbol\eta:\boldsymbol\eta\in(\numZ/\ell\numZ)^{1,n}\}$.
\end{proof}


We need the following character projection relation.
\begin{lemm}
\label{r: character orthogonality}
Suppose that a family \(F_{x,s}\) is indexed by
\(x\in(\numZ/\ell\numZ)^{d,n}\) and
\(s\in\mathcal Y_n\), with \(\mathcal X_n\) and \(\mathcal Y_n\) as in
\eqref{r: finite groups}, and satisfies
\begin{equation}
\label{r: abstract Fourier relation}
 \ell^{(d-1)n}F_{R^{-1}x,t}
 =\sum_{\substack{y\in x+\mathcal X_n\\s\in\mathcal Y_n}}
 \langle y-x,s-t\rangle F_{y,s},
\end{equation}
for every \(x\in(\numZ/\ell\numZ)^{d,n}\) and
\(t\in\mathcal Y_n\).  Let
\(\boldsymbol\lambda,\boldsymbol\kappa
\in(\numZ/\ell\numZ)^{1,n}\) satisfy\footnote{If \eqref{r: admissible characters} holds, then necessarily $\ell\geq3$.}
\begin{equation}
\label{r: admissible characters}
 \boldsymbol r\boldsymbol\kappa\in\mathcal X_n,
 \qquad
 \boldsymbol\lambda\boldsymbol\kappa^T\neq0.
\end{equation}
Then
\begin{equation}
\label{r: character sum}
 \sum_{\substack{\boldsymbol a\in(\numZ/\ell\numZ)^{1,n}\\
                  t\in\mathcal Y_n}}
 \xi^{\boldsymbol\lambda\boldsymbol a^T}
 \langle\boldsymbol r\boldsymbol\kappa,t\rangle
 F_{\boldsymbol r\boldsymbol a,t}=0.
\end{equation}
\end{lemm}

\begin{proof}
Since \(R^{-1}\boldsymbol r=\boldsymbol r-\ell e_d\), we have
\(F_{R^{-1}\boldsymbol r\boldsymbol a,t}
=F_{\boldsymbol r\boldsymbol a,t}\).  We multiply
\eqref{r: abstract Fourier relation}, with
\(x=\boldsymbol r\boldsymbol a\), by
\(\xi^{\boldsymbol\lambda\boldsymbol a^T}
\langle\boldsymbol r\boldsymbol\kappa,t\rangle\) and sum over
\(\boldsymbol a\) and \(t\).  The second conclusion of Lemma~\ref{lemm:XnYn} and orthogonality on \(\mathcal Y_n\)
give
\[
 \sum_{t\in\mathcal Y_n}
 \langle\boldsymbol r\boldsymbol\kappa,t\rangle\langle y-\boldsymbol r\boldsymbol a,-t\rangle=
 \sum_{t\in\mathcal Y_n}
 \langle\boldsymbol r\boldsymbol\kappa-
 (y-\boldsymbol r\boldsymbol a),t\rangle
 =\begin{cases}
   \ell^{(d-1)n},&
   y-\boldsymbol r\boldsymbol a
      =\boldsymbol r\boldsymbol\kappa,\\
   0,&\text{otherwise}.
  \end{cases}
\]
The value \(y=\boldsymbol r(\boldsymbol a+\boldsymbol\kappa)\)
lies in \(\boldsymbol r\boldsymbol a+\mathcal X_n\) by
\eqref{r: admissible characters}.  Writing \(S\) for the sum in
\eqref{r: character sum}, the left-hand side of the summed
Fourier relation is \(\ell^{(d-1)n}S\), whereas its right-hand side
is
\[
 \ell^{(d-1)n}
 \sum_{\boldsymbol a,s}
 \xi^{\boldsymbol\lambda\boldsymbol a^T}
 \langle\boldsymbol r\boldsymbol\kappa,s\rangle
 F_{\boldsymbol r(\boldsymbol a+\boldsymbol\kappa),s}
 =\ell^{(d-1)n}
  \xi^{-\boldsymbol\lambda\boldsymbol\kappa^T}S.
\]
Since
\(\xi^{-\boldsymbol\lambda\boldsymbol\kappa^T}\ne1\), it
follows that \(S=0\).
\end{proof}

Since \(\boldsymbol1^T\boldsymbol r=\ell(\ell-1)/2\), the first
condition in \eqref{r: admissible characters} is automatic when
\(\ell\) is odd; when \(\ell\) is even, it requires every component
of \(\boldsymbol\kappa\) to lie in \(2\numZ/\ell\numZ\).  In the latter case, replacing
\(\boldsymbol\lambda\) by
\(\boldsymbol\lambda+(\ell/2)\boldsymbol\varepsilon\), with
\(\boldsymbol\varepsilon\in\{0,1\}^{1,n}\), preserves the second
condition in \eqref{r: admissible characters}.  This observation
gives the parity refinement in the following theorem, which is the main result of this section.
We identify $(\numZ/\ell\numZ)^{1,n}$ with $(\numZ/\ell\numZ)^n$,
so $\boldsymbol\lambda\cdot\boldsymbol\kappa=\boldsymbol\lambda\boldsymbol\kappa^T$.

\begin{thm}
\label{r: unified p power identity}
Let \(\ell\geq3\), \(n\geq1\), \(\tau\in\mathfrak H_n\), and
\(z\in\numC^{d,n}\), where $d=\ell-1$.  For
\(\boldsymbol\lambda,\boldsymbol\kappa
\in(\numZ/\ell\numZ)^n\), assume that
\(\boldsymbol\lambda\cdot\boldsymbol\kappa\not\equiv0
\pmod\ell\).  If \(\ell\) is even, assume in addition that every
component of \(\boldsymbol\kappa\) lies in
\(2\numZ/\ell\numZ\).  Then
\begin{equation}
\label{r: unified p power formula}
 \sum_{\substack{\boldsymbol a,\boldsymbol b_1,\ldots,
                  \boldsymbol b_{d-1}
                  \in(\numZ/\ell\numZ)^n}}
 \xi^{\boldsymbol\lambda\cdot\boldsymbol a+
       \boldsymbol\kappa\cdot
       (\boldsymbol b_1+\cdots+\boldsymbol b_{d-1})}
 \Theta_{M_\ell,n}^{\,\ell}
 \!\left[\frac{\boldsymbol r\boldsymbol a}{\ell},
          \frac{B}{\ell}\right](\tau,z)=0.
\end{equation}
Here
\(B=[\boldsymbol b_1,\ldots,\boldsymbol b_{d-1},\boldsymbol0]\).
Moreover, if \(\ell\) is even, then for every $\boldsymbol a_0\in(\numZ/\ell\numZ)^n$ we have
\begin{equation}
\label{r: unified p power formula 2}
 \sum_{\substack{{\boldsymbol b_1,\ldots,
                  \boldsymbol b_{d-1}
                  \in(\numZ/\ell\numZ)^n}\\
                   {\boldsymbol a\in\boldsymbol a_0+(2\numZ/\ell\numZ)^n}}}
 \xi^{\boldsymbol\lambda\cdot\boldsymbol a+
       \boldsymbol\kappa\cdot
       (\boldsymbol b_1+\cdots+\boldsymbol b_{d-1})}
 \Theta_{M_\ell,n}^{\,\ell}
 \!\left[\frac{\boldsymbol r\boldsymbol a}{\ell},
          \frac{B}{\ell}\right](\tau,z)=0.
\end{equation}
\end{thm}

\begin{proof}
Throughout the proof, we use the notation introduced at the beginning
of this subsection and the properties in \eqref{r: Coxeter identities}.
For \(x\in(\numZ/\ell\numZ)^{d,n}\) and
\(s\in\mathcal Y_n\), we choose integral lifts and define
\[
 \gamma_s=-\frac{Ls}{\ell},
 \qquad
 F_{x,s}=\Theta_{M_\ell,n}^{\,\ell}
 \!\left[\frac{x}{\ell},\gamma_s\right](\tau,z).
\]
Changing an integral lift of \(x\) by \(\ell h\) shifts the first
characteristic by \(h\) and does not change the theta series.  A
change \(s\mapsto s+\ell k\) changes \(\gamma_s\) by \(-Lk\);
by Lemma~\ref{lemm:twochar}, the resulting factor before taking the
\(\ell\)-th power is
\[
 \xi^{-\tr(x^TM_\ell Lk)}.
\]
Replacing \(s\) by
\(s+\boldsymbol1\boldsymbol\eta\) changes \(\gamma_s\) by
\(-\boldsymbol r\boldsymbol\eta/\ell
=M_\ell^{-1}(-e_d\boldsymbol\eta)\), and the corresponding factor
is
\(\xi^{-\tr(x^Te_d\boldsymbol\eta)}\).  Both factors disappear
after taking the \(\ell\)-th power.  Thus \(F_{x,s}\) is well
defined (independent of the integral lifts chosen for \(x\) and \(s\)).

We apply Theorem~\ref{thm:main} with
\[
 m=d,\qquad \mu=\ell,\qquad M_1=M_2=M_\ell,\qquad
 A=I_{d\ell}+\frac1\ell J_\ell\otimes(R-I_d),
\]
where \(J_\ell\) is the \(\ell\times\ell\) all-ones matrix.  A direct calculation shows that
\(A^TM_\ell^{(\ell)}A=M_\ell^{(\ell)}\).
The integral matrix \(\widetilde{A}\) is defined by
\[
 \widetilde{A}\cdot(x_1,\ldots,x_\ell)^T
 =(x_1-x_\ell,\ldots,x_{\ell-1}-x_\ell,x_\ell E^T)^T,\qquad
 \text{for all } x_1,\dots,x_\ell\in\numZ^{1,d}.
\]
Since \(E^{-1}=(R-I_d)/\ell\), reduction modulo
\(\numZ^{d\ell,n}\) gives
\begin{equation*}
 \widetilde{A}^{-1}\numZ^{d\ell,n}
 =\numZ^{d\ell,n}
 +\left\{\left[\frac{(R-I_d)u}{\ell},\ldots,
                    \frac{(R-I_d)u}{\ell}\right]:
              u\in\numZ^{d,n}\right\}
 =\numZ^{d\ell,n}+A\numZ^{d\ell,n}.
\end{equation*}
All the hypotheses of Theorem~\ref{thm:main} are therefore satisfied, and
\(|\det \widetilde{A}|=|\det E|=\ell^{d-1}\).

The two quotient groups in Theorem~\ref{thm:main} are parametrized
by \(\mathcal X_n\) and \(\mathcal Y_n\), respectively.  More
precisely, the first parametrization is the isomorphism
\[
 \mathcal X_n\longrightarrow
 \frac{(\widetilde{A}A)^{-1}\numZ^{d\ell,n}}{\numZ^{d\ell,n}},
 \qquad
 v\longmapsto
 \left[\frac v\ell,\ldots,\frac v\ell\right]
       +\numZ^{d\ell,n}.
\]
Changing the lift of \(v\) changes the displayed representative by
an integral matrix.  Moreover,
\[
 \widetilde{A}A\left[\frac v\ell,\ldots,\frac v\ell\right]
 =\widetilde{A}\left[\frac{Rv}{\ell},\ldots,\frac{Rv}{\ell}\right]
 =\left[0,\ldots,0,(R-I_d)^{-1}Rv\right],
\]
which is integral by the first equality in
\eqref{r: auxiliary lattice identities}.  Thus the map is well defined and is a homomorphism.
Its kernel is trivial, and the target has order
\(\abs{\det(\widetilde{A}A)}^n=\ell^{(d-1)n}=\abs{\mathcal X_n}\); hence the map is an isomorphism.

For the second quotient, define
\[
 \beta(u)=[\gamma_u,\ldots,\gamma_u,
             \gamma_u+M_\ell^{-1}u].
\]
Then
\[
 M_\ell^{(\ell)}A\beta(u)=[0,\ldots,0,u].
\]
Since
\([0,\ldots,0,u]\in \widetilde{A}^T\numZ^{d\ell,n}\) if and only if
\(u\in E^T\numZ^{d,n}\), the second equality in
\eqref{r: auxiliary lattice identities} shows that the class of
\(\beta(u)\) depends only on the class of \(u\) in \(\mathcal Y_n\)
and that
\[
 \mathcal Y_n\longrightarrow
 \frac{A^{-1}{M_\ell^{(\ell)}}^{-1}\numZ^{d\ell,n}}
 {A^{-1}{M_\ell^{(\ell)}}^{-1}\widetilde{A}^T\numZ^{d\ell,n}},
 \qquad s\longmapsto\beta(s)+A^{-1}{M_\ell^{(\ell)}}^{-1}\widetilde{A}^T\numZ^{d\ell,n}
\]
is injective.  Its target has order
\([\numZ^{d\ell,n}:\widetilde{A}^T\numZ^{d\ell,n}]
=\ell^{(d-1)n}=\abs{\mathcal Y_n}\), so this map is an isomorphism.

For fixed \(x\in(\numZ/\ell\numZ)^{d,n}\) and
\(t\in\mathcal Y_n\), we specialize Theorem~\ref{thm:main} by
taking
\[
 a_i=\frac{x}{\ell},\qquad
 b_i=R\gamma_t,\qquad z_i=Rz
 \qquad(1\leq i\leq\ell).
\]
Then
\(a_i'=R^{-1}x/\ell\), \(b_i'=\gamma_t\), and \(w_i=z\).
The identity
\[
 \Theta_{M_\ell,n}[u,Rv](\tau,Rz)
 =\Theta_{M_\ell,n}[R^{-1}u,v](\tau,z)
\]
shows that the product on the left-hand side of
\eqref{eq:mainIdentity} is \(F_{R^{-1}x,t}\).
Writing \(y=R^{-1}x+v\) and \(s=t+u\), the preceding
isomorphisms give the following representatives for the two
quotients in \eqref{eq:alphaset} and \eqref{eq:betaset}:
\[
 \left[\frac{y-R^{-1}x}{\ell},\ldots,
       \frac{y-R^{-1}x}{\ell}\right],
 \qquad y\in R^{-1}x+\mathcal X_n,
\]
and
\[
 [\gamma_u,\ldots,\gamma_u,
   \gamma_u+M_\ell^{-1}u],
 \qquad u\in\mathcal Y_n.
\]

The identity \(I_d-M_\ell L=R^T\) gives
\[
 M_\ell(\ell\gamma_u+M_\ell^{-1}u)=R^Tu.
\]
The exterior factor on the right-hand side of \eqref{eq:mainIdentity} is therefore
\[
\etp{-\tr\left(\sum_{i=1}^{\ell}a_i'^TM_\ell\beta_i\right)}=\xi^{-\tr(x^Tu)},
\]
whereas removing \(M_\ell^{-1}u\) from the
second characteristic in the last theta factor contributes
\(\xi^{\tr(y^Tu)}\).
For each pair \((y,s)\), the corresponding summand
on the right-hand side of \eqref{eq:mainIdentity} is therefore
\[
 \xi^{-\tr(x^Tu)}
 \xi^{\tr(y^Tu)}
 \Theta_{M_\ell,n}^{\,\ell}
 \!\left[\frac y\ell,\gamma_s\right](\tau,z)
 =\langle y-x,s-t\rangle F_{y,s}.
\]
Consequently, Theorem~\ref{thm:main} gives
\[
 \ell^{(d-1)n}F_{R^{-1}x,t}
 =\sum_{\substack{y\in R^{-1}x+\mathcal X_n\\s\in\mathcal Y_n}}
 \langle y-x,s-t\rangle F_{y,s}.
\]
Since
\(\boldsymbol1^T(R^{-1}-I_d)=-\ell e_1^T\), we have
\((R^{-1}-I_d)x\in\mathcal X_n\); hence the affine set in the last
sum is \(x+\mathcal X_n\).  We have therefore proved
\begin{equation}
\label{r: common finite Fourier relation}
 \ell^{(d-1)n}F_{R^{-1}x,t}
 =\sum_{\substack{y\in x+\mathcal X_n\\s\in\mathcal Y_n}}
 \langle y-x,s-t\rangle F_{y,s},
\end{equation}
which is exactly the relation required in \eqref{r: abstract Fourier relation}.
The two conditions in \eqref{r: admissible characters}
are also satisfied: the first by the argument in the paragraph preceding Theorem~\ref{r: unified p power identity},
and the second by assumption.  Lemma~\ref{r: character orthogonality} therefore yields \eqref{r: character sum} for
$F_{x,s}=\Theta_{M_\ell,n}^{\ell}\left[\frac{x}{\ell},\gamma_s\right](\tau,z)$.

It remains to express the character projection in the stated
coordinates.  Replacing \(s\) by
\(s+\boldsymbol1\boldsymbol\eta\) changes \(B=-Ls\) by
\(-\boldsymbol r\boldsymbol\eta\), and hence changes its last row by
\(-d\boldsymbol\eta\).  Since \(d\equiv-1\pmod\ell\), every class
\(s\in\mathcal Y_n\) has a unique representative for which the last
row of \(B\) is zero.  With this representative,
\begin{equation}
\label{r: B representative}
 \gamma_s=\frac B\ell,
 \qquad
 \boldsymbol r^Ts=-\boldsymbol r^TL^{-1}B
 \equiv\boldsymbol1^TB\pmod\ell.
\end{equation}
The result now follows by combining
\eqref{r: B representative} with \eqref{r: character sum}.

When \(\ell\) is even, applying the formula with
\(\boldsymbol\lambda+(\ell/2)\boldsymbol\varepsilon\) for every
\(\boldsymbol\varepsilon\in\{0,1\}^n\) is permissible because the
components of \(\boldsymbol\kappa\) are even.  Denote the sum in \eqref{r: unified p power formula} by $S_{\boldsymbol\lambda,\boldsymbol\kappa}$.
Then $S_{\boldsymbol\lambda+(\ell/2)\boldsymbol\varepsilon,\boldsymbol\kappa}=0$ for \(\boldsymbol\varepsilon\in\{0,1\}^n\).
It follows that
\begin{align*}
0=&\sum_{\boldsymbol\varepsilon\in(\numZ/2\numZ)^n}(-1)^{-\boldsymbol a_0\cdot\boldsymbol\varepsilon}S_{\boldsymbol\lambda+(\ell/2)\boldsymbol\varepsilon,\boldsymbol\kappa}\\
=&\sum_{\substack{{\boldsymbol\varepsilon\in(\numZ/2\numZ)^n}\\{\boldsymbol a\in(\numZ/\ell\numZ)^n}}}(-1)^{(\boldsymbol a-\boldsymbol a_0)\cdot\boldsymbol\varepsilon}\cdot\sum_{\substack{{\boldsymbol b_1,\ldots,
                  \boldsymbol b_{d-1}
                  \in(\numZ/\ell\numZ)^n}\\
                   }}
 \xi^{\boldsymbol\lambda\cdot\boldsymbol a+
       \boldsymbol\kappa\cdot
       (\boldsymbol b_1+\cdots+\boldsymbol b_{d-1})}
 \Theta_{M_\ell,n}^{\,\ell}
 \!\left[\frac{\boldsymbol r\boldsymbol a}{\ell},
          \frac{B}{\ell}\right](\tau,z)\\
  &=2^n\cdot\sum_{\substack{{\boldsymbol a-\boldsymbol a_0\in(2\numZ/\ell\numZ)^n}}}\sum_{\substack{{\boldsymbol b_1,\ldots,
                  \boldsymbol b_{d-1}
                  \in(\numZ/\ell\numZ)^n}\\
                   }}
 \xi^{\boldsymbol\lambda\cdot\boldsymbol a+
       \boldsymbol\kappa\cdot
       (\boldsymbol b_1+\cdots+\boldsymbol b_{d-1})}
 \Theta_{M_\ell,n}^{\,\ell}
 \!\left[\frac{\boldsymbol r\boldsymbol a}{\ell},
          \frac{B}{\ell}\right](\tau,z),
\end{align*}
which proves the parity refinement.

\end{proof}

\subsection{The specialization \texorpdfstring{$\ell=3$}{ell=3}}

For \(\ell=3\), the data in \eqref{r: p data} agree with the
cubic data in \eqref{r: cubic basic data}: we have \(d=2\),
\(M_\ell=M\),
\(\xi=\omega\), and \(\boldsymbol r=(1,2)^T\).  Write
\(B=[\boldsymbol b,\boldsymbol0]\).  Since
\[
 \frac B3-\sigma(-\boldsymbol b)
 =M^{-1}\begin{pmatrix}\boldsymbol b\\\boldsymbol0\end{pmatrix},
 \qquad
 \frac{\boldsymbol r\boldsymbol a}{3}
 \equiv\rho\boldsymbol a\pmod{\numZ^{2,n}},
\]
Lemma~\ref{lemm:twochar} gives
\begin{equation}
\label{r: ell three characteristic comparison}
 \Theta_{M,n}
 \!\left[\frac{\boldsymbol r\boldsymbol a}{3},
          \frac B3\right](\tau,z)
 =\omega^{\boldsymbol a\cdot\boldsymbol b}
   \Theta^{(n)}_{\boldsymbol a,-\boldsymbol b}(\tau,z).
\end{equation}
The phase in \eqref{r: ell three characteristic comparison}
disappears after cubing.  Consequently,
\eqref{r: unified p power formula} becomes
\[
 \sum_{\boldsymbol a,\boldsymbol b\in\F_3^n}
 \omega^{\boldsymbol\lambda\cdot\boldsymbol a+
          \boldsymbol\kappa\cdot\boldsymbol b}
 \Theta^{(n)}_{\boldsymbol a,-\boldsymbol b}(\tau,z)^3=0.
\]
The change of variables \(\boldsymbol b\mapsto-\boldsymbol b\)
turns this identity into
\eqref{r: cubic character formula}.  Thus
Theorem~\ref{r: unified p power identity} specializes exactly to
Corollary~\ref{r: cubic character identity}, which explains in what
sense it extends the cubic relation from \(A_2\) to
\(A_{\ell-1}\).

\subsection{The specialization \texorpdfstring{$\ell=4$}{ell=4}}
\label{r: ell four specialization}

The case \(\ell=4\) gives a nontrivial fourth-power identity on the
root lattice \(A_3\).  Fix
\(\tau\in\mathfrak H_1\) and write \(q=\etp{\tau}\).  For
\(x=(x_1,x_2,x_3)^T\in\numR^3\), define
\[
 Q_4(x)=x_1^2+x_2^2+x_3^2-x_1x_2-x_2x_3.
\]
For \(0\leq a\leq3\) and \(0\leq j\leq2\), define
\begin{equation}
\label{r: ell four explicit series}
 U_{a,j}(q)=\sum_{x\in\frac{a}{4}(1,2,3)^T+\numZ^3}
 \rmi^{j(-x_1+2x_2-x_3)}
 q^{Q_4\left(x\right)}.
\end{equation}

\begin{lemm}
Write \(R_4=\numZ/4\numZ\).  For \(0\leq a\leq3\),
\(b_1,b_2\in R_4\), and \(b=(b_1,b_2,0)^T\), define
\begin{equation}
\label{r: ell four general theta series}
\begin{aligned}
 \mathcal U_{a,b}(q)
 &:={}
 \Theta_{M_4,1}
 \left[\frac{a(1,2,3)^T}{4},\frac b4\right](\tau,0)\\
 &=\sum_{x\in\frac a4(1,2,3)^T+\numZ^3}
 \rmi^{(2b_1-b_2)x_1+(-b_1+2b_2)x_2-b_2x_3}
 q^{Q_4(x)}.
\end{aligned}
\end{equation}
Then \(\mathcal U_{a,(0,j,0)^T}(q)=U_{a,j}(q)\) for
\(0\leq j\leq2\), and
\begin{equation}
\label{r: ell four fourth power classification}
 \mathcal U_{a,b}(q)^4=
 \begin{cases}
  U_{a,0}(q)^4,
   &(b_1,b_2)=(0,0),\\
  U_{a,2}(q)^4,
   &(b_1,b_2)\in\{(0,2),(2,0),(2,2)\},\\
  U_{a,1}(q)^4,
   &\text{otherwise}.
 \end{cases}
\end{equation}
\end{lemm}

\begin{proof}
The second equality in
\eqref{r: ell four general theta series} follows from
\(M_4(1,2,3)^T=4e_3\) and the fact that the last coordinate of \(b\) is zero.
It also gives
\(\mathcal U_{a,(0,j,0)^T}(q)=U_{a,j}(q)\).
The three changes of summation indices
\[
 \begin{aligned}
 (x_1,x_2,x_3)&\longmapsto(-x_1+x_2,x_2,x_3),\\
 (x_1,x_2,x_3)&\longmapsto(x_1,x_1-x_2+x_3,x_3),\\
 (x_1,x_2,x_3)&\longmapsto(x_1,x_2,x_2-x_3)
 \end{aligned}
\]
are bijections of \(\frac a4(1,2,3)^T+\numZ^3\) and preserve
\(Q_4(x)\).  Direct substitution gives
\[
\begin{aligned}
 \mathcal U_{a,(b_1,b_2,0)^T}(q)
 &=\mathcal U_{a,(b_2-b_1,b_2,0)^T}(q),\\
 \mathcal U_{a,(b_1,b_2,0)^T}(q)
 &=\mathcal U_{a,(b_1,b_1-b_2,0)^T}(q),\\
 \mathcal U_{a,(b_1,b_2,0)^T}(q)
 &=\rmi^{ab_2}\mathcal U_{a,(b_1+b_2,-b_2,0)^T}(q).
\end{aligned}
\]
Here all entries are read modulo \(4\).  Iterating these transformations
fixes \((0,0)\), maps \((0,2),(2,0),(2,2)\) to \((0,2)\), and maps
each of the remaining twelve pairs to \((0,1)\).  The only factors
introduced are powers of
\(\rmi\), which disappear after taking fourth powers.  This proves
\eqref{r: ell four fourth power classification}.
\end{proof}

\begin{coro}
The six series defined in \eqref{r: ell four explicit series}, where
\(a\in\{0,2\}\) and
\(0\leq j\leq2\), satisfy
\begin{equation}
\label{r: ell four power identity}
 U_{0,0}(q)^4-4U_{0,1}(q)^4+3U_{0,2}(q)^4
 =U_{2,0}(q)^4-4U_{2,1}(q)^4+3U_{2,2}(q)^4.
\end{equation}
\end{coro}

\begin{proof}
In \eqref{r: unified p power formula 2}, take
\(\ell=4\), \(n=1\), \(\boldsymbol\lambda=1\),
\(\boldsymbol\kappa=2\), \(\boldsymbol a_0=0\), and \(z=0\).
The hypotheses hold because
\(\boldsymbol\lambda\cdot\boldsymbol\kappa
=2\not\equiv0\pmod4\) and
\(\boldsymbol\kappa\in2\numZ/4\numZ\).  Since
\(\boldsymbol a_0+2\numZ/4\numZ=\{0,2\}\) and \(\xi=\rmi\), the identity becomes
\begin{equation}
\label{r: ell four unreduced power identity}
 \sum_{a\in\{0,2\}}\rmi^a
 \sum_{b_1,b_2\in R_4}
 (-1)^{b_1+b_2}
 \mathcal U_{a,(b_1,b_2,0)^T}(q)^4=0.
\end{equation}

The first two cases in
\eqref{r: ell four fourth power classification} have signed weights
\(1\) and \(3\), while the total weight is
\((1-1+1-1)^2=0\).  Thus the remaining case has weight \(-4\), and
\[
 \sum_{b_1,b_2\in R_4}
 (-1)^{b_1+b_2}
 \mathcal U_{a,(b_1,b_2,0)^T}(q)^4
 =
 U_{a,0}(q)^4-4U_{a,1}(q)^4+3U_{a,2}(q)^4,
\]
for \(a\in\{0,2\}\).  Substituting this into
\eqref{r: ell four unreduced power identity} and using
\(\rmi^2=-1\) gives \eqref{r: ell four power identity}.
\end{proof}

\begin{rema}
Zhu~\cite[Theorem~7.13]{Zhu23} proved three independent quadratic relations among $\Theta_{M_4,1}\left[\frac{a(1,2,3)^T}{4},\frac b4\right](\tau,z)$, where \(a\in\{0,2\}\) and
\[
b\in\bigl\{(0,0,0)^T,(0,2,0)^T,(2,0,0)^T,(2,2,0)^T\bigr\}.
\]
These relations cannot be deduced from Theorem~\ref{r: unified p power identity}, although we expect Theorem~\ref{thm:main} to imply them.
\end{rema}

\subsection{The specialization \texorpdfstring{$\ell=5$}{ell=5}}

The case \(\ell=5\) gives a parallel fifth-power identity on \(A_4\).
Write
\[
 R_5=\numZ/5\numZ,\qquad
 \xi=\etp{1/5},\qquad
 q=\etp{\tau},\qquad \tau\in\mathfrak H_1.
\]
For \(x=(x_1,x_2,x_3,x_4)^T\in\numR^4\), define
\[
 Q_5(x)=x_1^2+x_2^2+x_3^2+x_4^2
        -x_1x_2-x_2x_3-x_3x_4.
\]
For \(0\leq a\leq4\), define
\begin{equation}
\label{r: p five explicit orbit series}
\begin{aligned}
 T_{a,0}(q)
 &=\sum_{x\in\frac a5(1,2,3,4)^T+\numZ^4}q^{Q_5(x)},\\
 T_{a,j}(q)
 &=\sum_{x\in\frac a5(1,2,3,4)^T+\numZ^4}
 \xi^{(1-j)x_1+(3-2j)x_2-(1+j)x_3+(2j-1)x_4}
 q^{Q_5(x)}\qquad(1\leq j\leq5).
\end{aligned}
\end{equation}

\begin{lemm}
\label{r: ell five classification lemma}
For \(0\leq a\leq4\), \(b_1,b_2,b_3\in R_5\), and
\(b=(b_1,b_2,b_3,0)^T\), define
\[
\begin{aligned}
 \mathcal A_{a,b}(q)
 &:=
 \Theta_{M_5,1}
 \left[\frac{a(1,2,3,4)^T}{5},\frac b5\right](\tau,0)\\
 &=\sum_{x\in\frac a5(1,2,3,4)^T+\numZ^4}
 \xi^{(2b_1-b_2)x_1+(-b_1+2b_2-b_3)x_2
       +(-b_2+2b_3)x_3-b_3x_4}q^{Q_5(x)}.
\end{aligned}
\]
Then
\begin{equation}
\label{eq:TA_relation}
\begin{aligned}
T_{a,0}(q)&=\mathcal A_{a,(0,0,0,0)^T}(q),\\
 T_{a,j}(q)&=
 \mathcal A_{a,(2-2j,3-3j,1-2j,0)^T}(q)
 \qquad(1\leq j\leq5),
\end{aligned}
\end{equation}
and
\[
 \mathcal A_{a,b}(q)^5=
 \begin{cases}
  T_{a,0}(q)^5,&b=(0,0,0,0)^T,\\
  T_{a,j}(q)^5,&Q_5(b)\equiv j\pmod5,\quad 1\leq j\leq4,\\
  T_{a,5}(q)^5,&Q_5(b)\equiv0\pmod5,\quad
                       b\ne(0,0,0,0)^T.
 \end{cases}
\]
\end{lemm}

\begin{proof}
The displayed expansion of \(\mathcal A_{a,b}\) follows from
\(M_5(1,2,3,4)^T=5e_4\) and the fact that the last coordinate of \(b\) is zero.

The four changes of summation indices
\[
\begin{aligned}
 (x_1,x_2,x_3,x_4)&\longmapsto(-x_1+x_2,x_2,x_3,x_4),\\
 (x_1,x_2,x_3,x_4)&\longmapsto(x_1,x_1-x_2+x_3,x_3,x_4),\\
 (x_1,x_2,x_3,x_4)&\longmapsto(x_1,x_2,x_2-x_3+x_4,x_4),\\
 (x_1,x_2,x_3,x_4)&\longmapsto(x_1,x_2,x_3,x_3-x_4)
\end{aligned}
\]
are bijections of
\(\frac a5(1,2,3,4)^T+\numZ^4\) and preserve \(Q_5(x)\).
Direct substitution gives, with the argument \(q\) suppressed,
\[
\begin{aligned}
 \mathcal A_{a,(b_1,b_2,b_3,0)^T}
 &=\mathcal A_{a,(b_2-b_1,b_2,b_3,0)^T},\\
 \mathcal A_{a,(b_1,b_2,b_3,0)^T}
 &=\mathcal A_{a,(b_1,b_1+b_3-b_2,b_3,0)^T},\\
 \mathcal A_{a,(b_1,b_2,b_3,0)^T}
 &=\mathcal A_{a,(b_1,b_2,b_2-b_3,0)^T},\\
 \mathcal A_{a,(b_1,b_2,b_3,0)^T}
 &=\xi^{ab_3}
   \mathcal A_{a,(b_1+b_3,b_2+2b_3,-b_3,0)^T}.
\end{aligned}
\]
Here all entries are read modulo \(5\).  In the last relation, the
two exponents differ by \(5b_3(x_3-x_2)\).  Since
\[
 x_3-x_2\in\frac a5+\numZ,
 \qquad
 5b_3(x_3-x_2)\in ab_3+5\numZ,
\]
the last relation has the stated factor \(\xi^{ab_3}\).  Thus
these transformations preserve fifth powers.

We now determine the orbits generated by these transformations.
For \(b=(b_1,b_2,b_3,0)^T\), define
\[
 c(b)=(c_1,c_2,c_3,c_4,c_5)
 :=(b_1,b_2-b_1,b_3-b_2,-b_3,0)\in R_5^5.
\]
Thus \(c_1+\cdots+c_5=0\).  Conversely, every zero-sum tuple with
last entry zero determines \(b\) uniquely by
\[
 b_i=c_1+\cdots+c_i\qquad(1\leq i\leq3).
\]
For the four transformed characteristics above, direct calculation
gives
\[
\begin{aligned}
 c((b_2-b_1,b_2,b_3,0)^T)
  &=(c_2,c_1,c_3,c_4,c_5),\\
 c((b_1,b_1+b_3-b_2,b_3,0)^T)
  &=(c_1,c_3,c_2,c_4,c_5),\\
 c((b_1,b_2,b_2-b_3,0)^T)
  &=(c_1,c_2,c_4,c_3,c_5),\\
 c((b_1+b_3,b_2+2b_3,-b_3,0)^T)
  &=(c_1,c_2,c_3,c_5,c_4)+b_3(1,1,1,1,1).
\end{aligned}
\]
Hence the four transformations realize the four adjacent
transpositions of the entries of \(c(b)\), followed by the unique
common translation that restores the last entry to zero.  Since
adjacent transpositions generate all permutations, \(b\) and \(b'\)
lie in the same orbit if and only if, for some permutation \(\pi\)
and some \(t\in R_5\),
\[
 c(b')=(c_{\pi(1)},c_{\pi(2)},c_{\pi(3)},c_{\pi(4)},c_{\pi(5)})
       +t(1,1,1,1,1).
\]
In particular, characteristics in the same orbit have equal fifth
powers.

It remains to classify the zero-sum multisets of five elements of
\(R_5\) up to a common translation.  The identity
\[
 \sum_{i=1}^5c_i^2=2Q_5(b)
\]
shows that \(Q_5(b)\pmod5\) is constant on every orbit, because
\[
 \sum_{i=1}^5(c_i+t)^2
 =\sum_{i=1}^5c_i^2+2t\sum_{i=1}^5c_i+5t^2
 =\sum_{i=1}^5c_i^2
\]
in \(R_5\), where \(2\) is invertible.

Suppose first that an entry is repeated.  After translating a value of
maximal multiplicity to zero, the
types \(4+1\) and \(3+2\) have the forms
\[
 \{0,0,0,0,d\},\qquad \{0,0,0,d,d\}.
\]
The zero-sum condition gives \(d=0\) and \(2d=0\), respectively, so
both reduce to the constant multiset.
The type \(2+1+1+1\) would require three distinct nonzero elements
\(r,s,u\in R_5\) to have sum zero.  If \(v\) is the remaining
nonzero residue, then
\[
 0=1+2+3+4=r+s+u+v=v,
\]
a contradiction.  The remaining nonconstant
possibilities are therefore
\[
 \{0,0,0,d,-d\},\qquad
 \{0,0,d,d,-2d\},
 \qquad d\in R_5^\times.
\]
Replacing \(d\) by \(-d\) leaves the first multiset unchanged and,
for the second multiset, is equivalent to translating all entries by \(-d\),
since
\[
 \{-d,-d,0,0,-3d\}
 =\{0,0,-d,-d,-2(-d)\}.
\]
Thus the cases \(d=1,2\) give all four classes.  Their values of \(Q_5(b)\)
are determined by
\[
 \frac12\bigl(d^2+(-d)^2\bigr)=d^2,
 \qquad
 \frac12\bigl(2d^2+(-2d)^2\bigr)=3d^2
 \pmod5,
\]
respectively, and hence are \(1,4\) and \(3,2\).  If no entry is
repeated, the multiset is \(R_5\), and
\[
 2Q_5(b)=0^2+1^2+2^2+3^2+4^2\equiv0\pmod5.
\]
Together with the constant multiset, which corresponds to \(b=0\),
this proves that the six orbits are
\[
 b=0,\qquad
 Q_5(b)\equiv j\pmod5\ (1\leq j\leq4),\qquad
 Q_5(b)\equiv0\pmod5\text{ and }b\ne0.
\]

Finally, we identify these six orbits with the series in
\eqref{r: p five explicit orbit series}.  For \(1\leq j\leq5\),
\[
 M_5(2-2j,3-3j,1-2j,0)^T
 =(1-j,3-2j,-1-j,2j-1)^T
\]
and
\[
 Q_5\bigl((2-2j,3-3j,1-2j,0)^T\bigr)
 =5j^2-9j+5\equiv j\pmod5.
\]
Replacing the characteristic vector \(b\) by \(b+5m\), where
\(m\in\numZ^4\) and \(m_4=0\), does not change the series.  Indeed,
for \(x=g+\frac a5(1,2,3,4)^T\), where \(g\in\numZ^4\),
\[
\begin{aligned}
 x^TM_5m
 &=g^TM_5m+\frac a5(1,2,3,4)^TM_5m\\
 &=g^TM_5m+am_4=g^TM_5m\in\numZ.
\end{aligned}
\]
Consequently,
\(\xi^{x^TM_5(b+5m)}=\xi^{x^TM_5b}\).
It follows from \eqref{r: p five explicit orbit series} that
\[
 T_{a,j}(q)=
 \mathcal A_{a,(2-2j,3-3j,1-2j,0)^T}(q)
 \qquad(1\leq j\leq5),
\]
where the characteristic is read modulo \(5\).  Hence \(T_{a,j}\)
corresponds to the orbit with \(Q_5(b)\equiv j\pmod5\) for
\(1\leq j\leq4\).  For \(j=5\), the characteristic is congruent to
\((2,3,1,0)^T\ne0\) and belongs to the nonzero orbit with
\(Q_5(b)\equiv0\pmod5\), while \(T_{a,0}\) corresponds to \(b=0\).
This proves \eqref{eq:TA_relation}.  The last assertion follows from
\eqref{eq:TA_relation} and the classification of the orbits.
\end{proof}

\begin{coro}
The series defined in \eqref{r: p five explicit orbit series} satisfy
\[
\begin{aligned}
&T_{0,0}(q)^5-5T_{0,1}(q)^5+5T_{0,2}(q)^5
 +5T_{0,3}(q)^5-5T_{0,4}(q)^5-T_{0,5}(q)^5\\
={}&T_{1,0}(q)^5-5T_{1,1}(q)^5+5T_{1,2}(q)^5
 +5T_{1,3}(q)^5-5T_{1,4}(q)^5\\
={}&T_{2,0}(q)^5-5T_{2,1}(q)^5+5T_{2,2}(q)^5
 +5T_{2,3}(q)^5-5T_{2,4}(q)^5.
\end{aligned}
\]
\end{coro}

\begin{proof}
In \eqref{r: unified p power formula}, take
\(\ell=5\), \(n=1\), \(\boldsymbol\lambda=\boldsymbol\kappa=1\),
and \(z=0\).  Since
\(\boldsymbol\lambda\cdot\boldsymbol\kappa=1\ne0\) in \(R_5\),
we obtain
\[
 \sum_{a,b_1,b_2,b_3\in R_5}
 \xi^{a+b_1+b_2+b_3}
 \mathcal A_{a,(b_1,b_2,b_3,0)^T}(q)^5=0.
\]

Write \(s=b_1+b_2+b_3\).  The linear map
\[
 (b_1,b_2,b_3)\longmapsto
 (s,b_2+2b_3,2b_1+b_2+3b_3)
\]
has determinant \(3\) in \(R_5\), and
\[
 2Q_5(b)=(b_2+2b_3)^2+s(2b_1+b_2+3b_3).
\]
For fixed \(s\ne0\), every value of \(Q_5(b)\) therefore occurs five
times.  For \(s=0\), the numbers of solutions for
\(Q_5(b)\equiv0,1,2,3,4\pmod5\) are respectively
\(5,0,10,10,0\), since \((b_2+2b_3)^2=2Q_5(b)\).  As
\(\sum_{s=1}^4\xi^s=-1\), separating \(b=0\) shows that the
character sums of the six cases in
Lemma~\ref{r: ell five classification lemma}, in the order
\(j=0,\ldots,5\), are
\[
 1,\ -5,\ 5,\ 5,\ -5,\ -1.
\]
Suppressing the common argument \(q\), define
\[
 H_a:=T_{a,0}^5-5T_{a,1}^5+5T_{a,2}^5
     +5T_{a,3}^5-5T_{a,4}^5-T_{a,5}^5,
\]
so the preceding identity becomes
\[
 \sum_{a=0}^4\xi^aH_a=0.
\]

The change of variables \(x\mapsto-x\) gives
\(\mathcal A_{5-a,b}=\mathcal A_{a,-b}\) for \(1\leq a\leq4\).
Since \(Q_5(-b)=Q_5(b)\), the classification in the lemma yields
\(H_4=H_1\) and \(H_3=H_2\).  Thus
\[
 H_0+(\xi+\xi^4)H_1+(\xi^2+\xi^3)H_2=0.
\]
The coefficientwise Galois automorphism \(\sigma\) of these
\(\numQ(\xi)\)-valued \(q\)-series, defined by
\(\sigma(\xi)=\xi^2\), sends \(\mathcal A_{a,b}^5\) to
\(\mathcal A_{a,2b}^5\).  Since
\(Q_5(2b)\equiv4Q_5(b)\pmod5\), it exchanges the classes \(1,4\)
and \(2,3\), and fixes the other two.  The corresponding coefficients
in \(H_a\) are equal, so \(\sigma\) fixes \(H_a\).  Applying \(\sigma\)
to the preceding identity and subtracting gives
\[
 \bigl(\xi+\xi^4-\xi^2-\xi^3\bigr)(H_1-H_2)=0.
\]
The first factor is nonzero, so \(H_1=H_2\), and then
\(1+\xi+\xi^2+\xi^3+\xi^4=0\) gives \(H_0=H_1\).

Finally, applying, respectively, the composition of the first and fourth
transformations and that of the second and third transformations from
the proof of the lemma to \((2,3,1,0)^T\) gives
\[
 T_{a,5}=\xi^a\mathcal A_{a,(2,0,4,0)^T},
 \qquad
 T_{a,5}=\mathcal A_{a,(2,0,4,0)^T}.
\]
Thus \(T_{a,5}=0\) for \(1\leq a\leq4\).  In particular,
\(T_{1,5}=T_{2,5}=0\), and the asserted identities follow from
\(H_0=H_1=H_2\).
\end{proof}

\section{Doubled-order identities}

The root-lattice construction of the preceding section can also be
carried out with \(2\ell\) factors.  It gives a family of
\(2\ell\)-th-power identities; the smallest case \(\ell=2\)
recovers Jacobi's quartic identity and its degree-two analogue.  Here
the term ``doubled order'' refers to the exponent \(2\ell\).  This
construction is distinct both from the square companions of
Section~\ref{r: square companions section}, which involve theta
series at \(\tau\) and \(2\tau\), and from the same-order
fourth-power identity on \(A_3\) in
Subsection~\ref{r: ell four specialization}.  We retain
\(M_\ell,\boldsymbol r,\xi\), and \(R\) from
\eqref{r: p data} and \eqref{r: Coxeter matrix}, together with the
identities and integral matrix in \eqref{r: Coxeter identities} and
\eqref{r: auxiliary integral matrix}.
Throughout this section, \(d=\ell-1\), \(\nu=2\ell\), and \(L\)
denotes the lower triangular matrix from
Section~\ref{r: arbitrary order section}.  For a matrix \(x\) with
entries in \(\numZ/\nu\numZ\), the notation
\(x\equiv0\pmod{\ell}\) means that every entry of \(x\) lies in
\(\ell\numZ/\nu\numZ\).

\subsection{The \texorpdfstring{$2\ell$}{2ell}-fold identity}

\begin{thm}
\label{r: doubled order power identity}
Let \(\ell\geq2\), \(n\geq1\), \(\tau\in\mathfrak H_n\), and
\(z\in\numC^{d,n}\).  If
\(\boldsymbol\lambda,\boldsymbol\kappa
\in(\numZ/\ell\numZ)^n\) satisfy
\(\boldsymbol\lambda\cdot\boldsymbol\kappa\not\equiv0\pmod\ell\),
then
\begin{equation}
\label{r: doubled order power formula}
 \sum_{\substack{\boldsymbol a\in(\numZ/\ell\numZ)^n,
                   \,\boldsymbol c\in\{0,1\}^n\\
                   \boldsymbol b_1,\ldots,\boldsymbol b_{d-1}
                   \in(\numZ/2\ell\numZ)^n}}
 \xi^{\boldsymbol\lambda\cdot\boldsymbol a+
       \boldsymbol\kappa\cdot
       (\boldsymbol b_1+\cdots+\boldsymbol b_{d-1}+\boldsymbol c)}
 \Theta_{M_\ell,n}^{\,2\ell}
 \!\left[\frac{\boldsymbol r\boldsymbol a}{\ell},
          \frac{B}{2\ell}\right](\tau,z)=0.
\end{equation}
Here
\(B=[\boldsymbol b_1,\ldots,\boldsymbol b_{d-1},\boldsymbol c]\).
When \(d=1\), the list
\(\boldsymbol b_1,\ldots,\boldsymbol b_{d-1}\) is empty.  The
exponent of \(\xi\) is computed modulo \(\ell\).
\end{thm}

As in the proof of Theorem~\ref{r: unified p power identity},
we first need to establish a character projection formula.
We write \(\zeta=\etp{1/\nu}\) and define
\begin{equation}
\label{r: doubled finite groups}
 \mathcal X_n^\dagger=
 \{x\in(\numZ/\nu\numZ)^{d,n}:
       \boldsymbol1^Tx\equiv0\pmod\ell\},
 \qquad
 \mathcal Y_n^\dagger=
 \frac{(\numZ/\nu\numZ)^{d,n}}
 {\{\boldsymbol1\boldsymbol\eta:
       \boldsymbol\eta\in(2\numZ/\nu\numZ)^{1,n}\}}.
\end{equation}
Both are finite abelian groups of order \(2^n\nu^{(d-1)n}\).
Let
\(E_\nu=\nu(R-I_d)^{-1}=2E\).  Equations
\eqref{r: auxiliary lattice identities} give
\[
 E_\nu^T\numZ^d=\nu\numZ^d+2\numZ\boldsymbol1.
\]
Reduction of the first equality in
\eqref{r: auxiliary lattice identities} modulo \(\nu\) gives
\begin{equation}
\label{eq:image-RminusId}
 (R-I_d)\cdot(\numZ/\nu\numZ)^{d,n}=
 \{x\in(\numZ/\nu\numZ)^{d,n}:
       \boldsymbol1^Tx\equiv0\pmod\ell\}
 =\mathcal X_n^\dagger.
\end{equation}
The map \(x\mapsto Rx\) is also an automorphism of
\(\mathcal X_n^\dagger\).
We define
\begin{equation}
\label{r: doubled finite pairing}
 \langle x,s\rangle_\nu=\zeta^{\tr(x^Ts)}
\end{equation}
for $(x,s)\in\mathcal X_n^\dagger\times\mathcal Y_n^\dagger$.
\begin{lemm}
\label{lemm:XndYnd}
Formula \eqref{r: doubled finite pairing} defines a well-defined
bicharacter on
\(\mathcal X_n^\dagger\times\mathcal Y_n^\dagger\).  Moreover, if
\(x\neq0\), then some \(s\in\mathcal Y_n^\dagger\) satisfies
\(\langle x,s\rangle_\nu\neq1\).
\end{lemm}
\begin{proof}
If \(s\) is replaced by \(s+\boldsymbol1\boldsymbol\eta\), where
\(\boldsymbol\eta\in(2\numZ/\nu\numZ)^{1,n}\), then
\[
 \tr\!\left(x^T\boldsymbol1\boldsymbol\eta\right)
 =\tr\!\left((\boldsymbol1^Tx)^T\boldsymbol\eta\right)
 \equiv0\pmod\nu,
\]
because \(\boldsymbol1^Tx\equiv0\pmod\ell\) and
\(\boldsymbol\eta\equiv0\pmod2\).  Hence the pairing is well defined
and is a bicharacter.  If \(x\ne0\), choose
\(y\in(\numZ/\nu\numZ)^{d,n}\) such that
\(\tr(x^Ty)\not\equiv0\pmod\nu\); the class of \(y\) in
\(\mathcal Y_n^\dagger\) then has nontrivial pairing with \(x\).
\end{proof}

\begin{lemm}
\label{lemm:character-orthogonality-2l}
Suppose that \(F_{x,s}\), indexed by
\(x\in\mathcal X_n^\dagger\) and \(s\in\mathcal Y_n^\dagger\), satisfies
\begin{equation}
\label{r: abstract Fourier relation dagger}
 2^n\nu^{(d-1)n}F_{R^{-1}x,t}
 =\sum_{\substack{y\in\mathcal X_n^\dagger,\,s\in\mathcal Y_n^\dagger}}
 \langle y-x,s-t\rangle_\nu F_{y,s},
\end{equation}
for every \(x\in\mathcal X_n^\dagger\) and
\(t\in\mathcal Y_n^\dagger\).  Let
\(\boldsymbol\lambda,\boldsymbol\kappa
\in(\numZ/\ell\numZ)^{1,n}\) satisfy \(\boldsymbol\lambda\boldsymbol\kappa^T\not\equiv0\pmod\ell\).
Then
\begin{equation}
\label{r: character sum dagger}
 \sum_{\substack{\boldsymbol a\in(\numZ/\ell\numZ)^{1,n}\\
                  t\in\mathcal Y_n^\dagger}}
 \xi^{\boldsymbol\lambda\boldsymbol a^T}
 \langle2\boldsymbol r\boldsymbol\kappa,t\rangle_\nu
 F_{2\boldsymbol r\boldsymbol a,t}=0.
\end{equation}
\end{lemm}
\begin{proof}
Since \(R^{-1}\boldsymbol r=\boldsymbol r-\ell e_d\), we have
\(F_{2R^{-1}\boldsymbol r\boldsymbol a,t}
=F_{2\boldsymbol r\boldsymbol a,t}\).  Here
\(2\boldsymbol r\boldsymbol a\) and
\(2\boldsymbol r\boldsymbol\kappa\) are regarded as elements of
\(\mathcal X_n^\dagger\) via the map
\((\numZ/\ell\numZ)^{d,n}\rightarrow
(2\numZ/\nu\numZ)^{d,n}\), \(v\mapsto2v\).
We multiply \eqref{r: abstract Fourier relation dagger}, with
\(x=2\boldsymbol r\boldsymbol a\), by
\(\xi^{\boldsymbol\lambda\boldsymbol a^T}
\langle2\boldsymbol r\boldsymbol\kappa,t\rangle_\nu\) and sum over
\(\boldsymbol a\) and \(t\).  Lemma~\ref{lemm:XndYnd} and character
orthogonality on \(\mathcal Y_n^\dagger\) give
\[
 \sum_{t\in\mathcal Y_n^\dagger}
 \langle2\boldsymbol r\boldsymbol\kappa,t\rangle_\nu\langle y-2\boldsymbol r\boldsymbol a,-t\rangle_\nu=
 \sum_{t\in\mathcal Y_n^\dagger}
 \langle2\boldsymbol r\boldsymbol\kappa-
 (y-2\boldsymbol r\boldsymbol a),t\rangle_\nu
 =\begin{cases}
   2^n\nu^{(d-1)n},&
   y-2\boldsymbol r\boldsymbol a
      =2\boldsymbol r\boldsymbol\kappa,\\
   0,&\text{otherwise}.
  \end{cases}
\]
The value \(y=2\boldsymbol r(\boldsymbol a+\boldsymbol\kappa)\)
lies in \(\mathcal X_n^\dagger\).  If \(S\) denotes the sum in
\eqref{r: character sum dagger}, the left-hand side of the summed
Fourier relation is \(2^n\nu^{(d-1)n}S\), whereas its right-hand side
is
\[
 2^n\nu^{(d-1)n}
 \sum_{\boldsymbol a,s}
 \xi^{\boldsymbol\lambda\boldsymbol a^T}
 \langle2\boldsymbol r\boldsymbol\kappa,s\rangle_\nu
 F_{2\boldsymbol r(\boldsymbol a+\boldsymbol\kappa),s}
 =2^n\nu^{(d-1)n}
  \xi^{-\boldsymbol\lambda\boldsymbol\kappa^T}S.
\]
Since
\(\xi^{-\boldsymbol\lambda\boldsymbol\kappa^T}\ne1\), it
follows that \(S=0\).
\end{proof}

\begin{proof}[Proof of Theorem~\ref{r: doubled order power identity}]

For \(x\in\mathcal X_n^\dagger\) and
\(s\in\mathcal Y_n^\dagger\), we choose integral lifts and define
\begin{equation}
\label{eq:Fxs-2ell}
 \gamma_s=-\frac{Ls}{\nu},
 \qquad
 F_{x,s}=\Theta_{M_\ell,n}^{\,\nu}
 \!\left[\frac{x}{\nu},\gamma_s\right](\tau,z).
\end{equation}
This is well defined.  Indeed, replacing \(s\) by
\(s+\boldsymbol1\boldsymbol\eta\), where $\boldsymbol\eta\in(2\numZ)^{1,n}$,
changes \(\gamma_s\) by \(-\boldsymbol r\boldsymbol\eta/\nu
=M_\ell^{-1}(-e_d\boldsymbol\eta/2)\).  By
Lemma~\ref{lemm:twochar}, the corresponding factor in a single theta
series is
\(\etp{-\frac{1}{2\nu}\tr(x^Te_d\boldsymbol\eta)}\), whose \(\nu\)-th power is
one.  Changing an integral lift of \(x\) shifts the first
characteristic by an integral matrix.  Changing an integral lift of
\(s\) by \(\nu k\), with \(k\in\numZ^{d,n}\), changes \(\gamma_s\)
by \(-Lk\).  By Lemma~\ref{lemm:twochar}, the latter change multiplies
a single theta series by
\[
 \zeta^{-\tr(x^TM_\ell Lk)},
\]
whose \(\nu\)-th power is one.
Thus \(F_{x,s}\) is independent of all choices.

We apply Theorem~\ref{thm:main} with
\[
 m=d,\qquad \mu=\nu,\qquad M_1=M_2=M_\ell,
 \qquad
 A=I_{d\nu}+\frac1\nu J_\nu\otimes(R-I_d),
\]
where \(J_\nu\) is the \(\nu\times\nu\) all-ones matrix.
A direct calculation gives
\(A^TM_\ell^{(\nu)}A=M_\ell^{(\nu)}\), as in the proof of
Theorem~\ref{r: unified p power identity}.
The integral matrix \(\widetilde{A}\) is defined by
\[
 \widetilde{A}\cdot(x_1,\ldots,x_\nu)^T
 =(x_1-x_\nu,\ldots,x_{\nu-1}-x_\nu,x_\nu E_\nu^T)^T,\qquad
 \text{for all } x_1,\dots,x_\nu\in\numZ^{1,d}.
\]
For \(h=[h_1,\ldots,h_\nu]\in\numZ^{d\nu,n}\), we have
\[
 Ah-h=
 \left[\frac{(R-I_d)u}{\nu},\ldots,
       \frac{(R-I_d)u}{\nu}\right],
 \qquad u=h_1+\cdots+h_\nu,
\]
and every \(u\in\numZ^{d,n}\) occurs in this way.  On the other
hand, \(\widetilde{A}v\) is integral precisely when all blocks of \(v\) are
congruent modulo \(\numZ^{d,n}\) and their common class belongs to
\(E_\nu^{-1}\numZ^{d,n}=\nu^{-1}(R-I_d)\numZ^{d,n}\).  Consequently,
\[
 \numZ^{d\nu,n}+A\numZ^{d\nu,n}
 =\widetilde{A}^{-1}\numZ^{d\nu,n}
 =\numZ^{d\nu,n}+
 \left\{\left[\frac{(R-I_d)u}{\nu},\ldots,
                    \frac{(R-I_d)u}{\nu}\right]:
              u\in\numZ^{d,n}\right\}.
\]
Thus all the hypotheses of Theorem~\ref{thm:main} are satisfied.
Moreover, \(|\det \widetilde{A}|=|\det E_\nu|=2^d\ell^{d-1}=2\nu^{d-1}\).

The two quotient groups in Theorem~\ref{thm:main} are parametrized
by \(\mathcal X_n^\dagger\) and \(\mathcal Y_n^\dagger\), respectively.
More precisely, the first parametrization is the group isomorphism
\[
 \mathcal X_n^\dagger\longrightarrow
 \frac{(\widetilde{A}A)^{-1}\numZ^{d\nu,n}}{\numZ^{d\nu,n}},
 \qquad
 v\longmapsto
 \left[\frac v\nu,\ldots,\frac v\nu\right]+\numZ^{d\nu,n},
\]
and the second parametrization is
\[
 \mathcal Y_n^\dagger\longrightarrow
 \frac{A^{-1}{M_\ell^{(\nu)}}^{-1}\numZ^{d\nu,n}}
 {A^{-1}{M_\ell^{(\nu)}}^{-1}\widetilde{A}^T\numZ^{d\nu,n}},
 \qquad u\longmapsto\beta(u)+A^{-1}{M_\ell^{(\nu)}}^{-1}\widetilde{A}^T\numZ^{d\nu,n},
\]
where
\[
 \beta(u)=[\gamma_u,\ldots,\gamma_u,
   \gamma_u+M_\ell^{-1}u].
\]
The proof that these maps are group isomorphisms is analogous to the
proof for the maps from \(\mathcal X_n\) and \(\mathcal Y_n\) in
Theorem~\ref{r: unified p power identity}.

For fixed \(x\in\mathcal X_n^\dagger\) and
\(t\in\mathcal Y_n^\dagger\), we specialize the parameters in
Theorem~\ref{thm:main} to
\[
 a_i=\frac{x}{\nu},\qquad b_i=R\gamma_t,\qquad z_i=Rz
 \qquad(1\leq i\leq\nu).
\]
Since \(A\) acts as \(R\) on the diagonal subspace, the transformed
parameters are
\[
 a_i'=\frac{R^{-1}x}{\nu},\qquad b_i'=\gamma_t,
 \qquad w_i=z \qquad(1\leq i\leq\nu).
\]
Because \(R^{-1}x\in\mathcal X_n^\dagger\), taking
\(v=y-R^{-1}x\) in the first parametrization gives the representatives
\[
 \left[\frac{y-R^{-1}x}{\nu},\ldots,
       \frac{y-R^{-1}x}{\nu}\right],
 \qquad y\in\mathcal X_n^\dagger.
\]
For the representative indexed by \(u\), write \(s=t+u\), or
equivalently \(u=s-t\), in \(\mathcal Y_n^\dagger\).  The identity
\(I_d-M_\ell L=R^T\) gives
\[
 M_\ell(\nu\gamma_u+M_\ell^{-1}u)=R^Tu.
\]
The exterior factor in \eqref{eq:mainIdentity} is therefore
\(\zeta^{-\tr(x^Tu)}\).  Removing \(M_\ell^{-1}u\) from the second
characteristic in the last theta factor contributes
\(\zeta^{\tr(y^Tu)}\).  Their product is
\(\langle y-x,s-t\rangle_\nu\).  The left-hand side is
\(2^n\nu^{(d-1)n}F_{R^{-1}x,t}\), since
\[
 \Theta_{M_\ell,n}[v,Rw](\tau,Rz)
 =\Theta_{M_\ell,n}[R^{-1}v,w](\tau,z).
\]
Thus Theorem~\ref{thm:main} yields
\begin{equation}
\label{r: doubled finite Fourier relation}
 2^n\nu^{(d-1)n}F_{R^{-1}x,t}
 =\sum_{\substack{y\in\mathcal X_n^\dagger,\,
                   s\in\mathcal Y_n^\dagger}}
 \langle y-x,s-t\rangle_\nu F_{y,s},
\end{equation}
which is precisely the relation in
\eqref{r: abstract Fourier relation dagger}.
Lemma~\ref{lemm:character-orthogonality-2l} therefore gives
\eqref{r: character sum dagger} for the \(F_{x,s}\) defined in
\eqref{eq:Fxs-2ell}.

It remains to choose coordinates on \(\mathcal Y_n^\dagger\).  The
unimodularity of \(L\) identifies a lift of \(t\) with
\(B=-Lt\) modulo \(\nu\).  Replacing \(t\) by
\(t+\boldsymbol1\boldsymbol\eta\) replaces \(B\) by
\(B-\boldsymbol r\boldsymbol\eta\).  Since
\(d=\ell-1\) is invertible modulo \(\ell\), every class has a
unique representative for which the last row of \(B\) belongs to
\(\{0,1\}^n\).  The remaining \(d-1\) rows range independently over
\((\numZ/\nu\numZ)^n\); hence these representatives have the form
\[
 B=[\boldsymbol b_1,\ldots,\boldsymbol b_{d-1},\boldsymbol c]
\]
used in the statement.  Finally,
\(-\boldsymbol r^TL^{-1}=(1,\ldots,1,-d)\equiv\boldsymbol1^T
\pmod\ell\), and therefore
\[
 \langle2\boldsymbol r\boldsymbol\kappa,t\rangle_\nu
 =\xi^{\boldsymbol\kappa\cdot
   (\boldsymbol b_1+\cdots+\boldsymbol b_{d-1}+\boldsymbol c)}.
\]
Combining this equation with \eqref{r: character sum dagger} and
\eqref{eq:Fxs-2ell} yields \eqref{r: doubled order power formula}.
\end{proof}

\subsection{Quartic identities}

We now specialize to \(\ell=2\), so that \(d=1\) and \(M_2=(2)\).
Let \(\F_2=\numZ/2\numZ\).  For
\(\boldsymbol p,\boldsymbol q\in\F_2^n\), we define
\begin{equation}
\label{r: indexed quartic theta}
 \Phi^{(n)}_{\boldsymbol p,\boldsymbol q}(\tau,z)
 :=\Theta_{(2),n}
 \!\left[\frac{\boldsymbol p}{2},
          \frac{\boldsymbol q}{4}\right](\tau,z),
 \qquad
 \phi^{(n)}_{\boldsymbol p,\boldsymbol q}(\tau)
 :=\Phi^{(n)}_{\boldsymbol p,\boldsymbol q}(\tau,0).
\end{equation}
Changing the lift of \(\boldsymbol p\) has no effect, whereas changing
the lift of \(\boldsymbol q\) to \(\boldsymbol q+2\boldsymbol k\),
with \(\boldsymbol k\in\numZ^n\), multiplies the series by
\((-1)^{\boldsymbol p\cdot\boldsymbol k}\), by
Lemma~\ref{lemm:twochar}.  Thus the fourth powers are independent of
the choice of representatives.

\begin{coro}
Let \(n\geq1\), \(\tau\in\mathfrak H_n\), and
\(z\in\numC^{1,n}\).  If
\(\boldsymbol\lambda,\boldsymbol\kappa\in\F_2^n\) satisfy
\(\boldsymbol\lambda\cdot\boldsymbol\kappa=1\), then
\begin{equation}
\label{r: quartic character formula}
 \sum_{\boldsymbol p,\boldsymbol q\in\F_2^n}
 (-1)^{\boldsymbol\lambda\cdot\boldsymbol p+
       \boldsymbol\kappa\cdot\boldsymbol q}
 \Phi^{(n)}_{\boldsymbol p,\boldsymbol q}(\tau,z)^4=0.
\end{equation}
Moreover,
\begin{equation}
\label{r: quartic theta nullwert vanishing}
 \boldsymbol p\cdot\boldsymbol q=1
 \quad\Longrightarrow\quad
 \phi^{(n)}_{\boldsymbol p,\boldsymbol q}(\tau)=0.
\end{equation}
\end{coro}

\begin{proof}
Equation~\eqref{r: quartic character formula} follows from
Theorem~\ref{r: doubled order power identity} upon specializing to
\(\ell=2\), with
\(\boldsymbol a=\boldsymbol p\) and
\(\boldsymbol c=\boldsymbol q\).  For \eqref{r: quartic theta nullwert vanishing},
the involution \(g\mapsto-g-\boldsymbol p\) on the defining series
has no fixed point.  Indeed,
\(\boldsymbol p\cdot\boldsymbol q=1\) implies \(p_i=1\) for some
\(i\), whereas a fixed point would satisfy \(2g_i=-1\).  The ratio
of the two terms in each pair is
\((-1)^{\boldsymbol p\cdot\boldsymbol q}=-1\), so the paired terms
cancel.
\end{proof}

In degree one, \eqref{r: quartic character formula} gives an
elliptic-variable form of Jacobi's quartic identity.

\begin{coro}
For \(\tau\in\mathfrak H_1\) and \(z\in\numC\),
\begin{equation}
\label{r: Jacobi quartic formula with variable}
 \Phi^{(1)}_{0,0}(\tau,z)^4+
 \Phi^{(1)}_{1,1}(\tau,z)^4
 =\Phi^{(1)}_{1,0}(\tau,z)^4+
  \Phi^{(1)}_{0,1}(\tau,z)^4.
\end{equation}
At \(z=0\), this becomes
\begin{equation}
\label{r: Jacobi quartic theta constant formula}
 \phi^{(1)}_{0,0}(\tau)^4
 =\phi^{(1)}_{1,0}(\tau)^4+\phi^{(1)}_{0,1}(\tau)^4.
\end{equation}
\end{coro}

\begin{proof}
For \(n=1\) and \(\lambda=\kappa=1\), the four terms in
\eqref{r: quartic character formula} are
\[
 \Phi^{(1)}_{0,0}(\tau,z)^4-
 \Phi^{(1)}_{0,1}(\tau,z)^4-
 \Phi^{(1)}_{1,0}(\tau,z)^4+
 \Phi^{(1)}_{1,1}(\tau,z)^4=0.
\]
Rearrangement gives \eqref{r: Jacobi quartic formula with variable}.
At \(z=0\), \eqref{r: quartic theta nullwert vanishing} gives
\(\phi^{(1)}_{1,1}=0\), and hence
\eqref{r: Jacobi quartic theta constant formula}.
\end{proof}

For an explicit \(q\)-series normalization, write
\(q=\etp{\tau}\) and \(u=\etp{2z}\), with
\(u^\alpha=\etp{2\alpha z}\).  Then
\eqref{r: Jacobi quartic formula with variable} is
\begin{equation}
\label{r: Jacobi quartic q series formula}
\begin{aligned}
&\left(\sum_{k\in\numZ}q^{k^2}u^k\right)^4
 +\left(\sum_{k\in\numZ}(-1)^k
        q^{(k+1/2)^2}u^{k+1/2}\right)^4\\
={}&\left(\sum_{k\in\numZ}
        q^{(k+1/2)^2}u^{k+1/2}\right)^4
 +\left(\sum_{k\in\numZ}(-1)^kq^{k^2}u^k\right)^4.
\end{aligned}
\end{equation}
Here \(\Phi^{(1)}_{1,1}\) is \(\rmi\) times the corresponding
displayed series, and \(\rmi^4=1\).  Chan, Chan, Liu, and
Zudilin~\cite[Theorem~1.2]{CCLZ26} proved an identity whose
specialization \(x=y=u\) gives
\eqref{r: Jacobi quartic q series formula}.
Jacobi~\cite[p.~90, Eq.~(14)]{Jac12} gives the classical
theta-constant identity \eqref{r: Jacobi quartic theta constant formula}.
In \(q\)-series notation, this is
\begin{equation*}
 \left(\sum_{k\in\numZ}q^{k^2}\right)^4
 =\left(\sum_{k\in\numZ}q^{(k+1/2)^2}\right)^4
  +\left(\sum_{k\in\numZ}(-1)^kq^{k^2}\right)^4.
\end{equation*}

We finally specialize to degree two.  We write
\[
 \tau=\begin{pmatrix}\tau_1&\tau_0\\\tau_0&\tau_2\end{pmatrix}
 \in\mathfrak H_2,
 \qquad Q=\etp{\tau_1},\quad S=\etp{\tau_0},\quad
 T=\etp{\tau_2}.
\]
For binary representatives
\(\boldsymbol p=(p_1,p_2)\) and
\(\boldsymbol q=(q_1,q_2)\), the definition
\eqref{r: indexed quartic theta} gives
\begin{equation}
\label{r: degree two quartic q series}
 \phi^{(2)}_{\boldsymbol p,\boldsymbol q}(\tau)
 =\rmi^{p_1q_1+p_2q_2}
  \sum_{j,k\in\numZ}(-1)^{q_1j+q_2k}
  Q^{(j+p_1/2)^2}
  S^{2(j+p_1/2)(k+p_2/2)}
  T^{(k+p_2/2)^2}.
\end{equation}
We use the abbreviations
\[
 \boldsymbol0=(0,0),\qquad
 \boldsymbol e_1=(1,0),\qquad
 \boldsymbol e_2=(0,1),\qquad
 \boldsymbol e=(1,1)
 \quad\text{in }\F_2^2.
\]
The ten pairs with
\(\boldsymbol p\cdot\boldsymbol q=0\) are
\begin{equation}
\label{r: degree two quartic even pairs}
\begin{gathered}
 (\boldsymbol p,\boldsymbol q)=(\boldsymbol0,\boldsymbol0),
 (\boldsymbol0,\boldsymbol e_1),
 (\boldsymbol0,\boldsymbol e_2),
 (\boldsymbol0,\boldsymbol e),
 (\boldsymbol e_1,\boldsymbol0),
 (\boldsymbol e_1,\boldsymbol e_2),\\
 (\boldsymbol e_2,\boldsymbol0),
 (\boldsymbol e_2,\boldsymbol e_1),
 (\boldsymbol e,\boldsymbol0),
 (\boldsymbol e,\boldsymbol e).
\end{gathered}
\end{equation}
The relations in the following theorem are the classical genus-two
quartic relations; see Cl\'ery, van der Geer, and
Grushevsky~\cite{CGG15}.

\begin{thm}
The fourth powers of the theta series defined by
\eqref{r: degree two quartic q series} and indexed by
\eqref{r: degree two quartic even pairs} satisfy the following five
independent relations:
\begin{align*}
 \phi_{\boldsymbol0,\boldsymbol0}^4
 &=\phi_{\boldsymbol0,\boldsymbol e_2}^4
  +\phi_{\boldsymbol e_2,\boldsymbol0}^4
  +\phi_{\boldsymbol e,\boldsymbol e}^4,\\
 \phi_{\boldsymbol0,\boldsymbol0}^4
 &=\phi_{\boldsymbol0,\boldsymbol e_2}^4
  +\phi_{\boldsymbol e_2,\boldsymbol e_1}^4
  +\phi_{\boldsymbol e,\boldsymbol0}^4,\\
 \phi_{\boldsymbol0,\boldsymbol0}^4
 &=\phi_{\boldsymbol0,\boldsymbol e_1}^4
  +\phi_{\boldsymbol e_1,\boldsymbol0}^4
  +\phi_{\boldsymbol e,\boldsymbol e}^4,\\
 \phi_{\boldsymbol0,\boldsymbol0}^4
 &=\phi_{\boldsymbol0,\boldsymbol e_1}^4
  +\phi_{\boldsymbol e_1,\boldsymbol e_2}^4
  +\phi_{\boldsymbol e,\boldsymbol0}^4,\\
 \phi_{\boldsymbol0,\boldsymbol0}^4
 &=\phi_{\boldsymbol0,\boldsymbol e}^4
  +\phi_{\boldsymbol e_2,\boldsymbol e_1}^4
  +\phi_{\boldsymbol e_1,\boldsymbol0}^4.
\end{align*}
(We suppress the superscript \(^{(2)}\) and the argument \(\tau\) in
this display.)
These relations form a basis of all linear relations among the ten
fourth powers.
\end{thm}

\begin{proof}
We order the ten fourth powers according to the lexicographic list of
even characteristics used by Cl\'ery, van der Geer, and
Grushevsky~\cite[Section~3]{CGG15}:
\[
\begin{aligned}
(x_1,\ldots,x_{10})={}&(
 \phi_{\boldsymbol0,\boldsymbol0}^4,
 \phi_{\boldsymbol0,\boldsymbol e_2}^4,
 \phi_{\boldsymbol0,\boldsymbol e_1}^4,
 \phi_{\boldsymbol0,\boldsymbol e}^4,
 \phi_{\boldsymbol e_2,\boldsymbol0}^4,
 \phi_{\boldsymbol e_2,\boldsymbol e_1}^4,\\
&\hspace{36mm}
 \phi_{\boldsymbol e_1,\boldsymbol0}^4,
 \phi_{\boldsymbol e_1,\boldsymbol e_2}^4,
 \phi_{\boldsymbol e,\boldsymbol0}^4,
 \phi_{\boldsymbol e,\boldsymbol e}^4).
\end{aligned}
\]
The six admissible pairs
\((\boldsymbol\lambda,\boldsymbol\kappa)\) are, in the order used
below,
\[
 (\boldsymbol e_2,\boldsymbol e_2),\quad
 (\boldsymbol e_2,\boldsymbol e),\quad
 (\boldsymbol e_1,\boldsymbol e_1),\quad
 (\boldsymbol e_1,\boldsymbol e),\quad
 (\boldsymbol e,\boldsymbol e_2),\quad
 (\boldsymbol e,\boldsymbol e_1).
\]
After the six odd theta constants are removed by
\eqref{r: quartic theta nullwert vanishing}, the corresponding six
instances of \eqref{r: quartic character formula} have coefficient
matrix
\[
 \begin{pmatrix}
  1&-1& 1&-1&-1&-1& 1&-1&-1& 1\\
  1&-1&-1& 1&-1& 1& 1&-1&-1&-1\\
  1& 1&-1&-1& 1&-1&-1&-1&-1& 1\\
  1&-1&-1& 1& 1&-1&-1& 1&-1&-1\\
  1&-1& 1&-1&-1&-1&-1& 1& 1&-1\\
  1& 1&-1&-1&-1& 1&-1&-1& 1&-1
 \end{pmatrix}.
\]
If its rows are denoted by \(\rho_1,\ldots,\rho_6\), the coefficient
vectors of the five displayed relations are, respectively,
\[
 \frac{\rho_2+\rho_5}{2},\qquad
 \frac{\rho_1+\rho_4}{2},\qquad
 \frac{\rho_4+\rho_6}{2},\qquad
 \frac{\rho_2+\rho_3}{2},\qquad
 \frac{\rho_3+\rho_5}{2}.
\]
This proves all five relations.  Their coefficient vectors are
independent: the coefficients of
\(x_5,x_8,x_4,x_3,x_2\), in this order, successively force all five
coefficients in a linear dependence to vanish.

For comparison with the standard theta constants, we use the
normalization
\[
 \vartheta\!\left[
   \begin{matrix}\boldsymbol p\\ \boldsymbol q\end{matrix}
  \right](\Omega,0)
 :=\sum_{g\in\numZ^2}
 \exp\!\left(\pi\rmi\left\{
  (g+\boldsymbol p/2)^T\Omega(g+\boldsymbol p/2)
  +(g+\boldsymbol p/2)^T\boldsymbol q\right\}\right).
\]
Definition~\eqref{r: indexed quartic theta} then gives the exact
equality
\[
 \phi^{(2)}_{\boldsymbol p,\boldsymbol q}(\tau)
 =\vartheta\!\left[
   \begin{matrix}\boldsymbol p\\ \boldsymbol q\end{matrix}
  \right](2\tau,0).
\]
Under the ordering above, the fourth powers of these quantities are
precisely the forms \(x_i=\vartheta_{n_i}^4\) in the notation of
Cl\'ery, van der Geer,
and Grushevsky~\cite[Notation~6.1]{CGG15}, evaluated at \(2\tau\),
with no additional scalar factor.  The same authors~\cite[Section~6,
Eq.~(6.1)]{CGG15} show that these ten forms span a
five-dimensional space.  Hence their space of linear relations has
dimension five, and the five independent relations displayed in the
theorem form a basis.
\end{proof}

\bibliographystyle{plain}
\bibliography{ThetaRelations_nw}

@book {Igu72,
    AUTHOR = {Igusa, Jun-ichi},
     TITLE = {Theta functions},
    SERIES = {Die Grundlehren der mathematischen Wissenschaften, Band 194},
 PUBLISHER = {Springer-Verlag, New York-Heidelberg},
      YEAR = {1972},
     PAGES = {x+232},
   MRCLASS = {14K25 (32N05)},
  MRNUMBER = {0325625},
MRREVIEWER = {H. Klingen},
}

@book {AZ95,
    AUTHOR = {Andrianov, A. N. and Zhuravl\"{e}v, V. G.},
     TITLE = {Modular forms and {H}ecke operators},
    SERIES = {Translations of Mathematical Monographs},
    VOLUME = {145},
      NOTE = {Translated from the 1990 Russian original by Neal Koblitz},
 PUBLISHER = {American Mathematical Society, Providence, RI},
      YEAR = {1995},
     PAGES = {viii+334},
      ISBN = {0-8218-0277-1},
   MRCLASS = {11F25 (11F11 11F46 11F60)},
  MRNUMBER = {1349824},
       DOI = {10.1090/mmono/145},
       URL = {https://doi.org/10.1090/mmono/145},
}

@article {Zie89,
    AUTHOR = {Ziegler, C.},
     TITLE = {Jacobi forms of higher degree},
   JOURNAL = {Abh. Math. Sem. Univ. Hamburg},
  FJOURNAL = {Abhandlungen aus dem Mathematischen Seminar der Universit\"at
              Hamburg},
    VOLUME = {59},
      YEAR = {1989},
     PAGES = {191--224},
      ISSN = {0025-5858,1865-8784},
   MRCLASS = {11F55 (11F37 11F46)},
  MRNUMBER = {1049896},
MRREVIEWER = {Daniel\ Bump},
       DOI = {10.1007/BF02942329},
       URL = {https://doi.org/10.1007/BF02942329},
}

@article {ZWL10,
    AUTHOR = {Zhou, Haigang and Wang, Tianze and Lu, Hongwen},
     TITLE = {On dimension formulas for {J}acobi forms},
   JOURNAL = {Math. Z.},
  FJOURNAL = {Mathematische Zeitschrift},
    VOLUME = {264},
      YEAR = {2010},
    NUMBER = {3},
     PAGES = {709--725},
      ISSN = {0025-5874,1432-1823},
   MRCLASS = {11F50 (11F46)},
  MRNUMBER = {2591827},
MRREVIEWER = {Aloys\ Krieg},
       DOI = {10.1007/s00209-009-0539-2},
       URL = {https://doi.org/10.1007/s00209-009-0539-2},
}

@article {BB91,
    AUTHOR = {Borwein, J. M. and Borwein, P. B.},
     TITLE = {A cubic counterpart of {J}acobi's identity and the {AGM}},
   JOURNAL = {Trans. Amer. Math. Soc.},
  FJOURNAL = {Transactions of the American Mathematical Society},
    VOLUME = {323},
      YEAR = {1991},
    NUMBER = {2},
     PAGES = {691--701},
      ISSN = {0002-9947,1088-6850},
   MRCLASS = {33C75 (11F11 11Y60 33C05)},
  MRNUMBER = {1010408},
MRREVIEWER = {Bruce\ C.\ Berndt},
       DOI = {10.1090/S0002-9947-1991-1010408-0},
       URL = {https://doi.org/10.1090/S0002-9947-1991-1010408-0},
}

@article {HGB93,
    AUTHOR = {Hirschhorn, Michael and Garvan, Frank and Borwein, Jon},
     TITLE = {Cubic analogues of the {J}acobian theta function
              {$\theta(z,q)$}},
   JOURNAL = {Canad. J. Math.},
  FJOURNAL = {Canadian Journal of Mathematics. Journal Canadien de
              Math\'ematiques},
    VOLUME = {45},
      YEAR = {1993},
    NUMBER = {4},
     PAGES = {673--694},
      ISSN = {0008-414X,1496-4279},
   MRCLASS = {33D10 (11F27)},
  MRNUMBER = {1227653},
MRREVIEWER = {Bruce\ C.\ Berndt},
       DOI = {10.4153/CJM-1993-038-2},
       URL = {https://doi.org/10.4153/CJM-1993-038-2},
}

@article {Sch13,
    AUTHOR = {Schultz, Daniel},
     TITLE = {Cubic theta functions},
   JOURNAL = {Adv. Math.},
  FJOURNAL = {Advances in Mathematics},
    VOLUME = {248},
      YEAR = {2013},
     PAGES = {618--697},
      ISSN = {0001-8708,1090-2082},
   MRCLASS = {33E05},
  MRNUMBER = {3107523},
MRREVIEWER = {Thomas\ R.\ Shemanske},
       DOI = {10.1016/j.aim.2013.08.021},
       URL = {https://doi.org/10.1016/j.aim.2013.08.021},
}

@article {CCLZ26,
    AUTHOR = {Chan, Heng Huat and Chan, Song Heng and Liu, Zhi-Guo and
              Zudilin, Wadim},
     TITLE = {On {S}chultz's generalization of {B}orweins' cubic identity},
   JOURNAL = {Adv. Math.},
  FJOURNAL = {Advances in Mathematics},
    VOLUME = {495},
      YEAR = {2026},
     PAGES = {Paper No. 110967, 29},
      ISSN = {0001-8708,1090-2082},
   MRCLASS = {11F27},
  MRNUMBER = {5060518},
       DOI = {10.1016/j.aim.2026.110967},
       URL = {https://doi.org/10.1016/j.aim.2026.110967},
}

@article {Zhu23,
    AUTHOR = {Zhu, Xiao-Jie},
     TITLE = {Taylor expansions of {J}acobi forms and linear relations among
              theta series},
   JOURNAL = {Dissertationes Math.},
  FJOURNAL = {Dissertationes Mathematicae},
    VOLUME = {590},
      YEAR = {2023},
     PAGES = {66},
      ISSN = {0012-3862,1730-6310},
   MRCLASS = {11F50 (11F11 11F27 11F37)},
  MRNUMBER = {4684413},
       DOI = {10.4064/dm880-12-2023},
       URL = {https://doi.org/10.4064/dm880-12-2023},
}

@book {Fre11,
    AUTHOR = {Freitag, Eberhard},
     TITLE = {Complex analysis. 2},
    SERIES = {Universitext},
      NOTE = {Riemann surfaces, several complex variables, abelian
              functions, higher modular functions},
 PUBLISHER = {Springer, Heidelberg},
      YEAR = {2011},
     PAGES = {xiv+506},
      ISBN = {978-3-642-20553-8},
   MRCLASS = {30-01 (30Fxx 32-01 32A99 32N05)},
  MRNUMBER = {2810329},
MRREVIEWER = {Hiroshige\ Shiga},
       DOI = {10.1007/978-3-642-20554-5},
       URL = {https://doi.org/10.1007/978-3-642-20554-5},
}

@book {Lan23,
    AUTHOR = {Lange, Herbert},
     TITLE = {Abelian varieties over the complex numbers---a graduate
              course},
    SERIES = {Grundlehren Text Editions},
 PUBLISHER = {Springer, Cham},
      YEAR = {2023},
     PAGES = {xii+384},
      ISBN = {978-3-031-25569-4; 978-3-031-25570-0},
   MRCLASS = {14Kxx},
  MRNUMBER = {4573077},
MRREVIEWER = {Jean\ Kieffer},
       DOI = {10.1007/978-3-031-25570-0},
       URL = {https://doi.org/10.1007/978-3-031-25570-0},
}

@article {Sag26,
    AUTHOR = {SageMath},
     TITLE = {The Sage Mathematics Software System (Version 10.7)},
   JOURNAL = {The Sage Developers},
      YEAR = {2026},
       URL = {https://www.sagemath.org},
}

@article {BBG94,
    AUTHOR = {Borwein, J. M. and Borwein, P. B. and Garvan, F. G.},
     TITLE = {Some cubic modular identities of {R}amanujan},
   JOURNAL = {Trans. Amer. Math. Soc.},
  FJOURNAL = {Transactions of the American Mathematical Society},
    VOLUME = {343},
      YEAR = {1994},
    NUMBER = {1},
     PAGES = {35--47},
      ISSN = {0002-9947,1088-6850},
   MRCLASS = {11B65 (11F27 33D10)},
  MRNUMBER = {1243610},
MRREVIEWER = {George\ E.\ Andrews},
       DOI = {10.1090/S0002-9947-1994-1243610-6},
       URL = {https://doi.org/10.1090/S0002-9947-1994-1243610-6},
}

@article {CGG15,
    AUTHOR = {Cl\'{e}ry, Fabien and van der Geer, Gerard and
              Grushevsky, Samuel},
     TITLE = {Siegel modular forms of genus 2 and level 2},
   JOURNAL = {Int. J. Math.},
  FJOURNAL = {International Journal of Mathematics},
    VOLUME = {26},
      YEAR = {2015},
    NUMBER = {5},
     PAGES = {1550034, 51},
      ISSN = {0129-167X,1793-6519},
      NOTE = {With an appendix by Shigeru Mukai},
       DOI = {10.1142/S0129167X15500342},
       URL = {https://doi.org/10.1142/S0129167X15500342},
}

@book {Jac12,
    AUTHOR = {Jacobi, Carl Gustav Jacob},
     TITLE = {Fundamenta nova theoriae functionum ellipticarum},
    SERIES = {Cambridge Library Collection---Mathematics},
 PUBLISHER = {Cambridge University Press},
   ADDRESS = {Cambridge},
      YEAR = {2012},
     PAGES = {vi+191},
      ISBN = {978-1-108-05200-9},
       DOI = {10.1017/CBO9781139344081},
       URL = {https://doi.org/10.1017/CBO9781139344081},
      NOTE = {Reprint of the 1829 original},
}

@article {Igu79,
    AUTHOR = {Igusa, Jun-ichi},
     TITLE = {On the ring of modular forms of degree two over {${\bf Z}$}},
   JOURNAL = {Amer. J. Math.},
  FJOURNAL = {American Journal of Mathematics},
    VOLUME = {101},
      YEAR = {1979},
    NUMBER = {1},
     PAGES = {149--183},
      ISSN = {0002-9327,1080-6377},
   MRCLASS = {10D20 (14K10 14K15)},
  MRNUMBER = {527830},
MRREVIEWER = {A.\ N.\ Andrianov},
       DOI = {10.2307/2373943},
       URL = {https://doi.org/10.2307/2373943},
}

@book {Fre91,
    AUTHOR = {Freitag, Eberhard},
     TITLE = {Singular modular forms and theta relations},
    SERIES = {Lecture Notes in Mathematics},
    VOLUME = {1487},
 PUBLISHER = {Springer-Verlag, Berlin},
      YEAR = {1991},
     PAGES = {iv+172},
      ISBN = {3-540-54704-5},
   MRCLASS = {11F46 (11F27 32N10)},
  MRNUMBER = {1165941},
MRREVIEWER = {Rolf\ Berndt},
       DOI = {10.1007/BFb0093829},
       URL = {https://doi.org/10.1007/BFb0093829},
}

@article {GW25,
    AUTHOR = {Gritsenko, Valery and Wang, Haowu},
     TITLE = {Orbits of {J}acobi forms and theta relations},
   JOURNAL = {Sci. China Math.},
  FJOURNAL = {Science China. Mathematics},
    VOLUME = {68},
      YEAR = {2025},
    NUMBER = {5},
     PAGES = {1015--1030},
      ISSN = {1674-7283,1869-1862},
   MRCLASS = {11F50 (14K25)},
  MRNUMBER = {4894295},
MRREVIEWER = {Haigang\ Zhou},
       DOI = {10.1007/s11425-023-2324-4},
       URL = {https://doi.org/10.1007/s11425-023-2324-4},
}

\end{document}